\documentclass[12pt,reqno]{amsart} 
\usepackage{amsmath, mathtools}
\usepackage{mathrsfs} 
\usepackage{amsthm}
\usepackage{amssymb}
\usepackage{amscd,stmaryrd}

\usepackage{rotating} %trying to have sideways tables
\usepackage{pdflscape}
\usepackage{floatrow}
\usepackage{dirtytalk} 
\usepackage{parskip}
\usepackage{bbm}
\usepackage[english]{babel}
\usepackage[a4paper,margin = 1in]{geometry}
\usepackage{eurosym}
\usepackage{amsthm}
\usepackage{enumitem}
\usepackage{multicol}
\usepackage{array}
\usepackage{url}
\usepackage{hyperref} 
\hypersetup{
	colorlinks=true,
	linkcolor=red,
	filecolor=magenta,      
	urlcolor=cyan,
	citecolor=blue
}
\usepackage{color}
\usepackage[dvipsnames]{xcolor}
\usepackage{textcase}
\usepackage{appendix}
\usepackage{tikz-cd}
\usepackage{tikz-3dplot}
\usepackage{subcaption}
\usepackage{makecell} % in preamble
\usepackage{booktabs} % in preamble
\usepackage{tikz}
\usetikzlibrary{arrows.meta}

\newcommand{\eHK}{e_\text{HK}}
\newcommand{\eHKp}{e_{\text{HK},p}}

\newcommand{\limeHK}{e_{\text{HK},\infty}}
\newcommand{\lims}{s_{\infty}}
\newcommand{\q}{\mathfrak{q}}
\newcommand{\p}{\mathfrak{p}}
\newcommand{\n}{\mathfrak{n}}
\newcommand{\m}{\mathfrak{m}}

\newcommand{\Spec}{\operatorname{Spec}}
\newcommand{\MaxSpec}{\operatorname{MaxSpec}}

\newtheorem{theorem}{Theorem}[section]

\newtheorem{definition}[theorem]{Definition}
\newtheorem{notation}[theorem]{Notation}
\newtheorem{theoremdefinition}[theorem]{Theorem-Definition}
\newtheorem{corollary}[theorem]{Corollary}
\newtheorem{lemma}[theorem]{Lemma}
\newtheorem{proposition}[theorem]{Proposition}

\newtheoremstyle{example}{10pt}{10pt}{}{}{\bfseries}{.}{.5em}{}
\theoremstyle{example}
\newtheorem{example}[theorem]{Example}

\theoremstyle{remark}
\newtheorem{remark}[theorem]{Remark}

\title{Limit F-invariants of the simple singularities and the Fermat hypersurfaces}
\author{Joel Castillo-Rey}
\address{BCAM – Basque Center for Applied Mathematics, Mazarredo 14, 48009 Bilbao, Basque Country – Spain}
\email{jcastillo@bcamath.org}
\date{}

\begin{document}
\begin{abstract}
    In this article, we compute the limit Hilbert–Kunz multiplicities and limit F-signatures of all the simple singularities and the Hilbert–Kunz multiplicities of the Fermat hypersurfaces, generalizing the well-known $\sec + \tan$ formula by Gessel and Monsky for the $A_1$ singularities, using the theory of $h$-functions. We obtain a functional equation relating the $\phi$-function of a hypersurface with its reflection. Using limit $F$-invariants, we generalize both a characterisation of regular local rings through $F$-invariants and the known cases of the Watanabe–Yoshida conjecture to characteristic zero.
\end{abstract}
\maketitle
\setcounter{tocdepth}{1}
\tableofcontents
\section{Introduction}
The Hilbert–Kunz multiplicity $\eHK(R)$ and the $F$-signature $s(R)$ are fundamental numerical invariants of a local ring $(R,\m)$ of positive characteristic capable of capturing regularity, as well as subtle $F$-singularity data, \cite{monskyHilbertKunzFunction1983,kunz1969,hunekeTwoTheoremsMaximal2004,tuckerFsignatureExists2012,WY00,blickleRingsSmallHK2004}. One may as well define characteristic zero versions of these invariants by passing to reductions mod $p$ and taking the limit $p\rightarrow \infty$, provided the limit exists, although their existence, even for hypersurfaces, is still an open question in the field.

A very well-known result on limit $F$-invariants, and one of the starting point of this work, is the following result by Gessel and Monsky on the limit Hilbert–Kunz multiplicity and the limit $F$-signature of the $A_1$ singularity:
\begin{theorem}[Limit Hilbert–Kunz multiplicities of the $A_1$ singularities, \cite{gesselLimit2010}]
	\label{th:gessel_monsky}
Let $d\geq 0$, and $R_{d} = \mathbb{C}[[x_0,\dots,x_d]]/(x_0^2+\dots+x_d^2)$. Then,
\begin{equation*}
\limeHK(R_{d}) = 1+\dfrac{c_d}{d!}, \quad \text{and} \quad \lims(R_{d}) = 1-\dfrac{c_d}{d!}
\end{equation*}
where $c_d$ are the Euler zig-zag numbers. In other words, $\sum c_d\frac{z^d}{d!} = \sec z + \tan z$.
\end{theorem}

\subsection{The simple singularities}
The simple singularities are characterised as the hypersurface singularities that may only be deformed into finitely many other non-isomorphic singularities (see \cite{greuelSimpleSingularitiesPositive1990}). They were completely classified in characteristic zero in every dimension by Arnol'd \cite{arnoldNormalFormsFunctions1972} in 1972 in a list, which the reader can find in \ref{def:simple_singularities}. 

More is known about the $F$-invariant theory of the diagonal simple singularities due to the theory developed by Han and Monsky on the diagonal hypersurfaces, \cite{hanmonskySurprising1993,gesselLimit2010}. The two-dimensional case is the most studied, since all the Hilbert–Kunz functions are known, \cite{WY00,brinkmann2017}. On the other hand, the Hilbert–Kunz theory of the non-diagonal simple singularities in higher dimensions has been less explored; the rationality of their $F$-invariants is known by \cite{seibertFCMtype1997}. In any case, even for the diagonal simple singularities there are still plenty of open questions, as showcased by the ever-growing literature on the Hilbert–Kunz theory of the $A_1$ singularities, see \cite{Yo09,gesselLimit2010,castilloreyStrong2025,pakQuadricsEhrhart2026,mengQuadrics2026}.

In this work, we will make use of mainly two properties of these singularities: the fact that they are Gorenstein double points (which yields Corollary \ref{cor:reflection_h_function_double_points}), and that their equations are either diagonal or a sum of a binomial plus monomials of the form $x^2$ (which allows us to exploit Meng's integral formula \cite[Theorem 5.23]{mengLimits2026}).

\subsection{Main results}
In the spirit of Gessel and Monsky's result, one of the main results of this work is computing the limit $F$-invariants of all the simple singularities. We do this in Sections \ref{sec:limit_h_functions_A_n} and \ref{sec:f_invariants_simple} through the theory of $h$-functions \cite{mengHfunction2025}, using the powerful analysis techniques recently introduced in \cite{mengLimits2026}. 
\begin{theorem}[Theorems \ref{cor:limit_eHK_A_n_singularities}, \ref{th:torics_limit_eHK} and \ref{th: D_n_singularities_eHK_formula}]
Let $d\geq 2$. Let $g(x,y)\in \mathbb{Z}[x,y]$ like in table \ref{tab:gen_functions_intro}, and $R_d = \mathbb{C}[[x_0,\dots,x_d]]/(g(x_0,x_1)+x_2^2+\dots+x_d^2)$. Then,
\begin{equation*}
\limeHK(R_{d}) = 1+\dfrac{c_d}{d!}, \quad \text{and} \quad \lims(R_{d}) = 1-\dfrac{c_d}{d!}
\end{equation*}
where $c_d/d!$ is the $d$-th Taylor coefficient of the function $F_g(z)$, for $F_g(z)$ as in Table \ref{tab:gen_functions_intro}.
\begin{table}[h]
\centering
\renewcommand{\arraystretch}{2}
\begin{tabular}{c|c|c}
\toprule
 & $g(x,y)$ & $F_g(z)$\\
\midrule
$A_n$
& $x^{n+1}+y^2$
& $\displaystyle \frac{n+1}{2}\frac{\cos\frac{(n-1)z}{n+1}+\sin\frac{2z}{n+1}}{\cos z}$\\
\midrule
$D_n$
& $x^2y+y^{n-1}$
& \makecell[c]{$\frac{(n-1)^2}{4(n-2)}\cos\frac{z}{n-1}+\frac{(n-3)^2}{4(n-2)}\cos\frac{z}{n-3}$\\[4pt]
    $(1+\frac{(n-1)^2}{4(n-2)}\sin\frac{z}{n-1}+\frac{(n-3)^2}{4(n-2)}\sin\frac{z}{n-3})T(z)$}\\
\midrule
$E_6$
& $x^3+y^4$
& $3\dfrac{\cos\frac{z}{2} + \sin\frac{z}{2}}{\cos z}\cos\dfrac{z}{3}$\\
\midrule
$E_7$
& $x^3+xy^4$
& \makecell[c]{$\frac{27}{16}\left(\cos\frac{z}{9}-\sin\frac{2z}{9}\right)+$\\[4pt]
    $+\frac{27}{16}\left(\sin\frac{z}{9}+\cos\frac{2z}{9}\right)T(z)$}\\
\midrule
$E_8$
& $x^3+y^5$
& $\dfrac{15}{4}\dfrac{\cos\frac{3z}{5} + \sin\frac{2z}{5}}{\cos z}\cos\dfrac{z}{3}$\\
\bottomrule
\end{tabular}
\caption{Summary of Theorems \ref{cor:limit_eHK_A_n_singularities}, \ref{th:torics_limit_eHK} and \ref{th: D_n_singularities_eHK_formula}. Here, $T(z) = \sec z +\tan z$.}
\label{tab:gen_functions_intro}
\end{table}
\end{theorem}
In the case of the $A_n$ singularities, these formulas let us generalize the connection with combinatorics already observed in the limit $F$-invariants of the $A_1$ singularity:
\begin{corollary}[Corollary \ref{cor:combinatorial_interpretation}]
The limit Hilbert–Kunz multiplicity of $x_0^{n+1}+x_1^2+\dots+x_d^2$ is
    \begin{gather*}
    1 + \frac{n+1}{2}\dfrac{N_d(2,n+1)}{(n+1)^d d!},
    \end{gather*}
    where $N_d(2,n+1)$ is the number of alternating augmented $(2,n+1)$-signed $d$-permutations.
\end{corollary}
Furthermore, we show that the limit $h$-functions of all the $A_n$ singularities are combinations of Euler polynomials. In the case of hypersurfaces, our case of study, the $h$-function captures both the Hilbert–Kunz multiplicity and the $F$-signature, coincides with the $\phi$-function (\cite{shidelerFsignatureFunctionsDiagonal2024}, see Definition \ref{def:phi_function}), and recovers the $F$-signature of pairs (\cite{blickleFsignaturePairsAsymptotic2012,blickleFsignaturePairsContinuity2013}, see Definition \ref{def:F_signature_of_pairs}). Therefore, this result generalizes \cite[Theorem 4.4]{shidelerFsignatureFunctionsDiagonal2024}.
\begin{theorem}[Theorem \ref{th:limit_h_function_A_n}]
    Let $n \geq 2$. Let $f = x_0^n+x_1^2+\dots+x_d^2 \in \mathbb{Z}[[x_0,\dots,x_d]]$, and set
\begin{equation*}
a_{d,n} = \begin{cases}
        \frac{1}{2}-\frac{1}{n} & d \text{ even,}\\
        \frac{1}{n} & d \text{ odd.}
        \end{cases} \qquad c_{d,n} = (-1)^{\lfloor d/2 \rfloor} \frac{2^{d-3}\, n}{d!}
\end{equation*}
Then, for $d\geq 1$,
\begin{equation*}
h_{f,\infty}(x) = \begin{cases}
x + c_{d+1,n}\left(E_{d+1}(x+a_{d+1,n})-E_{d+1}(-x+a_{d+1,n})\right) & x\in [0,a_{d+1,n}],\\[7pt]
x + c_{d+1,n}\left(E_{d+1}(x+a_{d+1,n})+E_{d+1}(-x+1+a_{d+1,n})\right) & x\in [a_{d+1,n},1-a_{d+1,n}],\\[7pt]
x + c_{d+1,n}\left(-E_{d+1}(x-1+a_{d+1,n})+E_{d+1}(-x+1+a_{d+1,n})\right) & x\in [1-a_{d+1,n},1],
\end{cases}
\end{equation*}
where $E_d$ is the $d$-th Euler polynomial (defined in \ref{def:Euler_Polynomials}).
\end{theorem}
This result was inspired by Hoffman's formulas in \cite[Theorem 6.2]{hoffmanDerivativePolynomialsEuler1999}, which express certain sequences as evaluations of Euler polynomials at certain rational values.

In Theorems \ref{th:h_function_D_n_surface} and \ref{th:h_function_E_7_surface} we start with the non-diagonal simple singularities, and compute the limit $h$-functions of the $D_n$ surfaces and the $E_7$ surfaces. The method was already sketched in a previous version of \cite{mengLimits2026}, and it can be systematically applied to obtain higher-dimensional cases.

In Section \ref{sec:wat_yos_char_0}, we explore the characteristic zero version of the Watanabe–Yoshida conjecture, generalize the known cases, and use the computation of the limit for the $A_1$ and $A_2$ singularities from previous sections to extend the settled cases of the strong form of the conjecture to characteristic zero, see \cite{castilloreyStrong2025}.
\begin{theorem}[Theorem \ref{th:watanabe_yoshida_low_dimension}]
    \label{th:intro_watanabe_yoshida_low_dimension}
    Let $K$ be a field of characteristic $0$, and 
    \begin{gather*}
        R_{0,d}:=\left(K[x_0,...,x_d]/(x_0^2+\dots+x_d^2)\right)_{(x_0,\dots,x_d)}.
    \end{gather*}
    Let $(R,\m,K)$ be a $d$-dimensional $K$-algebra of essentially finite type. Moreover, assume that $R$ is unmixed and nonregular. If either $2\leq d\leq 7$, or $R$ is a complete intersection of dimension $\geq 2$, then either $\liminf_p\eHKp(R)> \limeHK(R_{0,d})$, or $R$ is an $A_1$ singularity (in the sense of \ref{def:formally_A_1_singularity}).
    \end{theorem}

On a different note, we observed a functional equation relating the $h$-function or the $\phi$-function of any hypersurface with its reflection that we will include in the section on preliminaries. 
\begin{theorem}[Theorem \ref{th:reflection} and Corollary \ref{cor:reflection_h_function_double_points}]
Let $(R,\m)$ be a regular local domain of dimension $d$, and $0\neq f\in R$. Let the ideal $I\subset R/f$ be a minimal reduction of $\m/f$ that is a parameter ideal in $R/f$. Then,
\begin{gather*}
h_{R,f,I:\m}(1-t) = h_{R,f,\m}(t)-e(R/f)(t-1)-1.
\end{gather*}
Moreover, if $f$ is a double point, then
\begin{gather*}
h_{R,f,\m}(1-t) = h_{R,f,\m}(t)-2t+1,
\end{gather*}
or in terms of the $F$-signature of pairs and the $\phi$-function,
\begin{gather*}
s(R,f^{1-t}) = s(R,f^t)+2t-1 \quad \text{and} \quad \phi_{\m,f}(1-t) = \phi_{\m,f}(t)-2t+1.
\end{gather*}
\end{theorem}

Finally, Cheng Meng posed a conjecture about the generating function of the $\limeHK$ of the Fermat hypersurfaces that we confirm in Section \ref{sec:fermat}, but in this case we do it so by simply building on Gessel and Monsky's techniques from \cite{gesselLimit2010}:
\begin{corollary}[Theorem \ref{th:fermat_limit_eHK_gen}, Conjecture 8.8 in \cite{mengLimits2026}]
For $d\geq 2$, the generating function of the limit Hilbert–Kunz multiplicities of the Fermat hypersurfaces
\begin{equation*}
    \sum_{n=0}^{\infty}\limeHK\left(R_{0,n}^{(d,\dots,d)}\right)z^{n}
\end{equation*}
is a rational function on the functions $z, \cos(a_iz)$ and $\sin(a_iz)$, for some finitely many $a_i\in \mathbb{Q}(\zeta_{2d})\cap \mathbb{R}$ with $\zeta_{2d} = e^{\frac{\pi}{d}i}$. 
\end{corollary}
\textbf{Acknowledgments.} The author wishes to thank Ilya Smirnov, Kevin Tucker and Devlin Mallory, for discussions and questions that motivated and initiated this project. The author also wishes to thank the Online Enciclopedia of Integer Sequences.

The author was supported by the grants SEV-2023-2026 from the Spanish Ministry of Science and Innovation, RYC2020-028976-I funded by MICIU/AEI/10.13039/501100011033, and EI ESF "ESF Investing in your future".

\textbf{AI disclosure.} Anthropic's Claude (Opus and Sonnet) has been used for proof-checking the manuscript, making several of the figures, and simplifying formulas across Sections \ref{sec:limit_h_functions_A_n} and \ref{sec:f_invariants_simple}.

\section{Preliminaries}
\label{sec:Preliminaries}
Throughout, we will assume all local rings to be Noetherian, $p>0$ will always denote a prime, and $q$ a power of this prime. Also, we will denote the (usual) Hilbert–Samuel multiplicity of $R$ at the maximal ideal by $e(R)$, the Hilbert–Kunz multiplicity by $\eHK(R)$ and the F-signature by $s(R)$. In this section, we briefly survey some of the tools we will use throughout the work, and refer the reader to \cite{mengHfunction2025,mengLimits2026} for details.

\subsection{The \texorpdfstring{$h$}{h}-function and the \texorpdfstring{$F$}{F}-signature of pairs.} Here we recall the definition of the $h$-function and the $F$-signature of pairs. 
\begin{definition}[$h$-function, Definition 3.4 in \cite{mengHfunction2025}]
	\label{def:h_function}
Let $(R,\m,k)$ be a $d$-dimensional local ring of characteristic $p>0$. Let $I,J$ be ideals such that $I+J$ is $\m$-primary. Set
\begin{gather*}
    h_{R,I,J,e}(x) := \frac{1}{p^{ed}}\ell(R/(I^{\lceil xp^e\rceil}+J^{[p^e]})),
\end{gather*}
where, by convention, a non-positive power of an ideal is taken to be the unit ideal. The $h$-function of the triple $(R,I,J)$ is the limit of these functions:
\begin{gather*}
h_{R,I,J}(x) := \lim_{e\rightarrow\infty} h_{R,I,J,e}(x).
\end{gather*}
which exists by \cite[Theorem 3.29]{mengHfunction2025}. Also, when $R$ is clear by context, we write $h_{I,J} = h_{R,I,J}$, and when $R = k[[x_1,...,x_n]]$, $I = (f)$ and $J = \m$, we will often write $h_f := h_{R,f,\m}$, and call it the $h$-function of the hypersurface $f$.
\end{definition}
\begin{definition}[$F$-signature of pairs, \cite{blickleFsignaturePairsAsymptotic2012,blickleFsignaturePairsContinuity2013}]
	\label{def:F_signature_of_pairs}
	Let $(R,\m)$ be an $F$-finite regular local ring of dimension $d$, $\mathfrak{a}$ an ideal, and $t\in \mathbb{R}_{\geq 0}$. Then,
\begin{equation*}
s(R,\mathfrak{a}^t) := \lim_{e\rightarrow \infty} \dfrac{1}{p^{de}}\ell\left(\dfrac{R}{\left(\m^{[p^e]}:\mathfrak{a}^{\lceil tp^e\rceil}\right)}\right)
\end{equation*}
\end{definition}
The $h$-function of a hypersurface captures the $F$-signature of pairs:
\begin{lemma}[Relationship between $F$-signature of pairs and $h$-function]
    \label{lm:F_signature_pairs_and_h_function}
	Let $(R,\m)$ be an $F$-finite regular local ring of dimension $d$ and $f\in \m$ a nonzerodivisor. Then,
	\begin{equation*}
		h_{R,f,\m}(t) = 1-s(R,f^t).
	\end{equation*}
\end{lemma}
\begin{proof}
	It follows by definition from this short exact sequence,
	\begin{equation*}
	0 \longrightarrow \dfrac{R}{(\m^{[p^e]}:f^{\lceil tp^e\rceil})}\stackrel{f^{\lceil tp^e\rceil}}\longrightarrow \dfrac{R}{\mathfrak{m}^{[p^e]}} \longrightarrow \dfrac{R}{(\m^{[p^e]},f^{\lceil tp^e\rceil})}\longrightarrow 0.
	\end{equation*}
\end{proof}
Let us briefly recall some properties of the $h$-function of a hypersurface:
\begin{proposition}[Some analytic properties of the $h$-function of a hypersurface, \cite{mengHfunction2025} and \cite{mengLimits2026}]
\label{pr:h_function_analytic_properties}
Let $(R,\m)$ be a local $d$-dimensional domain of characteristic $p$, $0\neq f\in \m$, and $J$ an $\m$-primary ideal. Then, the $h$-function $h_{R,f,J}(x)$ has the following properties.
\begin{enumerate}
\item it is an increasing function,
\item it is continuous,
\item it is constantly 0 on $(-\infty,0]$, and for $s\gg 0$, the function is eventually $\ell(R/J)$. In particular, if $J = \m$, then $h_{R,f,\m}(s) = 1$ for $s\geq 1$.
\item Outside a countable subset of $(0,1)$, the derivative exists and is continuous. Furthermore, left and right derivatives exist throughout $(0,1)$ (see \cite[Theorem 5.5]{mengHfunction2025}).
\item it is concave in $[0,\infty)$ (see \cite[Theorem 3.4 (5)]{mengLimits2026}).
\end{enumerate}
\end{proposition}

The $h$-function of a hypersurface captures the Hilbert–Kunz multiplicity and the $F$-signature, which can be deduced from its relation to the $F$-signature of pairs:
\begin{theorem}[Theorem 4.4, \cite{blickleFsignaturePairsContinuity2013}]
    \label{th:F_invariants_lim_diff_h}
Let $(R,\m)$ be an F-finite regular local domain of dimension $d$ and $0\neq f\in \m$. Then,
\begin{gather*}
\lim_{x\rightarrow 0^+}\partial_x^+ h_{f}(x) = \eHK(R/f) \qquad \lim_{x\rightarrow 1^-} \partial_x^- h_{f}(x) = s(R/f).
\end{gather*}
\end{theorem}

There is a multivariate version of the $h$-function that we will need to use Theorem \ref{th: limit_h_binomials}.
\begin{definition}[The multivariate $h$-function, Definition 3.1 in \cite{mengLimits2026}]
\label{def:multivariate_h_function}
Let $(R,\m)$ be a local domain, $I\subset R$ an ideal and $\underline{f} = f_1,\dots,f_s$ a sequence in $R$ such that $I+(\underline{f})$ is an $\m$-primary ideal. Then, for $\mathbf{t} = (t_1,\dots,t_s)\in \mathbb{R}^s$, define the multivariate $h$-function as
\begin{gather*}
h_{R,I,\underline{f}}(t) := \lim_{q\rightarrow \infty} \dfrac{1}{q^d}\ell\left(\dfrac{R}{(I^{[q]},f_1^{\lceil qt_1\rceil},\dots,f_s^{\lceil qt_s \rceil})}\right),
\end{gather*}
which converges by \cite[Proposition 3.5]{mengLimits2026}. 
\end{definition}

\subsection{Reduction mod $p$ setup} 
Throughout most of this work, we will not deal with the positive characteristic versions of these invariants, but with their limit counterparts, which are the result of computing them at \say{reductions mod $p$} and then taking the limit when $p\rightarrow \infty$, provided it exists.
\begin{definition}[Descent data, \cite{hochsterTightClosureEqual1999}]
    \label{def:reduction_mod_p}
Given a characteristic zero field $K$, a $K$-algebra
\begin{equation*}
R := \left(\dfrac{K[x_1,\dots,x_n]}{(f_1,\dots,f_r)}\right)_{(x_1,\dots,x_n)},
\end{equation*}
and an $R$-module $M$, we call \textit{descent data} for $(K,R,M,f)$ a tuple $(A,R_A,M_A,f_A)$ such that
\begin{enumerate}
\item $A$ is a finitely generated $\mathbb{Z}$-subalgebra of $K$,
\item $R_A$ is a finitely generated $A$-subalgebra of $R$ such that $R_A$ is a free $A$-module and the inclusion $R_A\hookrightarrow R$ induces an isomorphism $R_A\otimes_A K\cong R$,
\item $M_A$ is a finitely generated $A$-submodule of $M$ such that $M_A$ is a free $A$-module and the inclusion $M_A\hookrightarrow M$ induces an isomorphism $M_A\otimes_A K\cong M$.
\item $f\in R_A\subseteq R$ and $f_A = f$.
\end{enumerate}
We will often say that $A\rightarrow R_A$ is a family of characteristic $p$ models to mean that $(A,R_A)$ is descend data for $(K,R)$.
\end{definition}
\begin{proposition}[Proposition (2.1.9), \cite{hochsterTightClosureEqual1999}]
For every quadruple $(K,R,M,f)$ as in \ref{def:reduction_mod_p}, there exists descent data $(A,R_A,M_A,f_A)$.
\end{proposition}

\begin{remark}
By generic smoothness, passing to a localisation of $A$ one can further assume that $A$ is a smooth $\mathbb{Z}$-algebra, which will be useful later to show some results about independence of the descent data.
\end{remark}
\begin{definition}[Characteristic $p$ model, \cite{liuPositivity2026}]
    \label{def:char_p_model}
Let $(K,R,M,f)$ be as in \ref{def:reduction_mod_p}. Let $p>0$ be a prime number, and $\m\in \MaxSpec(A)$ such that $\m\cap \mathbb{Z} = (p)$, and $k := k(\m) := A/\m$. We set $R_k := R_A\otimes_A k$, $M_k := M_A\otimes_A k$, $f_k = f\otimes_A k$, and call $(k,R_k,M_k,f_k)$ a characteristic $p$ model of the quadruple $(K,R,M,f)$. For this reason, we will often refer to $(A,R_A,M_A,f_A)$ and/or the map $A\rightarrow R_A$ as a family of characteristic $p$ models.
\end{definition}

We now define the limit $F$-invariants over every field of characteristic $0$, following \cite[Definition 3.3]{liuPositivity2026}:
\begin{definition}[Limit $F$-invariants]
    \label{def:limit_F_invariants}
Let $(A,R_A)$ be descent data for the pair $(K,R)$. We define the \textit{Hilbert–Kunz multiplicity and $F$-signature in characteristic $p$} of the family $A\rightarrow R_A$ as
\begin{gather*}
    e_{\text{HK},p}(R) := \inf_{\m\in\MaxSpec{A/pA}} \eHK(R_{k(\m)}) \qquad s_p(R) := \sup_{\m\in\MaxSpec{A/pA}} s(R_{k(\m)}).
\end{gather*}
We then define the \textit{liminf/limsup $\eHK$ and the liminf/limsup $F$-signature} as
\begin{align*}
\eHK^+(R) &:= \limsup_{p\rightarrow \infty} e_{\text{HK},p} (R) & \eHK^-(R) &:= \liminf_{p\rightarrow \infty} e_{\text{HK},p} (R)\\
s^+(R) &:= \limsup_{p\rightarrow \infty} s_p(R) & s^-(R) &:= \liminf_{p\rightarrow \infty} s_p(R).
\end{align*}
When $\eHK^+(R) = \eHK^-(R)$, or $s^+(R) = s^-(R)$, then we denote these limits by $\limeHK(R)$, and $\lims(R)$, respectively.
\end{definition}
\begin{remark}
Note that by Hilbert's Nullstellensatz, we know that $k(\m)/\mathbb{F}_p$ is a finite extension for every $\m\in \MaxSpec{A/pA}$; in particular, $R_{k(\m)}$ is of essentially finite type over a finite field, and therefore $F$-finite.
\end{remark}
\begin{example}
Let $R = \mathbb{Q}(t)[x]/(x^2)$, $A := \mathbb{Z}[t]$ and $R_A := A[x]/(tx^2)$. Note that $A\rightarrow R_A$ fails the freeness hypotheses required in the definitions above. At every prime $p\in \mathbb{Z}$, the fiber $R_A/(p,t) \cong \mathbb{F}_p[x]$ is regular, which might draw one to the incorrect conclusion that $e_{\text{HK},p}(R)$ and $s_{p}(R)$ are both one for all $p$. However, none of the closed points $(p,t)$ are in the flat locus of $A\rightarrow R_A$, and actually $\eHK(R_k) = 2$ and $s(R_k) = 0$ at every closed fiber of $A_t$, so $\limeHK(R) = 2$ and $\lims(R) = 0$.
\end{example}
The next lemma shows that the limit $F$-invariants do not depend on the model we choose, where we apply the same proof strategy than in \cite{liuPositivity2026}:
\begin{lemma}[Lemma 3.4, \cite{liuPositivity2026}]
Let $(A,R_A)$ and $(B,R_B)$ be descent data for the pair $(K,R)$. Let $A\rightarrow R_A$ and $B\rightarrow R_B$ two families of characteristic $p$ models of $(K,R)$ over $A$ and $B$ finitely generated smooth $\mathbb{Z}$-algebras. Then, for any prime $p\gg 0$,
\begin{enumerate}
\item we have,
\begin{equation*}
    \inf_{\m\subset A/pA \text{ maximal}} \eHK(R_{A/\m}) = \inf_{\n\subset B/pB \text{ maximal}} \eHK(R_{B/\n}).
\end{equation*}
\item Also,
\begin{equation*}
\sup_{\m\subset A/pA \text{ maximal}} s(R_{A/\m}) = \sup_{\n\subset B/pB \text{ maximal}} s(R_{B/\n}).
\end{equation*}
\end{enumerate}
In other words, for $p\gg 0$, $e_{\text{HK},p}(R)$ and $s_p(R)$ do not depend on the descent data.
\end{lemma}
\begin{proof}
Note that $A\rightarrow R_A$ is an $(x_1,\dots,x_n)$-family \cite[Definition 3.8]{smirnov2020semicontinuity} by definition.

Let $pA \subset \p$ be a minimal prime of $pA$, with fraction field $L = \operatorname{Frac}{A/\p}$. Firstly, we prove the fact that
\begin{equation}
    \label{eq:inf_attained_at_generic_fiber}
    \inf_{\m \subset A/\p \text{ maximal}} \eHK(R_{k(\m)}) = \eHK(R_L).
\end{equation}
By \cite[Corollary 4.2]{smirnov2020semicontinuity}, $\eHK(R_{k(\p')})\leq \eHK(R_{k(\p)})$ for every $\p'\subset \p$, so one inequality follows:
\begin{equation*}
\inf_{\m \subset A_p\text{ maximal}} \eHK(R_{k(\m)}) \geq \inf_{\p \in \Spec{A_p}} \eHK(R_{k(\p)}) = \eHK(R_L).
\end{equation*}
For the opposite inequality, Corollary 4.10 in \textit{op. cit.} implies that for every $\varepsilon >0$, the set $U_{\varepsilon}\subset \Spec{A/\p}$ of points such that $\eHK(R_{k(\p)})< \eHK(R_L)+\varepsilon$ is open and dense, so taking $\varepsilon\rightarrow 0$ yields a non-empty set of closed fibers $R_k$ such that $\eHK(R_k) = \eHK(R_L)$, which implies the identity \ref{eq:inf_attained_at_generic_fiber}.

Now, finding a common enlargement of $A$ and $B$, we may assume that $A$ is a subring of $B$. Using again \cite[Corollary 4.10]{smirnov2020semicontinuity}, one notes that $e_{\text{HK},p}(R)$ does not change if we invert an element of $A$ or $B$, since for every $\lambda>e_{\text{HK},p}(R)$, the set $V\subset \MaxSpec A/p$ such that $\eHK(R_k)<\lambda$ is dense in $\MaxSpec{A/pA}$. Therefore, by \cite[\href{https://stacks.math.columbia.edu/tag/00TF}{Tag 00TF}]{stacks-project} and generic smoothness applied to $\operatorname{Frac}(A)$, we may assume that the induced map $\varphi: \Spec{B}\rightarrow \Spec{A}$ is a dominant and smooth map of smooth schemes over $\mathbb{Z}$. In particular, for every $p\gg 0$, the map $\Spec{B/pB}\rightarrow \Spec{A/pA}$ is also dominant and smooth. Fix one such prime $p$.

Let $\p\supset pA$ and $\q\supset pB$ minimal primes of $pA$ and $pB$ respectively, and let $A_p := A/\p$ and $B_p = B/\q$ irreducible components of $A/pA$ and $B/pB$, respectively. The induced map $\varphi:\Spec{B_p}\rightarrow \Spec{A_p}$ is also dominant. Note that $A_{p}$ and $B_{p}$ are regular domains of finite type over $\mathbb{F}_p$, and, by smoothness, we can assume that $L_1 := \operatorname{Frac}(A_p) \subset L_2:=\operatorname{Frac}(B_p)$ is a separable extension. That implies that $\eHK(R_{L_1}) = \eHK(R_{L_2})$ by \cite[Theorem 3.7]{carvajal2021bertini}, and therefore the result follows from the second paragraph of the proof.

As for the limit $F$-signatures $s^+(R)$ and $s^-(R)$, the same arguments apply, using the lower semi-continuity from \cite[Proposition 5.7]{polstra2018uniform} instead, and the invariance under faithfully flat extension with regular fibers from \cite[Theorem 3.6]{carvajal2021bertini}.
\end{proof}
\begin{remark}
Throughout most parts of this work, except Section \ref{sec:wat_yos_char_0}, the algebras we will consider will be hypersurfaces with integer coefficients (see \ref{def:simple_singularities}, which lists the singularities we will consider throughout most of the text). In this case, the setup is simpler, since $R = (\mathbb{C}[x_1,\dots,x_n]/(f))_{(x_1,\dots,x_n)}$ with $f\in \mathbb{Z}[x_1,\dots,x_n]$, so we may set $A=\mathbb{Z}$, and $\eHKp(R) = \eHK(R_p)$, where $R_p := (\mathbb{F}_p[x_1,\dots,x_n]/(f))_{(x_1,\dots,x_n)}$.
\end{remark}

Finally, we define the limit $h$-function for a hypersurface with integer coefficients.
\begin{definition}[The limit $h$-function, Setting 5.3 in \cite{mengLimits2026}]
    The limit $h$-function of $f \in \mathbb{Z}[[x_1,\dots,x_n]]$ is
    \begin{gather*}
        h_{f,\infty}(x) = \lim_{p\rightarrow \infty} h_{f_p}(x),
    \end{gather*}
    provided the limit exists for every $x\in \mathbb{R}$, where $f_p$ are the reductions mod $p$ (Definition \ref{def:char_p_model}).
\end{definition}

\subsection{The integral equation and the \texorpdfstring{$D_\infty$}{Dinf} function}
One of the milestones of Meng's theory is the integral formula \ref{th:integral_formula} for the $h$-function of a hypersurface of the form $f(x_1,...,x_n)+g(y_1,...,y_m)$, which, roughly, expresses $h_{f+g}$ as the (Riemann–Stieltjes) integral of certain three-variable function $D$ along the integrators $-h_f'$ and $-h_g'$. We refer the reader to \cite[Section 2]{mengLimits2026} for details about Riemann–Stieltjes integration theory.

We start with the main auxiliar function used in most of the computations:
\begin{definition}[The function $D_\infty$, Definition 4.8 \cite{mengLimits2026}]
\label{def:D_infty}
For every prime $p>0$, let $k_p$ be a field of characteristic $p>0$. We define
\begin{align*}
    D_p:\mathbb{R}^3&\longrightarrow \mathbb{R}\\
(t_1,t_2,t_3) &\longmapsto \lim_{e\rightarrow \infty}\dfrac{1}{q^2}\dim_{k_p}\left(\dfrac{k_p[T_1,T_2]}{(T_1^{\lceil t_1q\rceil},T_2^{\lceil t_2q\rceil},(T_1+T_2)^{\lceil t_3q\rceil})}\right).
\end{align*}
Note that it only depends on the characteristic of $k_p$. Finally, set
\begin{align*}
    D_\infty(t_1,t_2,t_3) = \lim_{p\rightarrow \infty} D_p(t_1,t_2,t_3).
\end{align*}
The limit exists by \cite[Proposition 4.14]{mengLimits2026}. Since non-positive powers of ideals are taken to be the unit ideal by convention, then if $t_i\leq 0$ for some $i$, we have that $D_p=0$ and $D_\infty=0$.
\end{definition}
\begin{notation}[The relevant domain of the $D_\infty$ function]
Although the $D_\infty$ function is well-defined in the whole space, we will mostly use its expression in the region 
\begin{gather*}
    \mathscr{A} := (\mathbb{R}\times [0,1]^2)\cup([0,1]\times \mathbb{R}\times [0,1])\cup ([0,1]^2\times \mathbb{R})
\end{gather*}
\end{notation}
\begin{lemma}[Proposition 4.9, \cite{mengLimits2026}]
    \label{lm:D_infty_formula}
    Let $(t_1,t_2,t_3)\in \mathscr{A}$. Then,
\begin{equation*}
D_\infty(t_1,t_2,t_3) := \begin{cases}
t_2t_3 & B_1, \text{ or } 0\leq t_2,t_3\leq 1\leq t_1,\\
t_1t_3 & B_2, \text{ or } 0\leq t_1,t_3\leq 1\leq t_2,\\
t_1t_2 & B_3, \text{ or } 0\leq t_1,t_2\leq 1\leq t_3,\\
1-t_1-t_2-t_3+t_1t_2+t_1t_3+t_2t_3 & B_4,\\
\frac{2t_1t_2+2t_1t_3+2t_2t_3-t_1^2-t_2^2-t_3^2}{4} & T_0,\\
0 & \text{otherwise.}
\end{cases}
\end{equation*}
where $B_1 := \lbrace t_2+t_3\leq t_1\rbrace\cap [0,1]^3$, $B_2 := \lbrace t_1+t_3\leq t_2\rbrace\cap [0,1]^3$, $B_3 := \lbrace t_1+t_2\leq t_3\rbrace\cap [0,1]^3$, $B_4 := \lbrace 2\leq t_1+t_2+t_3\rbrace\cap [0,1]^3$ and $T_0 := [0,1]^{3}\setminus \bigcup B_i$ (see Figure \ref{fig:domain_D_infty}).

\tdplotsetmaincoords{70}{115}

\begin{figure}[ht]
    \centering
    \begin{subfigure}{0.48\textwidth}
        \centering
        \begin{tikzpicture}[scale=2, tdplot_main_coords]
        \fill[gray!20] (0,1,1) -- (1,1,1) -- (1,0,1) -- cycle;
        \fill[gray!50] (0,1,1) -- (1,1,1) -- (1,1,0) -- cycle;
        \fill[gray!35] (1,1,0) -- (1,1,1) -- (1,0,1) -- cycle;

    \draw[thick,->] (0,0,0) -- (1.5,0,0) node[anchor=north east]{$t_1$};
    \draw[thick,->] (0,0,0) -- (0,1.5,0) node[anchor=south west]{$t_2$};
    \draw[thick,->] (0,0,0) -- (0,0,1.5) node[anchor=south]{$t_3$};

    \draw[black, thick] (0,0,0) -- (1,0,0) -- (1,1,0) -- (0,1,0) -- cycle;
    \draw[black, thick] (0,0,1) -- (1,0,1) -- (1,1,1) -- (0,1,1) -- cycle;
    \draw[black, thick] (0,0,0) -- (0,0,1);
    \draw[black, thick] (1,0,0) -- (1,0,1);
    \draw[black, thick] (0,1,0) -- (0,1,1);
    \draw[black, thick] (1,1,0) -- (1,1,1);

    \foreach \x in {0,1}
    \foreach \y in {0,1}
    \foreach \z in {0,1} {
        \filldraw (\x,\y,\z) circle (0.5pt);
    }

    \draw[red, thick] (0,0,0) -- (0,1,1) -- (1,0,1) -- cycle;
    \draw[red, thick] (0,0,0) -- (1,0,1) -- (1,1,0) -- cycle;
    \draw[red, thick] (0,0,0) -- (0,1,1) -- (1,1,0) -- cycle;
    \draw[red, thick] (0,1,1) -- (1,0,1) -- (1,1,0) -- cycle;

    \filldraw[blue] (0,0,0) circle (1pt) node[anchor=south east]{$(0,0,0)$};
    \filldraw[blue] (0,1,1) circle (1pt) node[anchor=south west]{$(0,1,1)$};
    \filldraw[blue] (1,0,1) circle (1pt) node[anchor=south east]{$(1,0,1)$};
    \filldraw[blue] (1,1,0) circle (1pt) node[anchor=north west]{$(1,1,0)$};

    \filldraw[black] (1,0,0) circle (1pt) node[anchor=south east]{$(1,0,0)$};
    \filldraw[black] (1,1,1) circle (1pt) node[anchor=north west]{$(1,1,1)$};
    \filldraw[black] (0,0,1) circle (1pt) node[anchor=south east]{$(0,0,1)$};
    \filldraw[black] (0,1,0) circle (1pt) node[anchor=north west]{$(0,1,0)$};
\end{tikzpicture}
        \caption{The unit cube as a union of $B_1\sim B_4,T_0$}
    \end{subfigure}
    \hfill
    \begin{subfigure}{0.48\textwidth}
        \centering
        \begin{tikzpicture}[scale=2, tdplot_main_coords]

    \draw[red, thick] (0,0,0) -- (0,1,1) -- (1,0,1) -- cycle;
    \draw[red, thick] (0,0,0) -- (1,0,1) -- (1,1,0) -- cycle;
    \draw[red, thick] (0,0,0) -- (0,1,1) -- (1,1,0) -- cycle;
    \draw[red, thick] (0,1,1) -- (1,0,1) -- (1,1,0) -- cycle;

    \coordinate (A) at (0,0,0.3); 
    \coordinate (B) at (1,0,1.3); 
    \coordinate (C) at (0,1,1.3); 
    \coordinate (D) at (0,0,1.3);

    \draw[black, thick] (A) -- (B);
    \draw[black, thick] (A) -- (C);
    \draw[black, thick] (A) -- (D);
    \draw[black, thick] (B) -- (C);
    \draw[black, thick] (B) -- (D);
    \draw[black, thick] (C) -- (D);

    \coordinate (A) at (0.3,0,0);
    \coordinate (B) at (1.3,0,0); 
    \coordinate (C) at (1.3,0,1); 
    \coordinate (D) at (1.3,1,0);

    \draw[black, thick] (A) -- (B); 
    \draw[black, thick] (A) -- (C); 
    \draw[black, thick] (A) -- (D); 
    \draw[black, thick] (B) -- (C); 
    \draw[black, thick] (B) -- (D); 
    \draw[black, thick] (C) -- (D); 

    \coordinate (A) at (0,0.3,0); 
    \coordinate (B) at (0,1.3,0); 
    \coordinate (C) at (1,1.3,0); 
    \coordinate (D) at (0,1.3,1);

    \draw[black, thick] (A) -- (B);
    \draw[black, thick] (A) -- (C); 
    \draw[black, thick] (A) -- (D);
    \draw[black, thick] (B) -- (C);
    \draw[black, thick] (B) -- (D);
    \draw[black, thick] (C) -- (D);

    \coordinate (A) at (1.3,1.3,1.3); 
    \coordinate (B) at (1.3,0.3,1.3); 
    \coordinate (C) at (0.3,1.3,1.3); 
    \coordinate (D) at (1.3,1.3,0.3);

    \draw[black, thick] (A) -- (B);
    \draw[black, thick] (A) -- (C); 
    \draw[black, thick] (A) -- (D);
    \draw[black, thick] (B) -- (C);
    \draw[black, thick] (B) -- (D);
    \draw[black, thick] (C) -- (D);
    \draw[black,thick] (1.3,0,0) node[anchor=east]{$B_1$};
    \draw[black,thick] (0,1.3,0) node[anchor=west]{$B_2$};
    \draw[black,thick] (0,0,1.3) node[anchor=south]{$B_3$};
    \draw[black,thick] (0,1.1,1.2) node[anchor=west]{$\xleftarrow{}B_4$};
    \draw[red,thick] (1,0,1) node[anchor=east]{$T_0\to$};
\end{tikzpicture}
        \caption{Dividing the cube}
    \end{subfigure}
    \caption{Position of $B_1 \sim B_4,T_0$ (figure taken from \cite{mengLimits2026}). In {\color{gray}gray}: discontinuities of the first derivatives of $D_\infty$.}
    \label{fig:domain_D_infty}
\end{figure}
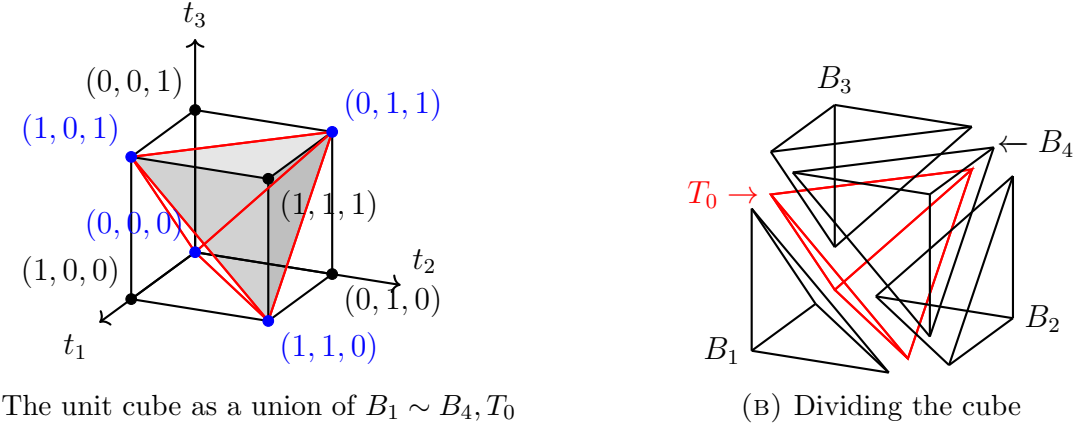
\end{lemma}

\begin{remark}[Differentiability of $D_\infty$]
    \label{rmk:differentiability_D_infty}
The $D_\infty$ function is continuous in $\mathbb{R}^3$, $\mathcal{C}^1$ on $(0,1)^3$. However, the partial derivatives are not well-defined at some points of the boundary of the cube $[0,1]^3$: in the region $B_4$, the left and right partial derivatives $\partial_{t_1}^{\pm}D_\infty$ do not agree when $t_1=1$ and $1+t_2+t_3>2$; and, at the coordinate planes, i.e., at the points $(t_1,t_2,t_3)\in \mathbb{R}^3$ where some $t_i=0$. We note that the discontinuities at the coordinate planes will not be relevant in any of the computations because, although they arise in computation, they will cancel out; however, the discontinuities on the planes $t_i=1$ will need to be taken into account in several computations, because the Riemann–Stieltjes integral captures them as Dirac deltas (see \ref{rmk:domain_of_integration_and_atoms}).
\end{remark}

To state the integral formula, we will use the following notation: let $(R_1,\m_1),(R_2,\m_2)$ and $(R,\m)$ be the local rings $(k[[x_1,...,x_n]],(x_1,\dots,x_n))$, $(k[[y_{1},...,y_m]],(y_{1},\dots,y_m))$, and $(k[[x_1,\dots,x_n,y_1,\dots,y_m]],(x_1,\dots,x_n,y_1,\dots,y_m))$, respectively.

\begin{theorem}[The integral formula for limit $h$-functions, Theorem 5.23, \cite{mengLimits2026}]
\label{th:integral_formula}
Using the notation above, set $f\in R_1$ and $g\in R_2$. Then,
\begin{gather*}
h_{R,\m,f+g,\infty}(x) = \int_{[0,\infty)^2} D_\infty(s,t,x) d_s(-h'_{R_1,\m_1,f,\infty}(s))d_t\left(-h'_{R_2,\m_2,g,\infty}(t)\right),
\end{gather*}
where the integral is a Riemann–Stieltjes integral.
\end{theorem}

\begin{remark}[Domain of integration and Dirac delta distribution]
    \label{rmk:domain_of_integration_and_atoms}
Recall that, by \ref{pr:h_function_analytic_properties}, the $h$-functions $h_{R_1,\m_1,f}$ and $h_{R_2,\m_2,g}$ are $1$ for $x\geq 1$, and therefore in both cases the derivative for $x>1$ is 0. However, in general, the derivative at $1$ coming from the left is often not 0 (by \ref{th:F_invariants_lim_diff_h}, one sees that this happens exactly when the hypersurface is strongly $F$-regular). Therefore, one should expect discontinuities at $1$ in the integrators $h'$. The Riemann–Stieltjes theory of integrals takes these discontinuities into account through the use of Dirac delta distributions. We refer the reader to \cite[Section 2]{mengLimits2026}, but we recall here a basic example that showcases how it works: if $f(x) = a$ for $x\leq 1$ and $b$ for $x\geq 1$, then $df = (b-a)\delta_1$.

Taking this into account, we will write
\begin{multline*}
h_{R,\m,f+g,\infty}(x) = \int_0^\infty \int_{0}^\infty D_\infty(s,t,x) d_s(-h'_{R_1,\m_1,f,\infty}(s))d_t\left(-h'_{R_2,\m_2,g,\infty}(t)\right) \\
= \int_0^{1^+}\int_0^{1^+} D_\infty(s,t,x) d_s\left(-h'_{R_1,\m_1,f,\infty}(s)\right)d_t\left(-h'_{R_2,\m_2,g,\infty}(t)\right),
\end{multline*}
where the $1^+$ accounts for possible Dirac deltas at $t=1$.
\end{remark}

\subsection{The reflection of the \texorpdfstring{$h$}{h}-function of a hypersurface}
We prove a functional relation of the $h$-function of a hypersurface with its reflection (in particular, the $F$-signature of pairs of a double point). This formula extends Huneke and Leuschke's identity \cite[Proposition 15]{hunekeTwoTheoremsMaximal2004} between the $F$-signature and the Hilbert–Kunz multiplicity to the entire $h$-function.

\begin{theorem}[The reflection of the $h$-function of a hypersurface]
	\label{th:reflection}
Let $(R,\m,k)$ be a regular local domain of dimension $d$ and equicharacteristic, of any characteristic, with infinite residue field $k$, and $0\neq f\in R$. Let the ideal $I\subset R/fR$ be a minimal reduction of $\m R/fR$ that is a parameter ideal in $\m R/fR$. Then,
\begin{gather*}
h_{R,f,I:\m}(1-t) = h_{R,f,\m}(t) - h_{R,f,I}(t) + e(R/fR)-1.
\end{gather*}
Since $h_{R,f,I}(t) = e(R/fR)t$, equivalently,
\begin{gather*}
h_{R,f,I:\m}(1-t) = h_{R,f,\m}(t)-e(R/fR)(t-1)-1.
\end{gather*}
Equivalently, for the $\phi$-function (see \ref{def:phi_function} below),
\begin{gather*}
\phi_{I:\m,f}(1-t) = \phi_{\m,f}(t)-e(R/fR)(t-1)-1
\end{gather*}
In particular, this recovers the hypersurface case of the well-known identity by Huneke and Leuschke that $\eHK(I:\m) = e(R/fR)-s(R/fR)$.
\end{theorem}
Before the proof, we derive the corollary most useful for our purposes in this article.

\begin{corollary}[The reflection of the $F$-signature of pairs of a Gorenstein double point]
	\label{cor:reflection_h_function_double_points}
	Under the same setting as \ref{th:reflection}, assume further that $e(R/fR) = 2$. Then,
    \begin{gather*}
        %\color{PineGreen}
    h_{f}(1-t)=h_f(t)-2t+1
    \end{gather*}
    or in terms of the $F$-signature of pairs,
	\begin{gather*}
        %\color{PineGreen}
	s(R,f^{1-t}) = s(R,f^t)+2t-1.
	\end{gather*}
\end{corollary}
\begin{proof}
Since $R/fR$ is a Gorenstein double point, we have that $(I,f):\m = \m$, and the result follows from \ref{th:reflection}.
\end{proof}

To prove this result, we need to interpret the recall the relation between the $h$-function and Monsky and Teixeira's $\phi$-function:
\begin{definition}[The $\phi$-function of a hypersurface, \cite{monskyPFractals2006,shidelerFsignatureFunctionsDiagonal2024}]
	\label{def:phi_function}
Let $k$ be a field of characteristic $p>0$, $f\in k[[x_1,...,x_n]]$, and $\mathscr{I} = [0,1]\cap \lbrace \frac{a}{p^e}:a,e\in \mathbb{N}\rbrace$. Let $I$ be an $\m$-primary ideal. Then,
\begin{equation*}
\begin{array}{rcl}
\phi_{I,f} : \mathscr{I}&\longrightarrow &\mathbb{Q}\\
		\dfrac{a}{p^e}&\longmapsto &\dfrac{1}{p^{ne}}\dim_k\left(\dfrac{k[[x_1,\dots,x_n]]}{I^{[p^e]}+(f^a)}\right).
\end{array}
\end{equation*}
\end{definition}
\begin{lemma}[Proposition 4.1, \cite{blickleFsignaturePairsContinuity2013}]
    \label{lm:phi_function_and_h_function}
	Under the same notation as definition \ref{def:phi_function}, we have that
\begin{equation*}
	h_{k[[x_1,...,x_n]],f,I}\left(\dfrac{a}{p^c}\right) = \phi_{I,f}\left(\dfrac{a}{p^c}\right).
\end{equation*}
\end{lemma}

The proof of \ref{th:reflection} requires the following technical lemma:
\begin{lemma}[cf. Propositions 5.6 and 5.7, \cite{monskyPFractals2006}]
	\label{lm:technical_reflection_lemma}
Let $(R,\m)$ be a $d$-dimensional regular local ring of characteristic $p>0$, and $x_1,\dots,x_{d-1},f$ a regular sequence. Let $I = (x_1,\dots,x_{d-1})$, $J := (I,f):\m$, and $0\leq a \leq q-1$. Then
\begin{equation*}
J^{[q]}:f^{q-a} = (f^a,I^{[q]}):(f^a,\m^{[q]}).
\end{equation*}
Since the ideal $(f^a,I^{[p^e]})$ is a complete intersection, it follows that
\begin{equation*}
\ell\left(\dfrac{R}{J^{[p^e]}:f^{p^e-a}}\right) = \ell\left(\dfrac{R}{(I^{[p^e]},f^a)}\right) - \ell\left(\dfrac{R}{(\m^{[p^e]},f^a)}\right).
\end{equation*}
\end{lemma}
\begin{proof}
For starters, note that since $R$ is regular, one has that $J^{[q]} = (I^{[q]},f^q):\m^{[q]}$, and therefore $J^{[q]}:f^{q-a} = ((I^{[q]},f^q):\m^{[q]}):f^{q-a} = ((I^{[q]},f^q):f^{q-a}):\m^{[q]}$.

Now, we note that $(I^{[q]},f^q):f^{q-a} = (I^{[q]},f^a)$ (cf. \cite[Lemma 3.6]{shidelerFsignatureFunctionsDiagonal2024}). Indeed, we only need to show one containment, since the other is clear: let $c\in (I^{[q]},f^q):f^{q-a}$, so
\begin{gather*}
cf^{q-a} = \sum a_ix_i + bf^q \Rightarrow f^{q-a}(c-bf^a) \in I^{[q]} \Rightarrow c\in (I^{[q]}:f^{q-a}) + f^aR,
\end{gather*}
but, since $x_1,\dots,x_{d-1},f$ form an $R$-regular sequence, also $x_1^q,\dots,x_{d-1}^q,f^{q-a}$ do, which means that $(I^{[q]}:f^{q-a}) = I^{[q]}$. In other words, this shows that $c\in (I^{[q]},f^q):f^{q-a}\subseteq (I^{[q]},f^{a})$. All this ultimately implies that $J^{[q]}:f^{q-a} = (I^{[q]},f^{a}):\m^{[q]} = (I^{[q]},f^{a}):(\m^{[q]},f^a).$

As for the second claim, observe that the ideal $(I^{[q]},f^{a})$ is an $\m$-primary complete intersection ideal, hence \cite[Exercise 3.2.15]{brunsherzog1998} implies that
\begin{gather*}
\dfrac{R}{(I^{[q]},f^a):(\m^{[q]},f^a)} \cong \dfrac{(\m^{[q]},f^a)}{(I^{[q]},f^a)},
\end{gather*}
as $R/(I^{[q]},f^a)$-modules, and the second claim follows.
\end{proof}

\begin{proof}[Proof of theorem \ref{th:reflection}]
Since $h_{R,f,\m}, h_{R,f,I},$ and $h_{R,f,J}$ are continuous and coincide with $\phi_{\m,f},\phi_{I,f}$, and $\phi_{J,f}$, respectively, in a dense set of $[0,1]$, it is enough to show the result for the respective $\phi$-functions.

Take generators $x_1,\dots,x_{d-1}\in R$ of the ideal $I$, so that they form a parameter sequence. This implies that $x_1,\dots,x_{d-1},f$ is a parameter sequence of $R$ and thus they form an $R$-regular sequence, since $R$ is regular, and thus we are in the hypotheses of Lemma \ref{lm:technical_reflection_lemma}.

Thus, we continue by noting that for any ideal $J$, and an element $x\in R$ we have that
\begin{gather*}
0\longrightarrow \dfrac{R}{J:x}\stackrel{\cdot x}\longrightarrow \dfrac{R}{J}\longrightarrow \dfrac{R}{(J,x)}\longrightarrow 0.
\end{gather*}
Letting $J = (I,f):\m$ and $x = f^{q-a}$, it follows
\begin{gather*}
\ell\left(\dfrac{R}{J^{[q]}:f^{q-a}}\right)+\ell\left(\dfrac{R}{(J^{[q]},f^{q-a})}\right) = \ell\left(\dfrac{R}{J^{[q]}}\right).
\end{gather*}
Using \ref{lm:technical_reflection_lemma} we thus conclude that
\begin{gather*}
\ell\left(\dfrac{R}{(I^{[q]},f^{a})}\right)-\ell\left(\dfrac{R}{(\m^{[q]},f^{a})}\right)+\ell\left(\dfrac{R}{(J^{[q]},f^{q-a})}\right) = \ell\left(\dfrac{R}{J^{[q]}}\right).
\end{gather*}
Note that $\ell\left(R/J^{[q]}\right) = \ell\left(R/(I,f)^{[q]}\right)-\ell\left(R/\m^{[q]}\right)$. Since $R$ is regular and $(I,f)$ is a parameter ideal, $\ell\left(R/(I,f)^{[q]}\right) = q^d\ell\left(R/(I,f)\right)$, and since $(I,f)/fR$ is a minimal reduction of $\m/fR$, we have that $\ell\left(R/(I,f)\right) = e(R/fR)$. Also, $\ell\left(R/\m^{[q]}\right) = q^d$. In other words, this all yields
\begin{gather*}
\phi_{I,f}\left(\dfrac{a}{q}\right) - \phi_{\m,f}\left(\dfrac{a}{q}\right) + \phi_{J,f}\left(\dfrac{q-a}{q}\right) = e(R/fR)-1,
\end{gather*}
and the result for $h$-functions thus follows. Note that since $\ell\left(R/(I,f)^{[q]}\right) = q^d\ell\left(R/(I,f)\right)$, then $h_{R,I,f}(t) = e(R/fR)t$.
\end{proof}

\subsection{The simple or ADE singularities}
Finally, let us introduce the equations of the family of singularities we will study in this work:
\begin{theoremdefinition}[The simple singularities, \cite{arnoldNormalFormsFunctions1972}]
    \label{def:simple_singularities}
Let $f\in \mathbb{C}[[x_0,\dots,x_d]]$. The ring $\mathbb{C}[[x_0,\dots,x_d]]/(f)$ is said to be (the coordinate ring of) a simple singularity if it is isomorphic as a $\mathbb{C}$-algebra to $\mathbb{C}[[x_0,\dots,x_d]]/(g(x_0,x_1)+x_2^2+\dots+x_d^2)$ where $g$ is one of the binomials in the table below.
\begin{table}[h] 
\centering 
\renewcommand{\arraystretch}{2} 
\begin{tabular}{c|c|c|c|c|c} 
    \toprule 
    Name & $A_n, n\geq 1$ & $D_n,n\geq 4$ & $E_6$ & $E_7$ & $E_8$\\ 
    \midrule 
    $g(x,y)$ & $x^{n+1}+y^2$ & $x^2y+y^{n-1}$ & $x^3+y^4$ & $x^3+xy^4$ & $x^3+y^5$\\ 
    \bottomrule 
\end{tabular}
\label{tab:simple_singularities_binomials_definition} 
\end{table}
\end{theoremdefinition}
Arnol'd's theorem shows that the simple singularities are \say{suspensions} of the so-called ADE surface singularities, i.e., successively adding $x^2$ terms.

\section{Limit \texorpdfstring{$h$}{h}-functions of the \texorpdfstring{$A_n$}{An} singularities.}
\label{sec:limit_h_functions_A_n}
The main results of this section are the formula for the generating function of the limit Hilbert–Kunz multiplicities of the $A_n$ singularities, and a formula for the limit $\phi$-function of all the $A_n$ singularities, which is shown to be a combination of Euler polynomials, generalizing \cite[Theorem 4.4]{shidelerFsignatureFunctionsDiagonal2024} (see \ref{th:limit_h_function_A_n}).

The following auxiliary function will be useful throughout the rest of the paper, see \cite[Section 8.1]{mengLimits2026}.
\begin{definition}
 \label{def:K_function}
Set
\begin{equation*}
K_n(t,x) := D_\infty\left(\frac{1}{n},t,x\right) = \begin{cases}
xt & (t,x)\in \Delta_1\\
\frac{1}{n}x & (t,x)\in \Delta_2\\
\frac{1}{n}t & (t,x)\in \Delta_3\cup \Delta_5\\
xt-(1-\frac{1}{n})t-(1-\frac{1}{n})x+1-\frac{1}{n} & (t,x)\in \Delta_4\\
-\frac{(x-t)^2}{4}+\frac{x}{2n}+\frac{t}{2n}-\frac{1}{4n^2} & (t,x)\in \Delta_0
\end{cases}
\end{equation*}
where $\Delta_i := B_i\cap \lbrace t_1=\dfrac{1}{n}\rbrace$ (see Figures \ref{fig:sections_K2_K3} and \ref{fig:domain_K_functions}), for $B_i$ and $D_\infty(s,t,x)$ as in \ref{lm:D_infty_formula}. For convenience, we will set $K(t,x) := K_2(t,x)$.

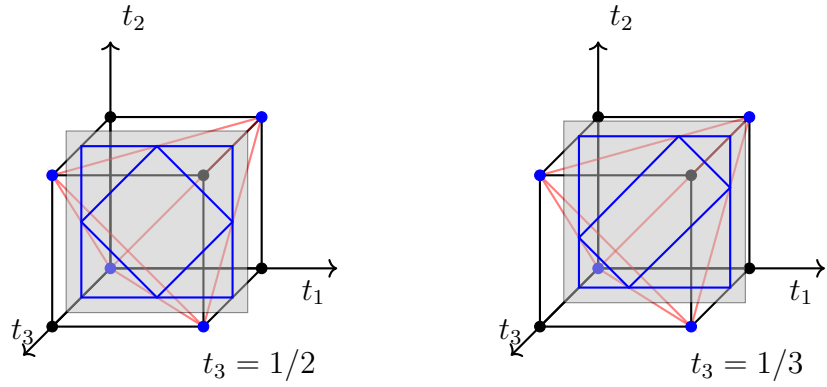
\begin{figure}[ht]
\centering
\begin{tikzpicture}[scale=2]
    
    \draw[thick,->] (0,0,0) -- (1.5,0,0) node[anchor=north east]{$t_1$};
    \draw[thick,->] (0,0,0) -- (0,1.5,0) node[anchor=south west]{$t_2$};
    \draw[thick,->] (0,0,0) -- (0,0,1.5) node[anchor=south]{$t_3$};

    \draw[black, thick] (0,0,0) -- (1,0,0) -- (1,1,0) -- (0,1,0) -- cycle;
    \draw[black, thick] (0,0,1) -- (1,0,1) -- (1,1,1) -- (0,1,1) -- cycle;
    \draw[black, thick] (0,0,0) -- (0,0,1);
    \draw[black, thick] (1,0,0) -- (1,0,1);
    \draw[black, thick] (0,1,0) -- (0,1,1);
    \draw[black, thick] (1,1,0) -- (1,1,1);

    \foreach \x in {0,1}
    \foreach \y in {0,1}
    \foreach \z in {0,1} {
        \filldraw (\x,\y,\z) circle (0.5pt);
    }

    \draw[red, thick, opacity=0.3] (0,0,0) -- (0,1,1) -- (1,0,1) -- cycle;
    \draw[red, thick, opacity=0.3] (0,0,0) -- (1,0,1) -- (1,1,0) -- cycle;
    \draw[red, thick, opacity=0.3] (0,0,0) -- (0,1,1) -- (1,1,0) -- cycle;
    \draw[red, thick, opacity=0.3] (0,1,1) -- (1,0,1) -- (1,1,0) -- cycle;

    \filldraw[blue] (0,0,0) circle (1pt) node[anchor=south east]{};
    \filldraw[blue] (0,1,1) circle (1pt) node[anchor=south west]{};
    \filldraw[blue] (1,0,1) circle (1pt) node[anchor=south east]{};
    \filldraw[blue] (1,1,0) circle (1pt) node[anchor=north west]{};

    \filldraw[black] (1,0,0) circle (1pt) node[anchor=south east]{};
    \filldraw[black] (1,1,1) circle (1pt) node[anchor=north west]{};
    \filldraw[black] (0,0,1) circle (1pt) node[anchor=south east]{};
    \filldraw[black] (0,1,0) circle (1pt) node[anchor=north west]{};

    \filldraw[black] (1,0,1.2) node[anchor=north west]{$t_3=1/2$};

    \coordinate (A) at (1.1,-0.1,0.5);
    \coordinate (B) at (1.1,1.1,0.5);
    \coordinate (C) at (-0.1,1.1,0.5);
    \coordinate (D) at (-0.1,-0.1,0.5);
    \filldraw[fill=gray!50, opacity=0.5] (A) -- (B) -- (C) -- (D) -- cycle;

    \draw[blue, thick] (0.5,0,0.5) -- (1,0.5,0.5) -- (0.5,1,0.5) -- (0,0.5,0.5) -- cycle;
    \draw[blue, thick] (0,0,0.5) -- (1,0,0.5) -- (1,1,0.5) -- (0,1,0.5) -- cycle;
\end{tikzpicture}
\hspace{4em}
\begin{tikzpicture}[scale=2]
    % Axes
    \draw[thick,->] (0,0,0) -- (1.5,0,0) node[anchor=north east]{$t_1$};
    \draw[thick,->] (0,0,0) -- (0,1.5,0) node[anchor=south west]{$t_2$};
    \draw[thick,->] (0,0,0) -- (0,0,1.5) node[anchor=south]{$t_3$};

    % Cube    
    \draw[black, thick] (0,0,0) -- (1,0,0) -- (1,1,0) -- (0,1,0) -- cycle;
    \draw[black, thick] (0,0,1) -- (1,0,1) -- (1,1,1) -- (0,1,1) -- cycle;
    \draw[black, thick] (0,0,0) -- (0,0,1);
    \draw[black, thick] (1,0,0) -- (1,0,1);
    \draw[black, thick] (0,1,0) -- (0,1,1);
    \draw[black, thick] (1,1,0) -- (1,1,1);

    \foreach \x in {0,1}
    \foreach \y in {0,1}
    \foreach \z in {0,1} {
        \filldraw (\x,\y,\z) circle (0.5pt);
    }

    % Central tetrahedron
    \draw[red, thick, opacity=0.3] (0,0,0) -- (0,1,1) -- (1,0,1) -- cycle;
    \draw[red, thick, opacity=0.3] (0,0,0) -- (1,0,1) -- (1,1,0) -- cycle;
    \draw[red, thick, opacity=0.3] (0,0,0) -- (0,1,1) -- (1,1,0) -- cycle;
    \draw[red, thick, opacity=0.3] (0,1,1) -- (1,0,1) -- (1,1,0) -- cycle;

    % Blue dots
    \filldraw[blue] (0,0,0) circle (1pt) node[anchor=south east]{};
    \filldraw[blue] (0,1,1) circle (1pt) node[anchor=south west]{};
    \filldraw[blue] (1,0,1) circle (1pt) node[anchor=south east]{};
    \filldraw[blue] (1,1,0) circle (1pt) node[anchor=north west]{};

    % Black dots
    \filldraw[black] (1,0,0) circle (1pt) node[anchor=south east]{};
    \filldraw[black] (1,1,1) circle (1pt) node[anchor=north west]{};
    \filldraw[black] (0,0,1) circle (1pt) node[anchor=south east]{};
    \filldraw[black] (0,1,0) circle (1pt) node[anchor=north west]{};

    % Hyperplane section tag
    \filldraw[black] (1,0,1.2) node[anchor=north west]{$t_3=1/3$};

    % Square shadowed section
    \coordinate (A) at (1.1,-0.1,0.33);
    \coordinate (B) at (1.1,1.1,0.33);
    \coordinate (C) at (-0.1,1.1,0.33);
    \coordinate (D) at (-0.1,-0.1,0.33);
    \filldraw[fill=gray!50, opacity=0.5] (A) -- (B) -- (C) -- (D) -- cycle;

    % Intersection of section with tetrahedron and cube
    \draw[blue, thick] (0.33,0,0.33) -- (1,0.66,0.33) -- (0.66,1,0.33) -- (0,0.33,0.33) -- cycle;

    \draw[blue, thick] (0,0,0.33) -- (1,0,0.33) -- (1,1,0.33) -- (0,1,0.33) -- cycle;
\end{tikzpicture}
\caption{Section of the unit cube at $t_3=1/2$ and $t_3 = 1/3$ (based on figures from \cite{mengLimits2026}).}
\label{fig:sections_K2_K3}
\end{figure}
\begin{figure}
\centering
\begin{tikzpicture}[scale=2.5, thick]

    \draw[->] (-0.25,0) -- (1.5,0) node[right]{$t = t_2$};
    \draw[->] (0,-0.25) -- (0,1.5) node[above]{$x = t_3 $};

    \draw (0,0) rectangle (1,1);

    \draw (0,0.5) -- (0.5,1) -- (1,0.5) -- (0.5,0) -- cycle;
    \draw[color=red] (1,1) -- (1,0.5);

    \node[anchor=north east] at (0.25,0.25) {\tiny $\Delta_1$};
    \node[anchor=north west] at (0.75,0.25) {\tiny $\Delta_2$};
    \node[anchor=south east] at (0.25,0.75) {\tiny $\Delta_3$};
    \node[anchor=south west] at (0.75,0.75) {\tiny $\Delta_4$};
    \node at (0.5,0.5) {\tiny $\Delta_0$};
    \node at (0.5,1.2) {\tiny $\Delta_5$};

    \draw (1,1) -- (1,1.5);
    \draw (1,1) -- (1.5,1);

    \filldraw (0,0) circle (1pt) node[below left]{$(0,0)$};
    \filldraw (1,0) circle (1pt) node[below right]{$(1,0)$};
    \filldraw (1,1) circle (1pt) node[above right]{$(1,1)$};
    \filldraw (0,1) circle (1pt) node[above left]{$(0,1)$};
    \filldraw (0.5,0) circle (1pt);
    \filldraw (1,0.5) circle (1pt);
    \filldraw (0.5,1) circle (1pt);
    \filldraw (0,0.5) circle (1pt);

\end{tikzpicture}
\begin{tikzpicture}[scale=2.5, thick]

    \draw[->] (-0.25,0) -- (1.5,0) node[right]{$t = t_2$};
    \draw[->] (0,-0.25) -- (0,1.5) node[above]{$x = t_3$};

    \draw (0,0) rectangle (1,1);
    \draw[color=red] (1,1) -- (1,0.66);

    \draw (0,0.33) -- (0.66,1) -- (1,0.66) -- (0.33,0) -- cycle;
    
    % Tags
    \node[anchor=north east] at (0.25,0.2) {\tiny $\Delta_1$};
    \node[anchor=north west] at (0.75,0.25) {\tiny $\Delta_2$};
    \node[anchor=south west] at (0.77,0.8) {\tiny $\Delta_4$};
    \node[anchor=south east] at (0.25,0.75) {\tiny $\Delta_3$};
    \node at (0.5,0.5) {\tiny $\Delta_0$};
    \node at (0.5,1.2) {\tiny $\Delta_5$};

    \draw (1,1) -- (1,1.5);
    \draw (1,1) -- (1.5,1);

    \filldraw (0,0) circle (1pt) node[below left]{$(0,0)$};
    \filldraw (1,0) circle (1pt) node[below right]{$(1,0)$};
    \filldraw (1,1) circle (1pt) node[above right]{$(1,1)$};
    \filldraw (0,1) circle (1pt) node[above left]{$(0,1)$};
    \filldraw (0,0.33) circle (1pt);
    \filldraw (1,0.66) circle (1pt);
    \filldraw (0.66,1) circle (1pt);
    \filldraw (0.33,0) circle (1pt);

\end{tikzpicture}
\caption{The regions of the $K_2$ and $K_3$ function (figure taken from \cite{mengLimits2026}). In {\color{red}red}, discontinuities of $\partial_t K_2$ and $\partial_t K_3$.}%Add drawing of sections of the cube as Meng?
\label{fig:domain_K_functions}
\end{figure}
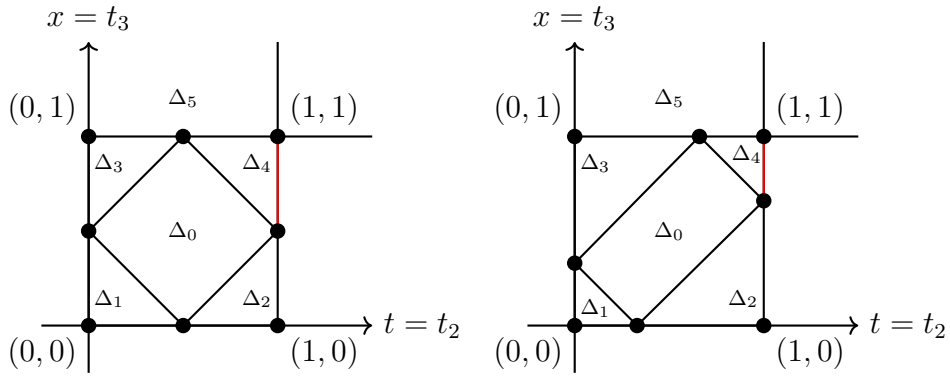
\end{definition}
\begin{remark}
	\label{rmk:derivative_K_function}
Note that
\begin{equation*}
\partial_tK_n(t,x) =\begin{cases}
x & (t,x)\in \Delta_1\\
0 & (t,x)\in \Delta_2\\
\frac{1}{n} & (t,x)\in \Delta_3\cup \Delta_5\\
x-\frac{n-1}{n} & (t,x)\in \Delta_4\\
\frac{(x-t)}{2}+\frac{1}{2n} & (t,x)\in \Delta_0
\end{cases} \qquad d_t\left(\partial_tK_n(t,x)\right) =\begin{cases}
-\frac{1}{2} & (t,x)\in \Delta_0\\
(\frac{n-1}{n}-x)\delta_1(t) & \frac{n-1}{n}\leq x \leq 1, t=1\\
0 & \text{otherwise}
\end{cases}
\end{equation*}
where the Dirac delta distribution $\delta$ comes from the discontinuity of $\partial_t K_n$ at $t=1$ and $x\in [\frac{n-1}{n},1]$ (see \ref{rmk:differentiability_D_infty}).
\end{remark}

\subsection{The limit \texorpdfstring{$h$}{h}-function of the \texorpdfstring{$A_n$}{An} singularities and Euler polynomials.}
In this section, we show that the limit $h$-functions of the $A_n$ singularities can be expressed as Euler polynomials.

\begin{definition}[Euler polynomials]
\label{def:Euler_Polynomials}
The Euler polynomials are the polynomials $E_k(x)\in \mathbb{Q}[x]$ such that
\begin{equation*}
\alpha(x) := \dfrac{2e^{zx}}{e^z+1} = \sum_{k=0}^{\infty} E_k(x) \dfrac{z^k}{k!}.
\end{equation*}
\end{definition}

The result was inspired by \cite[Theorem 6.2]{hoffmanDerivativePolynomialsEuler1999} and \cite[Theorem 4.4]{shidelerFsignatureFunctionsDiagonal2024}.
\begin{theorem}[The limit $h$-functions of the $A_n$ singularities]
    \label{th:limit_h_function_A_n}
Let $n \geq 2$. Let $f = x_0^n+x_1^2+\dots+x_d^2 \in k[[x_0,\dots,x_d]]$, and set
\begin{equation*}
a_{d,n} = \begin{cases}
        \frac{1}{2}-\frac{1}{n} & d \text{ even,}\\
        \frac{1}{n} & d \text{ odd.}
        \end{cases} \qquad c_{d,n} = (-1)^{\lfloor d/2 \rfloor} \frac{2^{d-3}\, n}{d!}
\end{equation*}
Then, for $d\geq 1$,
\begin{equation*}
h_{f,\infty}(x) = \begin{cases}
x + c_{d+1,n}\left(E_{d+1}(x+a_{d+1,n})-E_{d+1}(-x+a_{d+1,n})\right) & x\in [0,a_{d+1,n}],\\[7pt]
x + c_{d+1,n}\left(E_{d+1}(x+a_{d+1,n})+E_{d+1}(-x+1+a_{d+1,n})\right) & x\in [a_{d+1,n},1-a_{d+1,n}],\\[7pt]
x + c_{d+1,n}\left(-E_{d+1}(x-1+a_{d+1,n})+E_{d+1}(-x+1+a_{d+1,n})\right) & x\in [1-a_{d+1,n},1].
\end{cases}
\end{equation*}
and for $d=0$,
\begin{equation*}
    h_{x_0^n,\infty}(x) = \begin{cases}
    nx & x\in [0,\frac{1}{n}],\\[7pt]
    1 & x\in [\frac{1}{n},1].
    \end{cases}
\end{equation*}
\end{theorem}
\begin{proof}
    Let us call $h_d := h_{x_0^n+x_1^2+\dots+x_d^2}$ for the sake of simplicity. For $d=1$, by the integral formula \ref{th:integral_formula}, we have that
    \begin{equation*}
        h_1(x) = \int_0^{1^+} 2K(t,x)d_t\left(-h_0'(t)\right) = 2nK\left(\dfrac{1}{n},x\right),
    \end{equation*}
    where we are using that $d_t\left(-\dfrac{\partial}{\partial t}h_0'(t)\right) = n(\delta_{1/n}-\delta_0)$. Explicitly, by \ref{def:K_function},
    \begin{equation*}
    h_1(x) = \begin{cases}
    2x & x\in \left[0,\dfrac{1}{2}-\dfrac{1}{n}\right],\\[10pt]
    -2n(x-\frac{1}{n})^2+\frac{nx}{2}+\frac{1}{2}-\frac{n}{8} & x\in \left[\dfrac{1}{2}-\dfrac{1}{n},\dfrac{1}{2}+\dfrac{1}{n}\right],\\[10pt]
    1 & x\in \left[\dfrac{1}{2}+\dfrac{1}{n},1\right].
    \end{cases}
    \end{equation*}
    One can verify the formula directly, using that $E_2(t) = t^2-t$.
    
    Let us now assume that $d\geq 2$, and proceed by induction. The integral formula \ref{th:integral_formula} together with integration by parts (see \cite[Corollary 2.26]{mengLimits2026}) implies that, for $d\geq 1$, %we don't let d=0 so we get enough differentiability
    \begin{equation*}
        h_{d+1}(x) = \int_0^{1^+} 2h_d(t)d_t\left(\dfrac{\partial}{\partial t^2} K(t,x)\right) = \int_{0}^{1^+} 2h_d(t)d_t\left(\dfrac{\partial}{\partial t^2} K(t,x)\right),
    \end{equation*}
    using that $K(0,x)=0,\partial_t^{+}h_d(1)=0,\partial_t^+K(1,x)=0,\partial_tK(0,x)=0$. From the definition of the function $K$ \ref{def:K_function} and the remark afterwards, we conclude that for $0\leq x \leq \frac{1}{2}$,
    \begin{equation*}
    h_{d+1}(x) = \int_{\frac{1}{2}-x}^{\frac{1}{2}+x} h_d(t)dt.
    \end{equation*}
    
    By \ref{cor:reflection_h_function_double_points}, it suffices to compute $h_{d+1}$ at $[0,\frac{1}{2}]$. Let first $x\in [0,a_{d+2,n}]$. Note that
    \begin{equation*}
        a_{d+1,n} = \frac{1}{2}-a_{d+2,n}\leq \frac{1}{2}-x \leq \frac{1}{2}+x \leq \frac{1}{2}+a_{d+2,n} = 1-a_{d+1,n}
    \end{equation*}
    hence, for $x\in [0,a_{d+2,n}]$, we have that
    \begin{align*}
        \int_{\frac{1}{2}-x}^{\frac{1}{2}+x} h_d(t)dt &= \int_{\frac{1}{2}-x}^{\frac{1}{2}+x} t + c_{d+1,n}\left(E_{d+1}(t+a_{d+1,n})+E_{d+1}(-t+1+a_{d+1,n})\right) dt=\\
        &= x + c_{d+1,n}\dfrac{2}{d+2}\left(E_{d+2}(x+1-a_{d+2,n})-E_{d+2}(-x+1-a_{d+2,n})\right).
    \end{align*}
    where we are using the fact that $\int_a^b E_d(t)dt = \left.\dfrac{1}{d+1}E_{d+1}(t)\right|_{a}^b$ (see \cite[Lemma 4.2]{shidelerFsignatureFunctionsDiagonal2024} for some properties of the Euler Polynomials). Recall that $\sum E_k(x)\frac{z^k}{k!} = \frac{2e^{xz}}{e^z+1}$, and, for all $c\in \mathbb{R}$,
    \begin{align}
        \label{fm:exponential_rational_identity}
        \dfrac{e^{(\pm x+1-c)z}}{e^z+1} = \dfrac{e^z}{e^z+1}e^{(\pm x-c)z} = \dfrac{1}{e^{-z}+1}e^{(\pm x-c)z} = \dfrac{e^{(\mp x+c)(-z)}}{e^{-z}+1},
    \end{align}
    using that $(e^z+1)e^{-z} = 1+e^{-z}$. Therefore, $E_{d+2}(\pm x+1-a_{d+2,n}) = (-1)^{d+2}E_{d+2}(\mp x+a_{d+2,n})$, and we conclude that for $x\in [0,a_{d+2,n}]$,
    \begin{equation*}
    h_{d+1}(x) = x+c_{d+1,n}\frac{2(-1)^{d+2}}{d+2}\left(E_{d+2}(-x+a_{d+2,n})-E_{d+2}(x+a_{d+2,n})\right),
    \end{equation*}
    and the result follows from the fact that $c_{d+2} = -c_{d+1}\frac{2(-1)^{d+2}}{d+2}$.

    Let now $x\in [a_{d+2,n},\frac{1}{2}]$. Then,
    \begin{equation*}
        0\leq \dfrac{1}{2}-x \leq \dfrac{1}{2}-a_{d+2,n}, \quad \text{and} \quad \dfrac{1}{2}+a_{d+2,n} \leq \dfrac{1}{2}+x \leq 1,
    \end{equation*}
    so
    \begin{align*}
    \int_{\frac{1}{2}-x}^{\frac{1}{2}+x} h_d(t)dt &= \int_{\frac{1}{2}-x}^{\frac{1}{2}-a_{d+2,n}} h_d(t)dt+\int_{\frac{1}{2}-a_{d+2,n}}^{\frac{1}{2}+a_{d+2,n}} h_d(t)dt+\int_{\frac{1}{2}+a_{d+2,n}}^{\frac{1}{2}+x} h_d(t)dt,\\
        &= x + c_{d+1,n}\dfrac{-2}{d+2}\left(E_{d+2}(-x+1-a_{d+2,n})+E_{d+2}(x-a_{d+2,n})\right),
    \end{align*}
    where one needs to use the property that $E_k(0) = -E_k(1)$ (also in \cite[Lemma 4.2]{shidelerFsignatureFunctionsDiagonal2024}). It follows from that and the identity in \ref{fm:exponential_rational_identity} that
    \begin{equation*}
        h_{d+1}(x) = x + c_{d+2,n}\left(E_{d+2}(x+a_{d+2,n})+E_{d+2}(-x+1+a_{d+2,n})\right),
    \end{equation*}
    whenever $x\in [a_{d+2,n},\frac{1}{2}]$.
\end{proof}

See Figure \ref{fig:limit_h_functions_A2_example} for examples of $h$-functions of the $A_2$ singularity in dimensions $1$ through $3$.

\begin{figure}[ht]
\centering
\begin{tikzpicture}[
    xscale=12,
    yscale=1,
    >=stealth,
    every node/.style={font=\footnotesize},
]
\def\spaces{1.6}

% ------------------------------------------------------------------
%   s one  (breakpoints 1/6, 5/6)
% ------------------------------------------------------------------
\def\yone{1.2}
\draw[thick] (0,\yone) -- (1,\yone);
\foreach \xt/\xlbl in {0/0, 0.16666666/{\tfrac{1}{6}}, 0.83333333/{\tfrac{5}{6}}, 1/1} {
  \draw (\xt,\yone) -- (\xt,\yone-0.07);
  \node[below] at (\xt,\yone-0.07) {$\xlbl$};
}
\node[left] at (-0.02,\yone) {$\dim = 1$}; % Segment tag
% Polynomials for s=1 (representative odd case)
\node[above=2pt] at (0.0833,\yone) {${\color{blue}2}x$};
\node[above=2pt] at (0.5,\yone)    {$-\tfrac{3}{2}x^{2}+\tfrac{5}{2}x-\tfrac{1}{24}$};
\node[above=2pt] at (0.9167,\yone) {$1$};

% ------------------------------------------------------------------
%   s two  (breakpoints 1/3, 2/3)
% ------------------------------------------------------------------
\def\ytwo{\yone-\spaces}
\draw[thick] (0,\ytwo) -- (1,\ytwo);
\foreach \xt/\xlbl in {0/0, 0.33333333/{\tfrac{1}{3}}, 0.66666666/{\tfrac{2}{3}}, 1/1} {
  \draw (\xt,\ytwo) -- (\xt,\ytwo-0.07);
  \node[below] at (\xt,\ytwo-0.07) {$\xlbl$};
}
\node[left] at (-0.02,\ytwo) {$\dim = 2$}; % Segment tag
% Polynomials for s=2 (representative even case)
\node[above=2pt] at (0.1667,\ytwo) {$-x^{3}+{\color{blue}\tfrac{5}{3}}x$};
\node[above=2pt] at (0.5,\ytwo)    {$-x^{2}+2x-\tfrac{1}{27}$};
\node[above=2pt] at (0.8333,\ytwo) {$x^{3}-3x^{2}+\tfrac{10}{3}x-\tfrac{1}{3}$};

% ------------------------------------------------------------------
%   s three  (breakpoints 1/6, 5/6)
% ------------------------------------------------------------------
\def\ythree{\ytwo-\spaces}
\draw[thick] (0,\ythree) -- (1,\ythree);
\foreach \xt/\xlbl in {0/0, 0.16666666/{\tfrac{1}{6}}, 0.83333333/{\tfrac{5}{6}}, 1/1} {
  \draw (\xt,\ythree) -- (\xt,\ythree-0.07);
  \node[below] at (\xt,\ythree-0.07) {$\xlbl$};
}
\node[left] at (-0.02,\ythree) {$\dim = 3$}; % Segment tag
% Polynomials for s=1 (representative odd case)
\node[above=2pt] at (0.0833,\ythree) {$-\tfrac{2}{3}x^{3} + {\color{blue}\tfrac{77}{54}}x$};
\node[above=2pt] at (0.5,\ythree)    {$\tfrac{1}{2}x^{4} - x^{3} + \tfrac{1}{12}x^{2} + \tfrac{17}{12}x + \tfrac{1}{2592}$};
\node[above=2pt] at (0.9167,\ythree) {$\tfrac{2}{3}x^{3} - 2x^{2} + \tfrac{139}{54}x - \tfrac{13}{54}$};

% \foreach \xt in {0, 0.16666666, 0.3333333, 0.66666666, 0.83333333, 1} {
%   \draw[dotted,gray] (\xt,\ythree-0.07) -- (\xt,\yone-0.07);
% }

\end{tikzpicture}
\caption{The limit $h$-functions of the $A_2$ singularities ($n=3$), for $\dim = 1,2,3$, see Theorem~\ref{th:limit_h_function_A_n}. The coefficients in degree 1 of the first piece of the limit $h$-function (in {\color{blue}blue}) are the limit Hilbert–Kunz multiplicities of the corresponding singularities, due to \ref{th:F_invariants_lim_diff_h}.}
\label{fig:limit_h_functions_A2_example}
\end{figure}
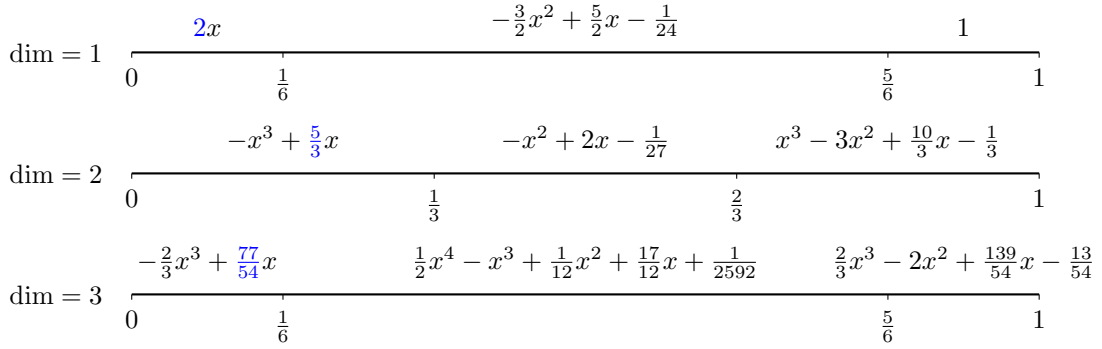

\subsection{The \texorpdfstring{$\limeHK$}{limeHK} of the \texorpdfstring{$A_n$}{An} singularities.}
In the next result, we compute the generating function $\Phi(x,z)$ of the limit $h$-functions, and then we will use those formulas to compute the generating function of the limit Hilbert–Kunz multiplicities in \ref{cor:limit_eHK_A_n_singularities}.

\begin{notation}[Generating function of limit $h$-functions]
    \label{not:generating_function_An_singularities}
Let $n\geq 2$. The sequence of functions $\lbrace h_{x_0^n+x_1^2+\dots+x_d^2,\infty}(x),d\geq 1\rbrace$ is uniformly bounded, and therefore, for every $0<z<1$, the following function is well-defined for every $x\in \mathbb{R}$:
    \begin{gather*}
        \Phi_n(x,z) := \sum_{d=0}^\infty h_{x_0^n+x_1^2+\dots+x_d^2,\infty}(x)z^d.
    \end{gather*}
    In this section, if $n$ is fixed, we will usually just write $\Phi$.
\end{notation}

The following lemma will be useful to compute the generating function of the $A_n$ singularities.
\begin{lemma}[Reflection of the generating function]
    \label{lm:reflection_generating_function_An}
    For $0<z\ll 1$ sufficiently small, $n\geq 2$, and $x\in [0,1]$,
    \begin{gather*}
        \Phi_n(1-x,z) = \begin{cases}
            \Phi_n(x,z) - \dfrac{(2x-1)z}{1-z} - nx+1 & 0\leq x\leq \dfrac{1}{n},\\
            \Phi_n(x,z) - \dfrac{(2x-1)z}{1-z} & \dfrac{1}{n}\leq x \leq 1-\dfrac{1}{n},\\
            \Phi_n(x,z) - \dfrac{(2x-1)z}{1-z} + n(1-x) -1 & 1-\dfrac{1}{n}\leq x \leq 1.
        \end{cases}
    \end{gather*}
\end{lemma}
\begin{proof}
Let $\tilde{\Phi}(x,z) := \sum_{d=1}^\infty h_{x_0^n+x_1^2+\dots+x_d^2,\infty}(x)z^d$. Corollary \ref{cor:reflection_h_function_double_points} implies that $\tilde{\Phi}(1-x,z) = \tilde{\Phi}(x,z) - \frac{(2x-1)z}{1-z}$. Hence,
\begin{equation*}
\Phi(1-x,z) = h_{x_0^n,\infty}(1-x)+\tilde{\Phi}(1-x,z) = h_{x_0^n,\infty}(1-x) + \tilde{\Phi}(x,z) - \dfrac{(2x-1)z}{1-z},
\end{equation*}
so the result follows from the formula for $h_{x_0^n,\infty}(x)$ (see \cite[Proposition 6.5]{mengLimits2026}).
\end{proof}

\begin{lemma}[Generating function of the limit $h$-functions of the $A_n$ singularities]
        \label{lm:gen_limit_h_function_A_n}
        Set $\Phi(x,z) = \sum_{d=0}^{\infty} h_d(x)z^d.$ Then, for $n\geq 4$,
		\begin{gather*}
        \Phi_n(x,z) = \begin{cases}
        \dfrac{xz}{1-z} + \dfrac{n}{2}\dfrac{\sin\frac{2z}{n}+\cos(z-\frac{2z}{n})}{\cos z}\dfrac{\sin 2zx}{2z} + \dfrac{nx}{2}  & x\in \left[0,\dfrac{1}{n}\right]\\[10pt]
        \dfrac{xz}{1-z} + \dfrac{n}{2}\dfrac{\sin\frac{2z}{n}}{\cos z}\dfrac{\cos(z-2zx)+\sin 2zx}{2z} +\dfrac{1}{2} & x\in \left[\dfrac{1}{n},\dfrac{1}{2}-\dfrac{1}{n}\right]\\[10pt]
        \dfrac{xz}{1-z} + \dfrac{n}{2}\dfrac{\sin\frac{2z}{n}+\cos\left(z-\frac{2z}{n}\right)}{\cos z}\dfrac{\cos(z-2zx)}{2z}- \dfrac{n}{4z} +\dfrac{1}{2} & x\in \left[\dfrac{1}{2}-\dfrac{1}{n},\dfrac{1}{2}\right]\\[10pt]
        \Phi_n(1-x,z)+(2x-1)\dfrac{z}{1-z} & x\in \left[\dfrac{1}{2},1-\dfrac{1}{n}\right]\\[10pt]
        \Phi_n(1-x,z)+(2x-1)\dfrac{z}{1-z}+1-n(1-x) & x\in \left[1-\dfrac{1}{n},1\right]\\[10pt]
        \end{cases}
        \end{gather*}
        for $n=3$,
		\begin{gather*}
            \Phi_3(x,z) = \begin{cases}
                \dfrac{xz}{1-z} + \dfrac{3}{2}\dfrac{\cos(z-\frac{2z}{3}) + \sin\frac{2z}{3}}{\cos z}\dfrac{\sin 2xz}{2z} + \dfrac{3}{2}x  & x\in \left[0,\dfrac{1}{6}\right],\\[10pt]
                \dfrac{xz}{1-z} + \dfrac{3}{2}\dfrac{\cos\frac{z}{3}}{\cos z}\dfrac{\sin 2xz + \cos(z-2xz)}{2z}-\dfrac{3}{4z} + \dfrac{3x}{2} & x\in \left[\dfrac{1}{6},\dfrac{1}{3}\right],\\[10pt]
                \dfrac{xz}{1-z} + \dfrac{3}{2}\dfrac{\cos\left(z-\frac{2z}{3}\right)+\sin\frac{2z}{3}}{\cos z}\dfrac{\cos(z-2zx)}{2z} - \dfrac{3}{4z} + \dfrac{1}{2} & x\in \left[\dfrac{1}{3},\dfrac{1}{2}\right],\\[10pt]
                \Phi_3(1-x,z)+(2x-1)\dfrac{z}{1-z} & x\in \left[\dfrac{1}{2},\dfrac{2}{3}\right]\\[10pt]
                \Phi_3(1-x,z)+(2x-1)\dfrac{z}{1-z}+1-3(1-x) & x\in \left[\dfrac{2}{3},1\right]\\[10pt]
            \end{cases}
        \end{gather*}
        and for $n=2$,
        \begin{gather*}
            \Phi_2(x,z) = \begin{cases}
                \dfrac{x}{1-z} + T(z)\dfrac{\sin 2zx}{2z} + \dfrac{\cos 2zx-1}{2z} & x\in [0,\frac{1}{2}],\\[10pt]
                \dfrac{x}{1-z} + T(z)\dfrac{\sin 2z(1-x)}{2z} + \dfrac{\cos 2z(1-x)-1}{2z} & x\in [\frac{1}{2},1].
            \end{cases}
        \end{gather*}
    \end{lemma}
    \begin{proof}
		Let us fix $n\geq 2$, and use the notation $\phi_d := h_{x_0^n+x_1^2+\dots+x_d^2}$. We will study the cases when $d$ is odd and even separatedly. Note that by \ref{lm:reflection_generating_function_An}, it is enough studying the functions at $[0,\frac{1}{2}]$. Define 
        \begin{equation*}
        \Phi^+(x,z):= \sum_{m=1}^{\infty} \phi_{2m-1}(x)z^{2m-1}, \quad \text{and} \quad \Phi^-(x,z) := \sum_{d=2m,m\geq 0}^{\infty} \phi_d(x)z^d
        \end{equation*}
        and set
    \begin{align*}
            P(x,2z) &:= \sum_{m = 1}^{\infty} (-1)^mE_{2m}(x) \dfrac{(2z)^{2m}}{(2m)!} = \dfrac{\alpha(x,2iz)+\alpha(x,-2iz)}{2}-1 = \frac{\cos ((2x-1)z)}{\cos z}-1,\\
            Q(x,2z) &:= \sum_{m=0}^{\infty} (-1)^mE_{2m+1}(x) \dfrac{(2z)^{2m+1}}{(2m+1)!} = \dfrac{\alpha(x,2iz)-\alpha(x,-2iz)}{2i} = \frac{\sin((2x-1)z)}{\cos z}.
    \end{align*}
        
        Let $d = 2m-1$, $m\geq 1$, and let $x\in [0,\frac{1}{2}-\frac{1}{n}]$. In this case,
        \begin{equation*}
        \phi_{2m-1}(x) = x + (-1)^m\dfrac{2^{2m-3}n}{(2m)!}\left(E_{2m}(x+\dfrac{1}{2}-\dfrac{1}{n})-E_{2m}(-x+\dfrac{1}{2}-\dfrac{1}{n})\right).
         \end{equation*}
    Then, one can check that
    \begin{align}
        \label{fm:psi_plus_first_interval}
    \nonumber\Phi^+(x,z) &= \dfrac{xz}{1-z^2} + \dfrac{n}{8z}\left(P\left(x+\dfrac{1}{2}-\dfrac{1}{n},2z\right)-P\left(-x+\dfrac{1}{2}-\dfrac{1}{n},2z\right)\right) \\
    \nonumber&= \dfrac{xz}{1-z^2} + \dfrac{n}{8z}\frac{\cos ((2x-\frac{2}{n})z)-\cos ((2x+\frac{2}{n})z)}{\cos z} \\
    &= \dfrac{xz}{1-z^2}+ \dfrac{n}{2}\dfrac{\sin\frac{2z}{n}}{\cos z}\dfrac{\sin 2zx}{2z}.
    \end{align}

    Now, in the interval $[\frac{1}{2}-\frac{1}{n},\frac{1}{2}]$, we have that
    \begin{equation*}
        \phi_d(x) = x + (-1)^m\dfrac{2^{2m-3}n}{(2m)!}\left(E_{2m}(x+\dfrac{1}{2}-\dfrac{1}{n})+E_{2m}(-x+1+\dfrac{1}{2}-\dfrac{1}{n})\right),
    \end{equation*}
    which yields that
    \begin{align}
        \label{fm:psi_plus_second_interval}
        \nonumber \Phi^{+}(x,z) &= \dfrac{xz}{1-z^2} + \dfrac{n}{8z}\left(P\left(x+\dfrac{1}{2}-\dfrac{1}{n},2z\right)+P\left(-x+1+\dfrac{1}{2}-\dfrac{1}{n},2z\right)-2\right)\\
        &= \dfrac{xz}{1-z^2}+ \dfrac{n}{2}\dfrac{\cos\left(z-2zx\right)}{\cos z}\dfrac{\cos\left(z-\frac{2z}{n}\right)}{2z}-\frac{n}{4z}
    \end{align}
    
    We follow a similar analysis for even $d$ on the intervals $[0,\frac{1}{n}]$ and $[\frac{1}{n},\frac{1}{2}]$: let $d = 2m, m\geq 1$, so
    \begin{gather*}
    \phi_d(x) = \begin{cases}
    x+ (-1)^m\dfrac{n}{4}\dfrac{2^{2m}}{(2m+1)!}\left(E_{2m+1}\left(x+\dfrac{1}{n}\right)-E_{2m+1}\left(-x+\dfrac{1}{n}\right)\right) & x\in \left[0,\dfrac{1}{n}\right]\\[10pt]
    x+ (-1)^m\dfrac{n}{4}\dfrac{2^{2m}}{(2m+1)!}\left(E_{2m+1}\left(x+\dfrac{1}{n}\right)+E_{2m+1}\left(-x+1+\dfrac{1}{n}\right)\right) & x\in \left[\dfrac{1}{n},\dfrac{1}{2}\right].
    \end{cases}
    \end{gather*}
    Hence,
    \begin{gather*}
    \Phi^-(x,z) = \dfrac{x}{1-z^2} + \begin{cases}\dfrac{n}{8z}\left(Q\left(x+\frac{1}{n},2z\right)-Q\left(-x+\frac{1}{n},2z\right)\right)& x\in \left[0,\dfrac{1}{n}\right]\\[10pt]
        \dfrac{n}{8z}\left(Q\left(x+\frac{1}{n},2z\right)+Q\left(-x+1+\frac{1}{n},2z\right)\right) & x\in \left[\dfrac{1}{n},\dfrac{1}{2}\right].\end{cases}
    \end{gather*}
    which in turn can be written as
    \begin{gather}
        \label{fm:psi_minus}
    \Phi^-(x,z) = \dfrac{x}{1-z^2} + \begin{cases}
    \dfrac{nx}{2}+ \dfrac{n}{2}\dfrac{\cos\left(z-\frac{2z}{n}\right)}{\cos z}\dfrac{\sin 2zx}{2z} & x\in \left[0,\dfrac{1}{n}\right]\\[10pt]
    \dfrac{1}{2}+\dfrac{n}{2}\dfrac{\sin\frac{2z}{n}}{\cos z}\dfrac{\cos(z- 2zx)}{2z} & x\in \left[\dfrac{1}{n},\dfrac{1}{2}\right].
    \end{cases}
    \end{gather}
    The result follows for $\Phi(x,z) = \Phi^+(x,z)+\Phi^-(x,z) + \phi_0(x)$ from putting together the expressions in \ref{fm:psi_plus_first_interval}, \ref{fm:psi_plus_second_interval}, \ref{fm:psi_minus}, recalling that $h_0(x) = nx$ for $x\in [0,\frac{1}{n}]$ and $1$ for $x\in [\frac{1}{n},1]$.
    \end{proof}

	\begin{corollary}[Limit Hilbert–Kunz multiplicities of the $A_n$ singularities]
	\label{cor:limit_eHK_A_n_singularities}
        Let $d,n\geq 1$, and
	\begin{gather*}
	R_d := \frac{\mathbb{C}[[x_0,\dots,x_d]]}{(x_0^{n+1}+x_1^2+\dots+x_d^2)}.
	\end{gather*}
	Then,
	\begin{equation*}
        \limeHK(R_{d}) = 1+\dfrac{c_d}{d!}, \quad \text{and} \quad \lims(R_{d}) = 1-\dfrac{c_d}{d!}
    \end{equation*}
	where $c_d/d!$ is the $d$-th Taylor coefficient of the function
	\begin{gather*}
	\frac{n+1}{2}\frac{\cos\frac{(n-1)z}{n+1} + \sin \frac{2z}{n+1}}{\cos z}% + \frac{n-1}{2}
	\end{gather*}
	See Table \ref{tab:lHK-A} below.
	\end{corollary}
	\begin{proof}
    Theorem \ref{th:F_invariants_lim_diff_h} and \ref{lm:gen_limit_h_function_A_n} yield
	\begin{align*}
	\sum_{d\geq 0} \limeHK(R_d)z^d &= \lim_{x\rightarrow 0^+} \dfrac{\partial}{\partial x} \left(\dfrac{xz}{1-z} + \dfrac{n+1}{2}x+ \dfrac{n+1}{2}\dfrac{\cos\frac{(n-1)z}{n+1}+\sin\frac{2z}{n+1}}{\cos z}\dfrac{\sin2zx}{2z}\right)\\
	&= \lim_{x\rightarrow 0^+}\left(\dfrac{z}{1-z}+\dfrac{n+1}{2} + \dfrac{n+1}{2}\dfrac{\cos\frac{(n-1)z}{n+1}+\sin\frac{2z}{n+1}}{\cos z}\cos2zx\right) \\
	&= \dfrac{z}{1-z}+\dfrac{n+1}{2} + \dfrac{n+1}{2}\dfrac{\cos\frac{(n-1)z}{n+1}+\sin\frac{2z}{n+1}}{\cos z}.
	\end{align*}
    For $d\geq 1$, $e(R_d) = 2$, and therefore $\lims(R_d) = 2-\limeHK(R_d)$.
	\end{proof}
\begin{remark}
In terms of Euler polynomials, it follows directly from \ref{th:limit_h_function_A_n} and \ref{th:F_invariants_lim_diff_h} that for $R$ as in \ref{cor:limit_eHK_A_n_singularities}, we have $\limeHK(R) = 1+\frac{c_d}{d!}$ where
    \begin{gather*}
    c_d = 2^{d}n\left|E_{d+1}\left(\frac{1-(-1)^d}{4}+\frac{(-1)^d}{n}\right)\right|,
    \end{gather*}
    cf. \cite[Theorem 6.2]{hoffmanDerivativePolynomialsEuler1999}.
\end{remark}
\begin{table}[h]
\centering
\renewcommand{\arraystretch}{2.5}
\begin{tabular}{r@{\hspace{1em}}l@{\hspace{1em}}l|c}
\toprule
 $\limeHK(A_n)$ & $\underset{n\rightarrow\infty}\longrightarrow$ & $\limeHK(A_\infty)$ & \\
\midrule

$2-\dfrac{1}{n}$
  & $\longrightarrow$
  & $2$ & $\dim 2$\\
\midrule

$\dfrac{3}{2}-\dfrac{2}{3n^2}$
  & $\longrightarrow$
  & $\dfrac{3}{2}$ & $\dim 3$\\
\midrule

$\dfrac{4}{3}-\dfrac{2}{3n^2}+\dfrac{1}{3n^3}$
  & $\longrightarrow$
  & $\dfrac{4}{3}$ & $\dim 4$\\
\midrule

$\dfrac{29}{24}-\dfrac{1}{3n^2}+\dfrac{2}{15n^4}$
  & $\longrightarrow$
  & $\dfrac{29}{24}$ & $\dim 5$\\
\midrule

$\dfrac{17}{15}-\dfrac{2}{9n^2}+\dfrac{2}{15n^4}-\dfrac{2}{45n^5}$
  & $\longrightarrow$
  & $\dfrac{17}{15}$ & $\dim 6$ \\
\bottomrule
\end{tabular}
\caption{Limit Hilbert--Kunz multiplicity of the $A_n$
and $A_\infty$ singularities in dimensions $2$ through $6$, by letting $n\to\infty$ by $\m$-adic continuity of the Hilbert–Kunz multiplicity (\cite{polstraContinuity2020}).}
\label{tab:lHK-A}
\end{table}

A remarkable consequence of \ref{cor:limit_eHK_A_n_singularities} is that, just as the case of $A_1$ singularity and the so-called \say{zig-zag} numbers, the limit Hilbert–Kunz multiplicities of all the $A_n$ singularities can also be interpreted combinatorially. We recall here a definition, were we use the notation in \cite{hoffmanDerivativePolynomialsEuler1999}:
\begin{definition}[Augmented $r$-signed permutations, \cite{ehrenborg1995,hoffmanDerivativePolynomialsEuler1999}]
    \label{def:augmented_r_signed_permutations}
Let $p\geq r$ be non-negative integers, and let $\zeta = e^{\frac{2\pi}{r}i}$. We define the set $S := \lbrace a\zeta^k:a,k \text{ nonnegative integers}\rbrace$, and we define the linear order $\Lambda(p,r)$, which we denote by $\prec$, according to
\begin{gather*}
\zeta^{r-1}\prec 2\zeta^{r-1}\prec 3\zeta^{r-1}\prec ...\prec \zeta^{r-2}\prec 2\zeta^{r-2}\prec 3\zeta^{r-2}\prec ...\\
...\prec \zeta^{p}\prec 2\zeta^{p}\prec 3\zeta^{p} \prec  0\prec \zeta^{p-1}\prec 2\zeta^{p-1}\prec 3\zeta^{p-1} \prec ...\prec 1\prec 2\prec 3\prec ...
\end{gather*}
An augmented $r$-signed permutation is thus a list of elements $(a_1\zeta^{k_1},\dots,a_n\zeta^{k_n})$ such that the $a_i$ form a permutation of $\lbrace 1,\dots,n\rbrace$. Among these sequences, we call $\Lambda$-alternating augmented $r$-signed permutations to those permutations $(x_1,\dots,x_n)\in S^n$ where $0\prec x_1\succ x_2\prec \dots x_n$ (i.e., $\Lambda$-alternating $n$-permutations whose first element is $\Lambda$-positive). The number of $\Lambda(p,r)$-alternating $r$-signed $n$-permutations is denoted by $N_n(p,r)$.
\end{definition}

\begin{proposition}[Proposition 7.2, \cite{ehrenborg1995}]
    \label{pr:augmented_r_signed_permutations_generating_function}
The (exponential) generating function of the numbers $N_d := N_d(p,r)$ of alternating augmented $(p,r)$-signed $d$-permutations is
\begin{gather*}
\sum_{d\geq 0} N_d\frac{z^d}{d!}=\dfrac{\cos (r-p)z + \sin pz}{\cos rz}.
\end{gather*}
\end{proposition}
\begin{corollary}%Checked with sage!
    \label{cor:combinatorial_interpretation}
    Under the notation in \ref{cor:limit_eHK_A_n_singularities}, the limit Hilbert–Kunz multiplicity of $x_0^{n+1}+x_1^2+\dots+x_d^2$ is
    \begin{gather*}
    1 + \frac{n+1}{2}\dfrac{N_d(2,n+1)}{(n+1)^d d!}.
    \end{gather*}
\end{corollary}
\begin{remark}
One can deduce from \ref{pr:augmented_r_signed_permutations_generating_function} that $N_d(2,2) = 2^dN_d(1,1)$. The sequence $N_d(1,1)$ are exactly the Euler's zig-zag numbers (see \cite[Theorem 1]{arnoldNormalFormsFunctions1972}, for instance), so Corollary \ref{cor:combinatorial_interpretation} indeed expands Gessel and Monsky's combinatorial interpretation of the numerators of the limit Hilbert–Kunz multiplicity of the $A_1$ singularities in Theorem \ref{th:gessel_monsky}.
\end{remark}

\section{Limit \texorpdfstring{$F$}{F}-invariants of the simple singularities.}
\label{sec:f_invariants_simple}
    In this section, we compute the limit $F$-invariants of the rest of the simple singularities. Computing all limit $h$-functions as we did in the case of the $A_n$ singularities is too involved, but it is not necessary to obtain the limit $F$-invariants. The main idea to have in mind is that the Hilbert–Kunz multiplicity only depends on the behavior of the $h$-function near $0$.

\subsection{The limit \texorpdfstring{$h$}{h}-function of the ADE surface singularities.} First, we will compute the entire limit $h$-functions of the two-dimensional simple singularities whose $h$-functions are not known, that is, the $D_n$ and $E_7$ singularities, the non-diagonal ones. The rest of the ADE surface singularities are diagonal, and therefore can be computed using \cite[Theorem 3.3]{shidelerFsignatureFunctionsDiagonal2024} and \ref{def:phi_function}. The approach to compute the limit $h$-functions of the ADE surface singularities was sketched in a previous version of \cite{mengLimits2026}.

The key observation is that the $D_n$ and $E_7$ singularities are both \say{suspensions} of a binomial. Thus, the key tool is the following result by Meng on the $h$-function of a binomial; see also Theorem A in \cite{brosowskyLimitFsignatureFunctions2025}. This result utilizes multivariate $h$-functions, see Definition \ref{def:multivariate_h_function}.
\begin{theorem}[\cite{mengLimits2026}, Section 6]
    \label{th: limit_h_binomials}
    Let $R = \mathbb{C}[[x_1,...,x_s,y_1,...,y_t]]$ with maximal ideal $\m$, and set 
    \begin{gather*}
    F := x_1^{a_1}\dots x_{s}^{a_s}y_1^{b_1}\dots y_t^{a_t}(x_1^{c_1}\dots x_s^{c_s}+y_1^{d_1}\dots y_t^{d_t}).
    \end{gather*}
    Let $f = x_1^{c_1}\dots x_s^{c_s}$ and $g = y_1^{d_1}\dots y_t^{d_t}.$ Then, for $0\leq r\leq \min\lbrace\frac{1}{a_i},\frac{1}{b_j}\rbrace,$ we have
    \begin{multline*}
    h_{F,\infty}(r) = 1-\prod (1-a_ir)\prod(1-b_ir)+\\
    +\int_{[0,\infty)^2}D_{T_1+T_2,\infty}(r_1,r_2,r)d(-h'_{k[\underline{x}],0,(\underline{x},f)}(\mathbf{1}-r\mathbf{a}))d(-h'_{k[\underline{y}],0,(\underline{y},g)}(\mathbf{1}-r\mathbf{b})).
    \end{multline*}
    In particular, if $F = x^ay^b(x^n+y^m),$ then
    \begin{multline*}
    h_{F,\infty}(r) = 1-(1-ar)(1-br)+\\
    +\int_{[0,\infty)^2}D_{T_1+T_2,\infty}(r_1,r_2,r)d(-h'_{k[x],0,(x,x^n)}(1-ra,t_1))d(-h'_{k[y],0,(y,y^m)}(1-rb,t_2)),
    \end{multline*}
    where
    \begin{gather*}
    h_{k[x],0,(x,x^n)}(1-ra,r_1) = \begin{cases}
    0 & r_1\leq 0,\\
    nr_1 & 0\leq r_1 \leq \frac{1-ra}{n},\\
    1-ra & \frac{1-ra}{n} \leq r_1,
    \end{cases}
    \end{gather*}
    and $h_{k[x],0,(x,x^n)}''(1-ra,r_1) = n\left(\delta_{\frac{1-ra}{n}}(r_1)-\delta_0(r_1)\right),$ where by $h''$ we mean $\partial h/\partial r^2$.
    \end{theorem}

Thus, with the help of this result, we first compute the limit $h$-function of the $D_n$ plane curve $y^2z+z^{n-1}$:
\begin{proposition}[The limit $h$-functions of the $D_n$ plane curves]
    \label{cor:h_function_D_n_curve}
% Let $y^2z+z^{n-1}\in k[[y,z]]$. Then,
\begin{gather*}
h_{\mathbb{C}[[y,z]],(y,z),y^2z+z^{n-1},\infty}(x)= \begin{cases} 3x - 2x^2 & 0 \le x \le \dfrac{n-4}{2(n-3)} \\[10pt] 1 - \dfrac{(n-1)^2}{2(n-2)}\left(x - \dfrac{n}{2(n-1)}\right)^{2} & \dfrac{n-4}{2(n-3)} \le x \le \dfrac{n}{2(n-1)} \\[10pt] 1 & \dfrac{n}{2(n-1)} \le x \le 1 \end{cases}
\end{gather*}
\end{proposition}
\begin{proof}
Theorem \ref{th: limit_h_binomials} yields
\begin{align*}
h_{y^2z+z^{n-1},\infty}(x) &= 1-(1-x)+\\
&+\int_{[0,\infty)^2}D_{\infty}(t_1,t_2,x)d(-h'_{k[y],0,(y,y^2)}(1,t_1))d(-h'_{k[z],0,(z,z^{n-2})}(1-x,t_2))\\
& = x+2(n-2)\int_0^\infty\int_0^\infty D_{\infty}(t_1,t_2,x)\left(\delta_{\frac{1}{2}}(t_1)-\delta_0(t_1)\right)\left(\delta_{\frac{1-x}{n-2}}(t_2)-\delta_0(t_2)\right)\\
&= x + 2(n-2)D_\infty\left(\frac{1-x}{n-2},\frac{1}{2},x\right) = x + 2(n-2)K\left(\frac{1-x}{n-2},x\right)
\end{align*}
\begin{figure}
\centering
\begin{tikzpicture}[scale=3, thick]

    \draw[->] (-0.25,0) -- (1.25,0) node[right]{$t = t_2$};
    \draw[->] (0,-0.25) -- (0,1.25) node[above]{$x = t_3 $};

    \draw (0,0) rectangle (1,1);

    \draw (0,0.5) -- (0.5,1) -- (1,0.5) -- (0.5,0) -- cycle;

    \draw (1,1) -- (1,1.25);
    \draw (1,1) -- (1.25,1);

    \filldraw (0,0) circle (1pt) node[below left]{$(0,0)$};
    \filldraw (1,0) circle (1pt) node[below right]{$(1,0)$};
    \filldraw (1,1) circle (1pt) node[above right]{$(1,1)$};
    \filldraw (0,1) circle (1pt) node[above left]{$(0,1)$};
    \filldraw (0.5,0) circle (1pt);
    \filldraw (1,0.5) circle (1pt);
    \filldraw (0.5,1) circle (1pt);
    \filldraw (0,0.5) circle (1pt);

    \draw[color=orange] (0,1) -- (1/2,0);
    \node[color=orange,rotate=-60] at (0.6,-0.2) {$n=4$};
    \filldraw[color=orange] (0.5,0) circle (.4pt);
    \filldraw[color=orange] (0.17,0.67) circle (.4pt);
    \draw[color=orange] (1/3,0) -- (0,1);
    \node[color=orange,rotate=-70] at (0.42,-0.2) {$n=5$};
    \filldraw[color=orange] (0.25,0.25) circle (.4pt);
    \filldraw[color=orange] (0.17,0.33) circle (.4pt);
    \draw[color=orange] (1/4,0) -- (0,1);
    \node[color=orange,rotate=-80] at (0.29,-0.2) {$n=6$};
    \filldraw[color=orange] (0.1,0.6) circle (.4pt);
\end{tikzpicture}
\caption{Segment $x\mapsto (\frac{1-x}{n-2},x)$ for $n=4,5$ and $6$ over the domain of $K_2$; proof of \ref{cor:h_function_D_n_curve}.}
\label{fig:segment_Dn_curve}
\end{figure}
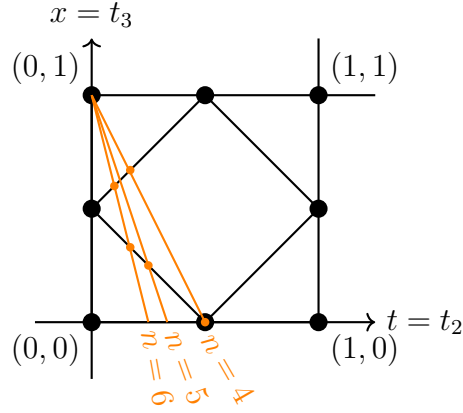
Using Definition \ref{def:K_function}, note that the points $\lbrace(\frac{1-x}{n-2},x)\rbrace_{0\leq x\leq 1}$ form a line segment going from $\left(\frac{1}{n-2},0\right)$ to $\left(0,1\right)$, thus contained within the union of regions $\Delta_0,\Delta_1$ and $\Delta_3$, and cutting with their boundaries at $x = \frac{n-4}{2(n-3)}$ and $\frac{n}{2(n-1)}$ (see Figure \ref{fig:segment_Dn_curve}; note that when $n=4$, the line segment intersects the boundaries of $\Delta_0$ when $x=0$ at $(1/2,0)$). The result follows from direct substitution on the definition of $K$.
\end{proof}

With this, we can now compute the limit $h$-function of the $D_n$ surface singularity.
\begin{theorem}[The limit $h$-functions of the $D_n$ surface singularities]
    \label{th:h_function_D_n_surface}
Let $f = x^2+y^2z+z^{n-1}\in \mathbb{C}[[x,y,z]]$. Then,
\begin{equation*}
h_{f,\infty}(x) = \begin{cases}
-\dfrac{(n-1)^2}{3(n-2)}x^3+\left(2-\dfrac{1}{4(n-2)}\right)x
& x \in \left[0, \dfrac{1}{2(n-1)}\right] \\[10pt]
\dfrac{1}{6(n-1)(n-2)}\left((1-n)x-\dfrac{1}{2}\right)^3+2x
& x \in \left[\dfrac{1}{2(n-1)}, \dfrac{1}{2(n-3)}\right] \\[10pt]
\dfrac{-2}{3}x^3 - \dfrac{1}{2}x^2 + 2x-\dfrac{1}{24(n-1)(n-3)}
& x \in \left[\dfrac{1}{2(n-3)}, \dfrac{1}{2}\right] \\[10pt]
h_{\infty,f}(1-x)+2x-1
& x \in \left[\dfrac{1}{2}, 1\right] \\[10pt]
\end{cases}
\end{equation*}
which, in particular, is continuous and $\mathcal{C}^2$ on $(0,1)$.
\end{theorem}
\begin{proof}
By the integral formula \ref{th:integral_formula}, we know that
\begin{gather*}
h_{f,\infty}(x)
= \int_{[0,\infty)^2} D_\infty(t_1, t_2, x) \, d(-h'_{k[x],(x),x^2}(t_1)) \, d(-h'_{k[y,z],(y,z),y^2z+z^{n-1}}(t_2)).
\end{gather*}
Using \ref{cor:h_function_D_n_curve}, the distribution $d(-h')$ is therefore
\begin{gather*}
d_x(-h'_{y^2z+z^{n-1},\infty}) = -3\delta_0 + 4\chi_{(0,\frac{n-4}{2(n-3)})}+\frac{(n-1)^2}{n-2}\chi_{(\frac{n-4}{2(n-3)},\frac{n}{2(n-1)})}.
\end{gather*}
See \ref{rmk:domain_of_integration_and_atoms} for discontinuous integrators and Dirac delta distributions. Thus, the limit $h$-function of the $D_n$ surface singularity is:
\begin{align}
    \label{eq:integral_D_n_surface}
h_{x^2+y^2z+z^{n-1}}(x) &= \int_{[0,\infty)^2} D_\infty(t_1,t_2,x)d(-h'_{x^2}(t_1))d(-h'_{y^2z+z^{n-1}}(t_2))\\
&=8\int_{0}^{\frac{n-4}{2(n-3)}} K(t,x)dt + \frac{2(n-1)^2}{n-2}\int_{\frac{n-4}{2(n-3)}}^{\frac{n}{2(n-1)}} K(t,x)dt.
\end{align}
This function splits into 6 pieces, given the definition of the function $K(t,x)$ (see Figure \ref{fig:D_n_surface_integral}). We only study the function for $x\leq \frac{1}{2}$ thanks to \ref{cor:reflection_h_function_double_points}.

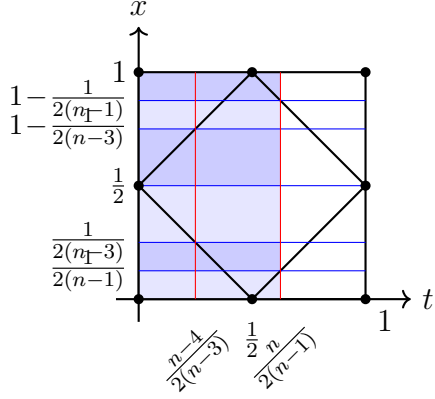
\begin{figure}
\begin{tikzpicture}[scale=3, thick]
    \draw[->] (-0.1,0) -- (1.2,0) node[right]{$t$};
\draw[->] (0,-0.1) -- (0,1.2) node[above]{$x$};

% Shade alternating vertical bands (was horizontal)
\fill[blue!10] (0,0) rectangle (0.625,0.125);
\fill[blue!20] (0,0.125) rectangle (0.625,0.25);
\fill[blue!10] (0,0.25) rectangle (0.625,0.5);
\fill[blue!20] (0,0.5) rectangle (0.625,0.75);
\fill[blue!10] (0,0.75) rectangle (0.625,0.875);
\fill[blue!20] (0,0.875) rectangle (0.625,1);
\draw (0,0) rectangle (1,1);
\draw[thick] (0.5,0) -- (1,0.5) -- (0.5,1) -- (0,0.5) -- cycle;

% Horizontal lines (was vertical)
\foreach \y in {0.125, 0.25, 0.5, 0.75, 0.875}{
\draw[blue, thin] (0,\y) -- (1,\y);
    }
% Vertical lines at projected t-values, including t=1/2 (was horizontal)
\foreach \t in {0.25, 0.625}{
\draw[red, thin] (\t,0) -- (\t,1);
    }
% Axis labels
\node[left] at (0,0.125) {\small $\frac{1}{2(n-1)}$};
\node[left] at (0,0.25) {\small $\frac{1}{2(n-3)}$};
\node[left] at (0,0.5) {\small $\frac{1}{2}$};
\node[left] at (0,0.75) {\small $1\!-\!\frac{1}{2(n-3)}$};
\node[left] at (0,0.875) {\small $1\!-\!\frac{1}{2(n-1)}$};
\node[below, rotate=45, anchor=north east] at (0.3,0) {\small $\frac{n-4}{2(n-3)}$};
\node[below] at (0.5,-0.05) {\small $\frac{1}{2}$};
\node[below, rotate=45, anchor=north east] at (0.7,0) {\small $\frac{n}{2(n-1)}$};
\filldraw (0.5,0) circle (0.5pt);
\filldraw (0,0.5) circle (0.5pt);
\filldraw (1,0.5) circle (0.5pt);
\filldraw (0.5,1) circle (0.5pt);
\filldraw (0,0) circle (0.5pt);
\filldraw (1,0) circle (0.5pt) node[below right]{$1$};
\filldraw (0,1) circle (0.5pt) node[left]{$1$};
\filldraw (1,1) circle (0.5pt);
\end{tikzpicture}
\caption{Regions of the integral \ref{eq:integral_D_n_surface}, in blue. The {\color{red}red} lines are the limits of integration, and the {\color{blue}blue} lines delimitate the values of $x$ where the expression of the integrand changes, as the limits of integration interact with the different regions of the domain of $K$ (delimited by the black lines).}
\label{fig:D_n_surface_integral}
\end{figure}

\begin{enumerate}
\item $x\in [0,\frac{1}{2(n-1)}]$
\begin{multline*}
8\int_{0}^{\frac{n-4}{2(n-3)}} \left.K_x(t)\right|_{\Delta_1} dt +\\
+ \frac{2(n-1)^2}{n-2}\left(\int_{\frac{n-4}{2(n-3)}}^{1/2-x}\left.K_x(t)\right|_{\Delta_1}dt + \int_{1/2-x}^{1/2+x}\left.K_x(t)\right|_{\Delta_0}dt+\int_{1/2+x}^{\frac{n}{2(n-1)}}\left.K_x(t)\right|_{\Delta_2}dt\right).
\end{multline*}
\item $x\in [\frac{1}{2(n-1)},\frac{1}{2(n-3)}]$
\begin{multline*}
8\left(\int_{0}^{\frac{1}{2}-x} \left.K_x(t)\right|_{\Delta_1} dt+\int_{\frac{1}{2}-x}^{\frac{n-4}{2(n-3)}} \left.K_x(t)\right|_{\Delta_0} dt\right) +\\
+ \frac{2(n-1)^2}{n-2}\left(\int_{\frac{n-4}{2(n-3)}}^{1/2+x}\left.K_x(t)\right|_{\Delta_0}dt + \int_{1/2+x}^{\frac{n}{2(n-1)}}\left.K_x(t)\right|_{\Delta_2}dt\right)
\end{multline*}
\item $x\in [\frac{1}{2(n-3)},\frac{1}{2}]$
\begin{gather*}
8\left(\int_{0}^{\frac{1}{2}-x} \left.K_x(t)\right|_{\Delta_1} dt+\int_{\frac{1}{2}-x}^{\frac{n-4}{2(n-3)}} \left.K_x(t)\right|_{\Delta_0} dt\right) + \frac{2(n-1)^2}{n-2}\int_{\frac{n-4}{2(n-3)}}^{\frac{n}{2(n-1)}}\left.K_x(t)\right|_{\Delta_0}dt
\end{gather*}

\end{enumerate}
The final result follows from directly evaluating the integrals.
\end{proof}
\begin{remark}
    \label{rmk:D_n_3_dimensional}
Repeating this process recursively (see \ref{lm:D_n_limit_h_function_kernel}), it is possible to compute these functions in further dimensions; although we do not include the computation here, we include here as an example that for $x\in [0,\frac{n-4}{2(n-3)}]$,
\begin{equation*}
h_{x^2+y^2+z^2t+t^{n-1}}(x) = \frac{1}{3}x^4-x^3+\left(\frac{19}{12}-\dfrac{1}{12(n-1)(n-3)}\right)x,
\end{equation*}
which implies by \ref{th:F_invariants_lim_diff_h} that the limit Hilbert–Kunz multiplicity of the 3-dimensional $D_n$ singularities is $\frac{19}{12}-\frac{1}{12(n-1)(n-3)}$, $n\geq 4$. See \ref{th: D_n_singularities_eHK_formula} or Table \ref{tab:lHK-D}.
\end{remark}

We now study the limit $h$-function of the $E_7$ surface singularity.
\begin{lemma}
    \label{lm:h_function_E_7_curve}
The limit $h$-function of the $E_7$ curve singularity is
\begin{gather*}
h_{y^3+yz^3}(x) = \begin{cases}
3x & 0\leq x \leq \frac{1}{9}\\
-\frac{27}{8}x^2+\frac{15}{4}x-\frac{1}{24} & \frac{1}{9}\leq x \leq \frac{5}{9}\\
1 & \frac{5}{9}\leq x.
\end{cases}
\end{gather*}
\end{lemma}
\begin{proof}
According to \ref{th: limit_h_binomials},
\begin{gather*}
    h_{y(y^2+z^3)}(x) = x+\int_{[0,\infty)^2} D_\infty(s,t,x)d_s\left(-h_{(z,z^3)}'(1,s)\right)d_t\left(-h'_{(y,y^2)}(1-x,t)\right)\\
    \intertext{where}
     h_{(y,y^n)}(x) = \begin{cases}
0 & x\leq 0,\\
nx & 0\leq x \leq \frac{1-x}{n},\\
1-xa & \frac{1-x}{n}\leq x\leq 1,
\end{cases}
\end{gather*}
which implies that
\begin{gather*}
h_{y(y^2+z^3)}(x) = x+6D_\infty\left(\frac{1}{3},\frac{1-x}{2},x\right) = x+6K_3\left(\frac{1-x}{2},x\right),
\end{gather*}
where we defined $K_n$ in \ref{def:K_function}. The regions of the function $K_3$ are already computed by Meng in \cite{mengLimits2026}, in Section 8. The result follows from a discussion of the regions of definition of $K_3$ similar to what we did in \ref{cor:h_function_D_n_curve}.
\end{proof}

\begin{theorem}[The limit $h$-function of the $E_7$ surface singularity.]
	\label{th:h_function_E_7_surface}
Let $f = x^2+y^3+yz^3\in \mathbb{C}[[x,y,z]]$. Then,
\begin{equation*}
h_{f,\infty}(x) = \begin{cases}
-\dfrac{9}{4}x^3 + \dfrac{95}{48}x
& x \in \left[0, \dfrac{1}{18}\right] \\[10pt]
-\dfrac{9}{8}x^3 - \dfrac{3}{16}x^2 + \dfrac{191}{96}x - \dfrac{1}{5184}
& x \in \left[\dfrac{1}{18}, \dfrac{7}{18}\right] \\[10pt]
-\dfrac{3}{2}x^2 + \dfrac{5}{2}x - \dfrac{43}{648}
& x \in \left[\dfrac{7}{18}, \dfrac{1}{2}\right] \\[10pt]
h_{\infty,f}(1-x)+2x-1
& x \in \left[\dfrac{1}{2}, 1\right] \\[10pt]
\end{cases}
\end{equation*}
which, in particular, is continuous and $\mathcal{C}^2$ on $(0,1)$.
\end{theorem}
\begin{proof}
Again by \ref{th:integral_formula}, we have that
\begin{gather*}
h_{f,\infty}(x)
= \int_{[0,\infty)^2} D_\infty(t_1, t_2, x) \, d(-h'_{k[x],(x),x^2}(t_1)) \, d(-h'_{k[y,z],(y,z),y^3+yz^3}(t_2)),
\end{gather*}
and \ref{lm:h_function_E_7_curve} yields that
\begin{gather*}
h'_{y^3+yz^3,\infty}(x)=\begin{cases} 
    3 & 0 \le x \le \frac{1}{9} \\
    -\frac{27}{4}x+\frac{15}{4} & \frac{1}{9} \le x \le \frac{5}{9} \\
    0 & \frac{5}{9} \le x \end{cases}
\end{gather*}
which is continuous everywhere except at $0$. The associated distribution is therefore
\begin{gather*}
d_x(-h'_{y^3+yz^3,\infty}) = -3\delta_0 + \frac{27}{4}\chi_{(\frac{1}{9},\frac{5}{9})},
\end{gather*}
which implies that
\begin{gather}
    \label{eq:integral_E_7_surface}
h_{x^2+y^3+yz^3,\infty}(x) = \int_{[0,\infty)^2} D_\infty(t_1,t_2,x)d(-h'_{x^2}(t_1))d(-h'_{y^3+yz^3}(t_2))=\frac{27}{2}\int_{\frac{1}{9}}^{\frac{5}{9}} K(t,x)dt.
\end{gather}

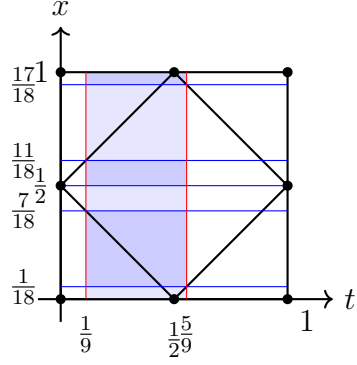
\begin{figure}
\begin{tikzpicture}[scale=3, thick]
    \draw[->] (-0.1,0) -- (1.2,0) node[right]{$t$};
\draw[->] (0,-0.1) -- (0,1.2) node[above]{$x$};

% Shade alternating vertical bands (was horizontal)
\fill[blue!10] (0.111,0) rectangle (0.555,0.055);
\fill[blue!20] (0.111,0.055) rectangle (0.555,0.389);
\fill[blue!10] (0.111,0.389) rectangle (0.555,0.5);
\fill[blue!20] (0.111,0.5) rectangle (0.555,0.611);
\fill[blue!10] (0.111,0.611) rectangle (0.555,0.945);
\fill[blue!20] (0.111,0.945) rectangle (0.555,1);
\draw (0,0) rectangle (1,1);
\draw[thick] (0.5,0) -- (1,0.5) -- (0.5,1) -- (0,0.5) -- cycle;

% Horizontal lines (was vertical)
\foreach \y in {0.055, 0.389, 0.5, 0.611, 0.945}{
\draw[blue, thin] (0,\y) -- (1,\y);
    }
% Vertical lines at projected t-values, including t=1/2 (was horizontal)
\foreach \t in {0.111, 0.555}{
\draw[red, thin] (\t,0) -- (\t,1);
    }
% Axis labels
\node[left] at (-0.05,0.055) {\small $\frac{1}{18}$};
\node[left] at (-0.05,0.389) {\small $\frac{7}{18}$};
\node[left] at (0,0.5) {\small $\frac{1}{2}$};
\node[left] at (-0.05,0.611) {\small $\frac{11}{18}$};
\node[left] at (-0.05,0.945) {\small $\frac{17}{18}$};
\node[below] at (0.111,-0.025) {\small $\frac{1}{9}$};
\node[below] at (0.5,-0.05) {\small $\frac{1}{2}$};
\node[below] at (0.555,-0.025) {\small $\frac{5}{9}$};
\filldraw (0.5,0) circle (0.5pt);
\filldraw (0,0.5) circle (0.5pt);
\filldraw (1,0.5) circle (0.5pt);
\filldraw (0.5,1) circle (0.5pt);
\filldraw (0,0) circle (0.5pt);
\filldraw (1,0) circle (0.5pt) node[below right]{$1$};
\filldraw (0,1) circle (0.5pt) node[left]{$1$};
\filldraw (1,1) circle (0.5pt);
\end{tikzpicture}
% (the drawing is done taking the $n=5$ case)
\caption{Region covered by the integral \ref{eq:integral_E_7_surface}, in blue. See Figure \ref{fig:D_n_surface_integral} for the color reference.}
\label{fig:E_7_surface_integral}
\end{figure}

Just as in the case of the $D_n$ singularity (see \ref{th:h_function_D_n_surface}), this function splits into 6 pieces according to the interaction of the domains of $K(t,x)$ and $h_{y^3+yz^3}$, but again, using the symmetry from \ref{cor:reflection_h_function_double_points}, we only need to work out the first half of the domain. The result thus follows from direct evaluation of the following formulas using the definition of $K$ in \ref{def:K_function}:
\begin{enumerate}
\item $x\in \left[0,\frac{1}{18}\right]$,
\begin{equation*}
\frac{27}{2}\int_{\frac{1}{9}}^{\frac{5}{9}} K(t,x)dt = \frac{27}{2}\int_{\frac{1}{9}}^{1/2-x} \left.K_x(t)\right|_{\Delta_1}dt + \frac{27}{2}\int_{1/2-x}^{1/2+x} \left.K_x(t)\right|_{\Delta_0}dt + \frac{27}{2}\int_{1/2+x}^{\frac{5}{9}} \left.K_x(t)\right|_{\Delta_2}dt,
\end{equation*}
\item $x\in \left[\frac{1}{18},\frac{7}{18}\right]$,
\begin{equation*}
\frac{27}{2}\int_{\frac{1}{9}}^{\frac{5}{9}} K(t,x)dt = \frac{27}{2}\int_{\frac{1}{9}}^{1/2-x} \left.K_x(t)\right|_{\Delta_1}dt + \frac{27}{2}\int_{1/2-x}^{\frac{5}{9}} \left.K_x(t)\right|_{\Delta_0}dt,
\end{equation*}
\item $x\in \left[\frac{7}{18},\frac{1}{2}\right]$,
\begin{equation*}
\frac{27}{2}\int_{\frac{1}{9}}^{\frac{5}{9}} K(t,x)dt = \frac{27}{2}\int_{\frac{1}{9}}^{\frac{5}{9}} \left.K_x(t)\right|_{\Delta_0}dt.
\end{equation*}
\end{enumerate}
\end{proof}
\begin{remark}
As we did in Remark \ref{rmk:D_n_3_dimensional}, we also observe here that doing one dimension further shows that for $x\in [0,\frac{1}{9}]$,
\begin{equation*}
    h_{x^2+y^2+z^3+zt^3}(x) = -x^3 + \frac{131}{81}x,
\end{equation*}
and so the 3-dimensional $E_7$ singularity has limit Hilbert–Kunz multiplicity $131/81$. See \ref{cor:limit_eHK_A_n_singularities} for the rest of the dimensions.
\end{remark}

\begin{lemma}

    \label{lm:differentiability_h_functions_A_n}

Let $f\in \mathbb{C}[[x_1,\dots,x_n]]$ be such that the limit $h$-function $h_{f}$ is a continuous polynomial on $\mathbb{R}$, and $\mathcal{C}^2$ everywhere except at $x=0$ and $1$. Then, the limit $h$-function of $f+y^2 \in \mathbb{C}[[x_1,\dots,x_n,y]]$ is also a continuous polynomial on $\mathbb{R}$ that is $\mathcal{C}^2$ everywhere except at $x=0,1$.
\end{lemma}
\begin{proof}
    
    We only need to check the claim on $(0,1)$, because $h$-functions of hypersurfaces are constant for $x<0$ and $x>1$. The hypotheses imply that 
    \begin{align*}
        d_t\left(-h'_{f}(t)\right) &= -h''_{f}(t)dt - \left.\left(h'_{f}(t)\right)\right|_{1^-}^{1^+}\delta_1(t)dt\\
        &= -h''_{f}(t)dt +h'_{f}(1^-)\delta_1(t)dt,
    \end{align*}
    where we account for the possible discontinuity at $t=1$ with a Dirac delta. Thus, applying the integral formula and the induction hypothesis,
\begin{align*}
    h_{f+y^2}(x) &= \int_{0}^{1^+} 2D_{\infty}\left(\frac{1}{2},t,x\right)d_t\left(-h'_{f}(t)\right)\\
    &= -\int_{0}^{1} 2D_{\infty}\left(\frac{1}{2},t,x\right)h''_{f}(t)dt + 2D_{\infty}\left(\frac{1}{2},1,x\right)h'_{f}(1^-)
\end{align*}
    where the integral in the second line is a Riemann integral. Note that since $D_{\infty}\left(\frac{1}{2},t,x\right)$ is $\mathcal{C}^2$ in $(t,x)\in (0,1)^2$ and $h''_{f}(t)$ is continuous and does not depend on $x$, we conclude that $\partial_x^2(D_{\infty}\left(\frac{1}{2},t,x\right)h''_{f}(t))$ is continuous, so Leibniz's rule of integration under integral sign applies. Also, $D_{\infty}\left(\frac{1}{2},1,x\right)$ is $\mathcal{C}^2$ along $x$ on $(0,1)$. We conclude that the second derivative of $h_{f}(x)$ is well-defined and continuous on $(0,1)$. %by leibniz rule of differentiation under the integral sign
\end{proof}
\begin{corollary}
For $d\geq 2$, the limit $h$-function of the ADE singularities are continuous piecewise polynomials in $\mathbb{R}$ and $\mathcal{C}^2$ everywhere except at $x=0,1$.
\end{corollary}
\begin{remark}
    \label{rmk:bis_differentiability_h_functions_A_n}
The first derivative of the $h$-function of a hypersurface is never continuous at $x=0$ because it coincides with the Hilbert–Kunz multiplicity from the right, and it is zero from the left, by Theorem.\ref{th:F_invariants_lim_diff_h}. The same theorem yields that the $h$-function is $\mathcal{C}^1$ at $x=1$ if and only if the hypersurface is strongly $F$-regular.
\end{remark}

\subsection{The \texorpdfstring{$\limeHK$}{limit eHK} of \texorpdfstring{$D_4,E_6$ and $E_8$}{D4, E6 and E8} and other of toric-type singularities.}
\label{subsec:torics}
Let us first fix the notation for the rest of the section. 

\begin{notation}
    \label{nt:notation_phi_psi_toric}
Let $m, n\geq 2, d\geq 1$, and $\psi_{n,d}(x) := h_{x_0^n+x_1^2+\dots+x_d^2}(x)$, where $\psi_{0}(x) := h_{x_0^n}(x)$. Then, define $\phi_0(x) := h_{x_0^m}(x)$, and for $d\geq 1$, $\phi_{n,m,d}(x) := h_{x_0^m+x_1^n+x_2^2+\dots+x_d^2}(x)$. Finally, we denote the corresponding generating functions by $\Psi_n(x,z) := \sum \psi_{n,d}(x)z^d$ and $\Phi_{n,m}(x,z) := \sum \phi_{n,m,d}(x)z^d$.
\end{notation}

\begin{lemma}
    \label{lm:torics_limit_h_function_small_x}
	For all $m, n\geq 2$ and $0<x$ small enough,
    \begin{gather*}
\Phi_{n,m}(x,z) = \frac{nz}{2}\int_{\frac{1}{n}-x}^{\frac{1}{n}+x} \Psi_m(t,z)dt + h_{x_0^n}(x)
    \end{gather*}
\end{lemma}
In particular,
\begin{gather*}
\sum \limeHK(R_{n,m,d})z^d = n + nz\Psi_m(\dfrac{1}{n},z)
\end{gather*}
\begin{proof}
    By the integral formula \ref{th:integral_formula},
\begin{gather}
	\label{eq:recursion_toric}
\phi_{n,m,d}(x) = \int mK_m(t,x)d_t\left(-h'_{x_1^n+x_2^2+\dots+x_d^2}(t)\right) = \int mK_m(t,x)d_t\left(-\psi'_{n,d-1}(t)\right),
\end{gather}
where $K_m$ is as in \ref{def:K_function}. This implies that
\begin{gather*}
\Phi_{n,m}(x,z) := \sum \phi_d(x)z^d = \int_{[0,\infty)} mzK_m(t,x)d_t\left(-\Psi_n'(t,z)\right) + \phi_0(x)
\end{gather*}
and we can simply apply integration by parts (\cite[Corollary 2.26]{mengLimits2026}), together with the fact that all the boundary conditions vanish, in order to get
\begin{gather*}
\Phi_{n,m}(x,z) = -mz\int_{[0,\infty)} \Psi_n(t,z)\frac{\partial^2 K_m(t,x)}{\partial t^2}dt + \phi_0(x).
\end{gather*}
Note that, near $x=0$, it happens that $\frac{\partial^2 K_m(t,x)}{\partial t^2} = -\frac{1}{2}$, according to Remark \ref{rmk:derivative_K_function}. Thus, for sufficiently small $x>0$, we have that
\begin{gather*}
\Phi_{n,m}(x,z) = \frac{mz}{2}\int_{\frac{1}{m}-x}^{\frac{1}{m}+x} \Psi_n(t,z)dt + mx
\end{gather*}

Now, by \ref{th:F_invariants_lim_diff_h}, we have that
\begin{equation*}
\sum \limeHK(R_{n,m,d})z^d = \lim_{x\rightarrow 0^+}\sum \phi_{n,m,d}'(x)z^d = \lim_{x\rightarrow 0^+}\partial_x \Phi_{n,m}.
\end{equation*}
Hence, 
\begin{equation*}
\partial_x \Phi_{n,m}(x,z) = \frac{nz}{2}\left(\Psi_m\left(\dfrac{1}{n}+x,z\right)+\Psi_m\left(\dfrac{1}{n}-x,z\right)\right) + h'_{x_0^n}(x).
\end{equation*}
Since the function $\Phi$ is continuous and $h_{x_0^n} = nx$ at sufficiently small values of $x$, we conclude that
\begin{equation*}
\lim_{x\rightarrow 0}\partial_x \Phi_{n,m}(x,z) = n + nz\Psi_m\left(\dfrac{1}{n},z\right)
\end{equation*}
\end{proof}

We now use Lemma \ref{lm:torics_limit_h_function_small_x} to obtain the generating function of the limit $F$-invariants of the diagonal ADE singularities, among others. We assume that $n\geq m$, and the computation splits in three cases: $m=2,3$ or $\geq 4$.
\begin{enumerate}
\item Let $n\geq m=2$. By \ref{lm:gen_limit_h_function_A_n}, for $0\leq x\leq \dfrac{1}{2}$ we have that
\begin{equation*}
    \label{fm:Psi_2_first_interval}
\Psi_2(x,z) = \frac{\cos (z-2zx)+\sin 2zx}{2z\cos z}-\frac{1}{2z}+\frac{x}{1-z}.
\end{equation*}
Thus,
\begin{equation*}
\sum \limeHK(R_{n,2,d})z^d = n + nz\Psi_2\left(\frac{1}{n},z\right) = \frac{n}{2}\frac{\cos(z-\frac{2z}{n})+\sin \frac{2z}{n}}{\cos z}+\frac{n}{2}+\frac{z}{1-z}
\end{equation*}
\item Let $n\geq m=3$. Then, using \ref{lm:gen_limit_h_function_A_n},
\begin{gather*}
    n+nz\Psi_3(\frac{1}{n},z) =  n + \dfrac{3z}{2} + \dfrac{z^2}{1-z} + \begin{cases}
    \dfrac{3n}{4}\dfrac{\cos\frac{z}{3} + \sin\frac{2z}{3}}{\cos z}\sin \frac{2z}{n}, & n\geq 6,\\[10pt]
    \dfrac{3n}{4}\dfrac{\sin\frac{2z}{n} + \cos(z-\frac{2z}{n})}{\cos z}\cos\dfrac{z}{3} -\dfrac{3n}{4} & n=3,4,5.
    \end{cases}
\end{gather*}
\item Let $n\geq m\geq 4$. Then, according to \ref{lm:gen_limit_h_function_A_n}
\begin{equation*}
    \label{fm:Psi_n_first_interval}
\sum_{s=0}^{\infty} \limeHK(R_{n,m,d})z^d = n+nz\Psi_m(\frac{1}{n},z) = n+\dfrac{mz}{2} + \frac{nm}{4}\dfrac{\cos(z-\frac{2z}{m}) + \sin\frac{2z}{m}}{\cos z}\sin \frac{2z}{n} + \frac{z^2}{1-z}.
\end{equation*}
\end{enumerate}

Thus, we just proved the following results:

\begin{corollary}[Limit $F$-invariants of the $D_4$, $E_6$ and $E_8$ singularities.]
    Let $d\geq 2$,
\begin{gather*}
R_d := \frac{\mathbb{C}[[x_0,\dots,x_d]]}{(x_0^n+x_1^3+x_2^2+\dots+x_d^2)}
\end{gather*}
for $n=3,4,5$ ($D_4, E_6$ and $E_8$ singularities, respectively). Then,
\begin{equation*}
\limeHK(R_{d}) = 1+\dfrac{c_d}{d!}, \quad \text{and} \quad \lims(R_{d}) = 1-\dfrac{c_d}{d!}
\end{equation*}
where $c_d/d!$ is the $d$-th Taylor coefficient of the function
\begin{equation*}
        \dfrac{3n}{4}\dfrac{\sin\frac{2z}{n} + \cos(z-\frac{2z}{n})}{\cos z}\cos\dfrac{z}{3}% +\dfrac{n}{4} -1 + \dfrac{z}{2}
\end{equation*}
\end{corollary}

\begin{corollary}
    \label{th:torics_limit_eHK}
Let $d\geq 2$ and
\begin{gather*}
R_d := \frac{\mathbb{C}[[x_0,\dots,x_d]]}{(x_0^n+x_1^m+x_2^2+\dots+x_d^2)},
\end{gather*}
for $n\geq m$ such that $n\geq 6$ and $m\geq 3$. Then,
\begin{equation*}
\limeHK(R_{d}) = 1+\dfrac{c_d}{d!}, \quad \text{and} \quad \lims(R_{d}) = 1-\dfrac{c_d}{d!}
\end{equation*}
where $c_d/d!$ is the $d$-th Taylor coefficient of the function
\begin{align*}
    \frac{nm}{4}\dfrac{\cos\frac{(m-2)z}{m} + \sin\frac{2z}{m}}{\cos z}\sin \frac{2z}{n}% + n-1+\left(\dfrac{m}{2}-1\right)z.
\end{align*}
\end{corollary}

\subsection{The \texorpdfstring{$\limeHK$}{limit eHK} of the \texorpdfstring{$D_n$}{Dn} singularities.}
In this section, we compute the limit Hilbert–Kunz multiplicity and limit $F$-signature of the higher-dimensional $D_n$ singularities.
\begin{notation}
    \label{nt:notation_phi_psi}
Let $n\geq 4$. For $d\geq 0$, we will set $\psi_d(x) := h_{x_0^2+\dots+x_d^2}(x)$, and for $d\geq 1$, $\phi_{n,d}(x) := h_{x_0^2x_1+x_1^{n-1}+x_2^2+\dots+x_d^2}(x)$. Finally, we denote the corresponding generating functions by $\Psi(x,z) := \sum \psi_d(x)z^d$ and $\Phi_{n}(x,z) := \sum \phi_{n,d}(x)z^d$.
\end{notation}

We note that the procedure to obtain the generating function of the limit $F$-invariants for these singularities is the same as we did before for the $D_4,E_6$ and $E_8$ singularities, but it is more involved because the \say{kernel} of the integral we have to compute is slightly more complicated. Still, the fact that we only need the behaviour near zero will play to our advantage.

\begin{lemma}[Integral formula for the limit $h$-function of the $D_n$ singularities]
    \label{lm:D_n_limit_h_function_kernel}
    For $i\geq 0$ and $n\geq 4$,
\begin{gather*}
\phi_{n,i+2} = \int_{0}^{1^+} L_n(t,x)d_{s}(-\psi_{i}'(s)),
\end{gather*}
where $a := \frac{n-4}{2(n-3)}, b := \frac{n}{2(n-1)}$, the functions $\phi,\psi$ as in \ref{nt:notation_phi_psi}, and
    \begin{gather*}
    L_n(t,x) := 4\int_{0}^a D_\infty(s,t,x)ds + \frac{(n-1)^2}{n-2}\int_{a}^b D_\infty(s,t,x)ds.
    \end{gather*}
We conclude that,
\begin{equation*}
\Phi_n(z,x) = \phi_{n,1}(x)z - \int_0^{1^+} z^2\Psi(z,t)\frac{\partial^2}{\partial t^2}L_n(t,x)dt
\end{equation*}
\end{lemma}
\begin{proof}
By \ref{th:integral_formula}, we obtain
\begin{gather*}
\phi_{n,i+2}(x) = \int_{[0,\infty)^2} D_{\infty}(t_1,t_2,x) d_{t_1}(-h'_{x_0^2x_1+x_1^{n-1}}(t_1))d_{t_2}(-h'_{\sum_{j=2}^{i+2}x_j^2}(t_2)).
\end{gather*}
As we saw in the proof of Theorem \ref{th:h_function_D_n_surface},
\begin{gather*}
-h''_{x_0^2x_1+x_1^{n-1}} = -3\delta_0+4\chi_{(0,a)}+\frac{(n-1)^2}{n-2}\chi_{(a,b)},
\end{gather*}
so the result follows from evaluating the integral. All this directly implies
\begin{gather*}
\Phi_n(z,x) = \phi_{n,1}(x)z + \int_0^{1^+} z^2L_n(t,x)d_t\left(-\frac{\partial}{\partial t}\Psi(z,t)\right),
\end{gather*}
so the second statement follows from integration by parts \cite[Corollary 2.26]{mengLimits2026}.
\end{proof}

The next goal will thus be computing $\partial ^2_t L_n$ near the line $x=0$.

\begin{lemma}
    Under the notation of \ref{lm:D_n_limit_h_function_kernel}, for every fixed $0<x\ll1$, we have that $L_n(t,x)$ is $\mathcal{C}^3$ on $0\leq t\leq 1$, and
\begin{gather*}
-\frac{\partial^2}{\partial t^2}L_n(t,x)=\begin{cases}
4t & 0\leq t\leq x \\[5pt]
4x & x\leq t \leq a-x\\[5pt]
2(a+x-t) + \frac{(n-1)^2}{2(n-2)}(t+x-a) & a-x\leq t \leq x+a\\[5pt]
\frac{(n-1)^2}{n-2}x & x+a\leq t \leq b-x\\[5pt]
\frac{(n-1)^2}{2(n-2)}(b-t+x) & b-x\leq t \leq x+b\\[5pt]
0 & x+b\leq t \leq 1
\end{cases}
\end{gather*}
\end{lemma}
\begin{figure}[h!]
    \centering
    
    \hspace{5em}
    
    \begin{subfigure}{0.45\textwidth}
        \centering
        \begin{tikzpicture}[scale=2]
    % Axes
    \draw[thick,->] (0,0,0) -- (1.5,0,0) node[anchor=north east]{$t$};
    \draw[thick,->] (0,0,0) -- (0,1.5,0) node[anchor=south west]{$x$};
    \draw[thick,->] (0,0,0) -- (0,0,1.5) node[anchor=south east]{$s$};

    % Cube    
    \draw[black, thick] (0,0,0) -- (1,0,0) -- (1,1,0) -- (0,1,0) -- cycle;
    \draw[black, thick] (0,0,1) -- (1,0,1) -- (1,1,1) -- (0,1,1) -- cycle;
    \draw[black, thick] (0,0,0) -- (0,0,1);
    \draw[black, thick] (1,0,0) -- (1,0,1);
    \draw[black, thick] (0,1,0) -- (0,1,1);
    \draw[black, thick] (1,1,0) -- (1,1,1);

    \foreach \x in {0,1}
    \foreach \y in {0,1}
    \foreach \z in {0,1} {
        \filldraw (\x,\y,\z) circle (0.5pt);
    }

    % Blue dots
    \filldraw[blue] (0,0,0) circle (1pt) node[anchor=south east]{};
    \filldraw[blue] (0,1,1) circle (1pt) node[anchor=south west]{};
    \filldraw[blue] (1,0,1) circle (1pt) node[anchor=south east]{};
    \filldraw[blue] (1,1,0) circle (1pt) node[anchor=north west]{};

    % Square shadowed section 1
    \coordinate (A) at (1,0,0);
    \coordinate (B) at (1,1,0);
    \coordinate (C) at (0,1,0);
    \coordinate (D) at (0,0,0);
    \filldraw[fill=gray!50, opacity=0.3] (A) -- (B) -- (C) -- (D) -- cycle;

    \draw[black, thick] (0,0,0) -- (1,0,0) -- (1,1,0) -- (0,1,0) -- cycle;

    % Square shadowed section 2
    \coordinate (A) at (1,0,0.33);
    \coordinate (B) at (1,1,0.33);
    \coordinate (C) at (0,1,0.33);
    \coordinate (D) at (0,0,0.33);
    \filldraw[fill=gray!50, opacity=0.3] (A) -- (B) -- (C) -- (D) -- cycle;

    \draw[black, thick] (0,0,0.33) -- (1,0,0.33) -- (1,1,0.33) -- (0,1,0.33) -- cycle;

    % Central tetrahedron
    \draw[red, thick, opacity=0.2] (0,0,0) -- (0,1,1) -- (1,0,1) -- cycle;
    \draw[red, thick, opacity=0.2] (0,0,0) -- (1,0,1) -- (1,1,0) -- cycle;
    \draw[red, thick, opacity=0.2] (0,0,0) -- (0,1,1) -- (1,1,0) -- cycle;
    % Intersection of section 1 and 2
    \draw[red, thick] (0,0,0) -- (1,1,0);
    \draw[red, thick] (0.33,0,0.33) -- (1,0.66,0.33) -- (0.66,1,0.33) -- (0,0.33,0.33) -- cycle;
    % Central tetrahedron again
    \draw[red, thick, opacity=0.2] (0,1,1) -- (1,0,1) -- (1,1,0) -- cycle;

    % Cube    
    \draw[black, thick] (0,0,0) -- (1,0,0) -- (1,1,0) -- (0,1,0) -- cycle;
    \draw[black, thick] (0,0,1) -- (1,0,1) -- (1,1,1) -- (0,1,1) -- cycle;
    \draw[black, thick] (0,0,0) -- (0,0,1);
    \draw[black, thick] (1,0,0) -- (1,0,1);
    \draw[black, thick] (0,1,0) -- (0,1,1);
    \draw[black, thick] (1,1,0) -- (1,1,1);

    % Black dots
    \filldraw[black] (1,0,0) circle (1pt) node[anchor=south east]{};
    \filldraw[black] (1,1,1) circle (1pt) node[anchor=north west]{};
    \filldraw[black] (0,0,1) circle (1pt) node[anchor=south east]{};
    \filldraw[black] (0,1,0) circle (1pt) node[anchor=north west]{};
    \end{tikzpicture}
            \begin{tikzpicture}[scale=2.5, thick]

        % Shading
        \fill [gray!40] (0,0) -- (1/3,0) -- (1/6,1/6) -- cycle;
        \fill [gray!20] (0,0) -- (0,1/3) -- (1/6,1/6) -- cycle;
        \fill [gray!40] (0,1/3) -- (1/6,1/6) -- (5/6,5/6) -- (2/3,1) -- cycle;
        \fill [gray!20] (1/3,0) -- (1/6,1/6) -- (5/6,5/6) -- (1,2/3) -- cycle;
        \fill [gray!40] (1/3,0) -- (1,0) -- (1,2/3) -- cycle;
        \fill [gray!40] (1,2/3) -- (1,1) -- (5/6,5/6) -- cycle;
        \fill [gray!20] (2/3,1) -- (1,1) -- (5/6,5/6) -- cycle;
        \fill [gray!20] (0,1/3) -- (2/3,1) -- (0,1) -- cycle;

        % Region for small x
        \fill [blue!30] (0,0) -- (1,0) -- (1,0.1) -- (0,0.1) -- cycle;
        \draw (0.02, 0.1) -- (-0.02, 0.1) node[left]{\small $x\ll 1$};

        % Axes
        \draw[->] (-0.2,0) -- (1.25,0) node[right]{$t$};
        \draw[->] (0,-0.2) -- (0,1.25) node[above]{$x$};

        % Outer square
        \draw (0,0) rectangle (1,1);

        % Inner diamond
        \draw[red] (1/3,0) -- (1,2/3) -- (2/3,1) -- (0,1/3) -- cycle;

        % Diagonal line
        \draw[red] (0,0) -- (1,1);

        % Dots at intersections
        \filldraw (0,0)   circle (.5pt);
        \filldraw (1,0)   circle (.5pt);
        \filldraw (1,1)   circle (.5pt);
        \filldraw (0,1)   circle (.5pt);
        \filldraw (1/3,0) circle (.5pt);
        \filldraw (1,2/3) circle (.5pt);
        \filldraw (2/3,1) circle (.5pt);
        \filldraw (0,1/3) circle (.5pt);

        % Tick marks and labels on t-axis
        \draw (1/3, 0.02) -- (1/3, -0.02) node[below]{\small $a$};
        \draw (2/3, 0.02) -- (2/3, -0.02) node[below]{\small $1-a$};
        \draw (1,   0.02) -- (1,   -0.02) node[below]{\small $1$};

        % Tick marks and labels on x-axis
        \draw (0.02, 1/3) -- (-0.02, 1/3) node[left]{\small $a$};
        \draw (0.02, 2/3) -- (-0.02, 2/3) node[left]{\small $1-a$};
        \draw (0.02, 1)   -- (-0.02, 1)   node[left]{\small $1$};

            \end{tikzpicture}
        \caption{\centering Regions of the domain of $D_\infty(s,t,x)$ at $0\leq s \leq a$, and $\int_0^a D_\infty(s,t,x)ds$, resp.}
    \end{subfigure}
    \hfill
    \begin{subfigure}{0.45\textwidth}
        \centering
        \begin{tikzpicture}[scale=2]
        % Axes
        \draw[thick,->] (0,0,0) -- (1.5,0,0) node[anchor=north east]{$t$};
        \draw[thick,->] (0,0,0) -- (0,1.5,0) node[anchor=south west]{$x$};
        \draw[thick,->] (0,0,0) -- (0,0,1.5) node[anchor=south east]{$s$};

        \foreach \x in {0,1}
        \foreach \y in {0,1}
        \foreach \z in {0,1} {
            \filldraw (\x,\y,\z) circle (0.5pt);
        }

        % Central tetrahedron
        \draw[red, thick, opacity=0.3] (0,0,0) -- (0,1,1) -- (1,0,1) -- cycle;
        \draw[red, thick, opacity=0.3] (0,0,0) -- (1,0,1) -- (1,1,0) -- cycle;
        \draw[red, thick, opacity=0.3] (0,0,0) -- (0,1,1) -- (1,1,0) -- cycle;
        \draw[red, thick, opacity=0.3] (0,1,1) -- (1,0,1) -- (1,1,0) -- cycle;

        % Blue dots
        \filldraw[blue] (0,0,0) circle (1pt) node[anchor=south east]{};
        \filldraw[blue] (0,1,1) circle (1pt) node[anchor=south west]{};
        \filldraw[blue] (1,0,1) circle (1pt) node[anchor=south east]{};
        \filldraw[blue] (1,1,0) circle (1pt) node[anchor=north west]{};

        % Square shadowed section 1
        \coordinate (A) at (1,0,0.25);
        \coordinate (B) at (1,1,0.25);
        \coordinate (C) at (0,1,0.25);
        \coordinate (D) at (0,0,0.25);
        \filldraw[fill=gray!50, opacity=0.3] (A) -- (B) -- (C) -- (D) -- cycle;
        
        % Intersection of section with tetrahedron and cube 1
        \draw[red, thick] (0.25,0,0.25) -- (1,0.75,0.25) -- (0.75,1,0.25) -- (0,0.25,0.25) -- cycle;

        \draw[black, thick] (0,0,0.25) -- (1,0,0.25) -- (1,1,0.25) -- (0,1,0.25) -- cycle;

        % Square shadowed section 2
        \coordinate (A) at (1,0,0.625);
        \coordinate (B) at (1,1,0.625);
        \coordinate (C) at (0,1,0.625);
        \coordinate (D) at (0,0,0.625);
        \filldraw[fill=gray!50, opacity=0.3] (A) -- (B) -- (C) -- (D) -- cycle;
        
        % Intersection of section with tetrahedron and cube 2
        \draw[red, thick] (0.625,0,0.625) -- (1,0.375,0.625) -- (0.375,1,0.625) -- (0,0.625,0.625) -- cycle;

        \draw[black, thick] (0,0,0.625) -- (1,0,0.625) -- (1,1,0.625) -- (0,1,0.625) -- cycle;

        % Black dots
        \filldraw[black] (1,0,0) circle (1pt) node[anchor=south east]{};
        \filldraw[black] (1,1,1) circle (1pt) node[anchor=north west]{};
        \filldraw[black] (0,0,1) circle (1pt) node[anchor=south east]{};
        \filldraw[black] (0,1,0) circle (1pt) node[anchor=north west]{};

        % Cube    
        \draw[black, thick] (0,0,0) -- (1,0,0) -- (1,1,0) -- (0,1,0) -- cycle;
        \draw[black, thick] (0,0,1) -- (1,0,1) -- (1,1,1) -- (0,1,1) -- cycle;
        \draw[black, thick] (0,0,0) -- (0,0,1);
        \draw[black, thick] (1,0,0) -- (1,0,1);
        \draw[black, thick] (0,1,0) -- (0,1,1);
        \draw[black, thick] (1,1,0) -- (1,1,1);
    \end{tikzpicture}
            \begin{tikzpicture}[scale=2.5, thick]
                % Shading
                \fill [gray!40] (1,0) -- (2/3,0) -- (1,1/3) -- cycle;
                \fill [gray!20] (1,1/3) -- (2/3,0) -- (11/24,5/24) -- (19/24,13/24) -- cycle;
                \fill [gray!40] (1,1/3) -- (19/24,13/24) -- (1,3/4) -- cycle;
                \fill [gray!20] (0,2/3) -- (1/3,1) -- (13/24,19/24) -- (5/24,11/24) -- cycle;
                \fill [gray!40] (2/3,0) -- (11/24,5/24) -- (1/4,0) -- cycle;
                \fill [gray!40] (11/24,5/24) -- (19/24,13/24) -- (13/24,19/24) -- (5/24,11/24) -- cycle;
                \fill [gray!40] (1/3,1) -- (13/24,19/24) -- (3/4,1) -- cycle;
                \fill [gray!20] (3/4,1) -- (13/24,19/24) -- (19/24,13/24) -- (1,3/4) -- cycle;
                \fill [gray!40] (0,2/3) -- (5/24,11/24) -- (0,1/4) -- cycle;
                \fill [gray!20] (1/4,0) -- (11/24,5/24) -- (5/24,11/24) -- (0,1/4) -- cycle;
                \fill [gray!40] (1,1) -- (1,3/4) -- (3/4,1) -- cycle;
                \fill [gray!40] (0,1) -- (1/3,1) -- (0,2/3) -- cycle;
                \fill [gray!40] (0,0) -- (0,1/4) -- (1/4,0) -- cycle;

                % Region for small x
                \fill [blue!30] (0,0) -- (1,0) -- (1,0.1) -- (0,0.1) -- cycle;
                \draw (0.02, 0.1) -- (-0.02, 0.1) node[left]{\small $x\ll 1$};

                % Axes
                \draw[->] (-0.2,0) -- (1.25,0) node[right]{$t$};
                \draw[->] (0,-0.2) -- (0,1.25) node[above]{$x$};
                % Outer square
                \draw (0,0) rectangle (1,1);
                % Inner diamonds
                \draw[red] (2/3,0) -- (0,2/3) -- (1/3,1) -- (1,1/3) -- cycle;
                \draw[red] (1/4,0) -- (0,1/4) -- (3/4,1) -- (1,3/4) -- cycle;
                % Dots at intersections
                \filldraw (1,0)   circle (0.5pt);
                \filldraw (0,0)   circle (0.5pt);
                \filldraw (0,1)   circle (0.5pt);
                \filldraw (1,1)   circle (0.5pt);
                \filldraw (2/3,0) circle (0.5pt);
                \filldraw (0,2/3) circle (0.5pt);
                \filldraw (1/3,1) circle (0.5pt);
                \filldraw (1,1/3) circle (0.5pt);
                \filldraw (1/4,0) circle (0.5pt);
                \filldraw (0,1/4) circle (0.5pt);
                \filldraw (3/4,1) circle (0.5pt);
                \filldraw (1,3/4) circle (0.5pt);

                % Tick marks and labels on t-axis
                \draw (1/3, 0.02) -- (1/3, -0.02) node[below]{\small $a$};
                \draw (1,   0.02) -- (1,   -0.02) node[below]{\small $1$};
                % Tick marks and labels on x-axis
                \draw (0.02, 1/3) -- (-0.02, 1/3) node[left]{\small $a$};
                \draw (0.02, 1)   -- (-0.02, 1)   node[left]{\small $1$};
            \end{tikzpicture}
        \caption{\centering Regions of the domain of $D_\infty(s,t,x)$ at $a\leq s \leq b$, and $\int_a^b D_\infty(s,t,x)ds$, resp.}
    \end{subfigure}
    \caption{Computing $\partial_t^2 L_n$ near $x=0$, Lemma \ref{lm:D_n_limit_h_function_kernel}. In {\color{blue}blue}, a sample of a region where $0<x\ll 1$. In {\color{red}red}, the intersection of the integration limits with the tetraheadron embedded in the domain of $D_\infty$, see Figure \ref{fig:domain_D_infty}.}
    \label{fig:regions_a_b_Dn}
\end{figure}
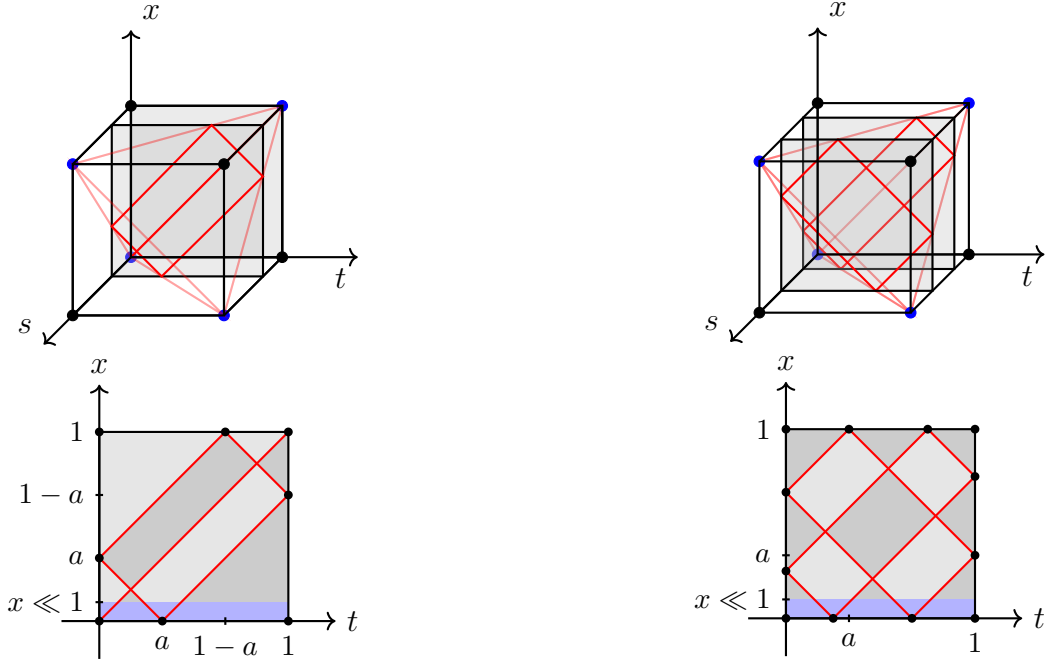
\begin{proof}
The function $D_\infty(s,t,x)$ is a piece-wise polynomial on $0<s\leq b$ and $0<x\ll 1$ small enough (see \ref{rmk:differentiability_D_infty}). In particular, for every fixed $x$, the differentiability of $L_n(t,x)$ only fails at finitely many points. Therefore, we can derive under the integral sign:
\begin{gather}
    \label{fm:second_derivative_L_n_combination}
\frac{\partial^2}{\partial t^2} L_n(t,x) = 4\int_0^a \frac{\partial^2}{\partial t^2}D_\infty (s,t,x) ds+\frac{(n-1)^2}{n-2}\int_a^b \frac{\partial^2}{\partial t^2}D_\infty (s,t,x) ds.
\end{gather}
Study of the different regions of the $D_\infty$ function show the following expressions for these two integrals (see Figure \ref{fig:regions_a_b_Dn}):
\begin{align*}
\int_0^a \frac{\partial^2}{\partial t^2}D_\infty \, ds &= \begin{cases}
-t & 0\leq t\leq x\\[10pt]
-x & x\leq t \leq a-x\\[10pt]
\frac{-1}{2}(a+x-t) & a-x\leq t \leq a+x\\[10pt]
0 & a+x\leq t
\end{cases}\\
\int_a^b \frac{\partial^2}{\partial t^2}D_\infty \, ds &= \begin{cases}
    0 & 0\leq t \leq a-x\\[10pt]
    \frac{-1}{2}(t+x-a) & a-x\leq t \leq a+x\\[10pt]
    -x & a+x\leq t \leq b-x\\[10pt]
    \frac{-1}{2}(b-t+x) & b-x\leq t\leq b+x\\[10pt]
    0 & b+x\leq t
\end{cases}
\end{align*}
which implies the result by substitution of these expressions in \ref{fm:second_derivative_L_n_combination}.
\end{proof}
From this, we can directly obtain the limit Hilbert–Kunz multiplicity and limit $F$-signature of the $D_n$ singularities.

\begin{theorem}[Limit Hilbert–Kunz multiplicities of the $D_n$ singularities]
    \label{th: D_n_singularities_eHK_formula}
Let $d\geq 2$, 
\begin{gather*}
R_d= \frac{\mathbb{C}[[x_0,\dots,x_d]]}{(x_0^2x_1+x_1^{n-1}+x_2^2+\dots+x_d^2)}.
\end{gather*}
Then,
\begin{equation*}
\limeHK(R_{d}) = 1+\dfrac{c_d}{d!}, \quad \text{and} \quad \lims(R_{d}) = 1-\dfrac{c_d}{d!}
\end{equation*}
where $c_d/d!$ is the $d$-th Taylor coefficient of the function
\begin{align*}
T(z) &+ \frac{(n-1)^2}{4(n-2)}\left(\cos\frac{z}{n-1}+T(z)\sin\frac{z}{n-1}\right)\\
&+\frac{(n-3)^2}{4(n-2)}\left(\cos\frac{z}{n-3}+T(z)\sin\frac{z}{n-3}\right)
\end{align*}
where $T(z) = \sec z + \tan z$.
\end{theorem}
\begin{table}[h]
\centering
\renewcommand{\arraystretch}{2.5}
\begin{tabular}{r@{\hspace{1em}}l@{\hspace{1em}}l|c}
\toprule
$\limeHK(D_n)$ & $\underset{n\rightarrow\infty}\longrightarrow$ & $\limeHK(D_\infty)$ & \\
\midrule
$2-\dfrac{1}{2\cdot 2!(n-2)}$
& $\longrightarrow$
& $2$ & $\dim 2$ \\
\midrule
$\dfrac{19}{12}-\dfrac{1}{2\cdot 3!(n-1)(n-3)}$
& $\longrightarrow$
& $\dfrac{19}{12}$ & $\dim 3$ \\
\midrule
\makecell[r]{$\dfrac{11}{8}-\dfrac{2}{4!(n-1)(n-3)}$\\[4pt]
    $+\dfrac{1}{2\cdot 4!(n-1)(n-2)(n-3)}$\\[4pt]
    $+\dfrac{1}{4!(n-1)^2(n-2)(n-3)^2}$}
& $\longrightarrow$
& $\dfrac{11}{8}$ & $\dim 4$\\
\bottomrule
\end{tabular}
\caption{Limit Hilbert--Kunz multiplicity of $D_n$ singularities in dimensions $d=2,3,4$,
converging to the $D_\infty$ values as $n\to\infty$ by $\m$-adic continuity (see Table \ref{tab:lHK-A}).}
\label{tab:lHK-D}
\end{table}
\begin{proof}
    Due to \ref{th:F_invariants_lim_diff_h}, we only need to compute $\lim_{x\rightarrow 0^+}\partial_x\Phi_n(z,x)$. By Lemma \ref{lm:D_n_limit_h_function_kernel},
\begin{align*}
\Phi_n(z,x) &= \phi_{n,1}(x)z + z^2\int_0^{1^+} \Psi(z,t) (-\frac{\partial^2}{\partial t^2})L_n(t,x)dt=\\
&= \phi_{n,1}(x)z + 4z^2\int_0^{x}\Psi(z,t)tdt + 4z^2x\int_x^{a-x}\Psi(z,t)dt \\
&+ z^2\int_{a-x}^{x+a}\Psi(z,t)(2(a+x-t)+\frac{(n-1)^2}{2(n-2)}(t+x-a))dt \\
&+ z^2\frac{(n-1)^2}{n-2}x\int_{x+a}^{b-x}\Psi(z,t)dt + z^2\frac{(n-1)^2}{2(n-2)}\int_{b-x}^{b+x}\Psi(z,t)(b-t+x)dt
\end{align*}
The continuity of the derivatives along $t$ of $L_n$, $\Psi_n$ and $\phi_{n,1}$ for small values of $x$ implies that taking limit and derivative with respect to $x$ commute at every summand involved in the expression above. Thus, we first observe that:
\begin{align*}
&\lim_{x\rightarrow 0^+}\int_{a-x}^{x+a}\Psi(z,t)(2(a+x-t)+\frac{(n-1)^2}{2(n-2)}(t+x-a))dt = 0,\\
&\lim_{x\rightarrow 0^+}\int_{b-x}^{b+x}\Psi(z,t)\frac{(n-1)^2}{2(n-2)}(b-t+x)dt = 0,\\
&\lim_{x\rightarrow 0^+} \int_0^{x}\Psi(z,t)tdt = 0,\\
&\lim_{x\rightarrow 0^+}\frac{\partial}{\partial x} \left(x\int_x^{a-x}\Psi(z,t)dt\right) = \int_0^a\Psi(z,t)dt,\\
&\lim_{x\rightarrow 0^+}\frac{\partial}{\partial x}\left(x\int_{x+a}^{b-x}\Psi(z,t)dt\right) = \int_a^b \Psi(z,t)dt.
\end{align*}
Hence,
\begin{gather*}
\lim_{x\rightarrow 0^+}\frac{\partial \Phi_n(z,x)}{\partial x} = 3z+4z^2\int_{0}^a\Psi(z,t)dt + \frac{(n-1)^2}{n-2}z^2\int_{a}^b\Psi(z,t)dt.
\end{gather*}
The result follows from direct evaluation of the integrals using \ref{lm:gen_limit_h_function_A_n}, taking into account that since $a< \frac{1}{2}< b$, the integral between $a$ and $b$ splits in two: the expression for $\Psi$ is different in $[a,\frac{1}{2}]$ and $[\frac{1}{2},b]$.
\end{proof}
\subsection{The \texorpdfstring{$\limeHK$}{limit eHK} of the \texorpdfstring{$E_7$}{E7} singularities.}

The computation the limit Hilbert–Kunz multiplicity and limit $F$-signature of the $E_7$ singularities is analogous to the computation we did for the $D_n$ singularities in the previous section. We set some notation again:
\begin{notation}
    \label{nt:notation_phi_psi_E_7}
For $d\geq 1$, set $\phi_{d}(x) := h_{x_0^3+x_0x_1^3+x_2^2+\dots+x_d^2}(x)$, and recall that $\psi_d(x) := h_{x_0^2+\dots+x_d^2}(x)$, for $d\geq 0$. Finally, we denote the corresponding generating functions by $\Psi(x,z) := \sum \psi_d(x)z^d$ and $\Phi(x,z) := \sum \phi_{d}(x)z^d$.
\end{notation}
\begin{lemma}[Integral formula for the limit $h$-function of the $E_7$ singularities]
    \label{lm:E_7_limit_h_function_kernel}
    Let $a = \frac{1}{9}, b=\frac{5}{9}$, and
    \begin{gather*}
    M(t,x) := \frac{27}{4}\int_{a}^b D_\infty(s,t,x)ds,
    \end{gather*}
Then, under notation \ref{nt:notation_phi_psi_E_7}, for $i\geq 0$,
\begin{gather*}
\phi_{i+2}(x) = \int_{0}^{1^+} M(t,x)d_{t}(-\psi_{i}'(t)).
\end{gather*}
See Figure \ref{fig:domain_M}. In particular,
\begin{gather*}
\Phi(z,x) = \phi_1(x)z - \int_0^{1^+} z^2\Psi(z,t)\frac{\partial^2}{\partial t^2}M(t,x)dt.
\end{gather*}
\end{lemma}
\begin{proof}
Again, using \ref{th:integral_formula}, we obtain
\begin{gather*}
\phi_{i+2}(x) = h_{x_0^3+x_0x_1^{3}+\sum_{j=2}^{i+2}x_j^2}(x) = \int_{[0,\infty)^2} D_{\infty}(t_1,t_2,x) d_{t_1}(-h'_{x_0^3+x_0x_1^{3}}(s))d_{t}(-h'_{\sum_{j=2}^{i+2}x_j^2}(t)).
\end{gather*}
According to Lemma \ref{lm:h_function_E_7_curve},
\begin{gather*}
-h''_{x_0^3+x_0x_1^{3}} = -3\delta_0 + \frac{27}{4}\chi_{(a,b)}
\end{gather*}
so the result follows from evaluation of the integral.

The above implies that
\begin{gather*}
\Phi(z,x) = \phi_1(x)z + \int_0^{1^+} z^2M(t,x)d_t\left(-\frac{\partial}{\partial t}\Psi(z,t)\right),
\end{gather*}
so the second statement follows from integration by parts, like in \ref{lm:D_n_limit_h_function_kernel}.
\end{proof}

\begin{lemma}
    \label{lm:formula_second_derivative_M}
    For $0<x\ll 1$, we have that
\begin{gather*}
-\frac{\partial^2}{\partial t^2}M(t,x) = \begin{cases}
0 & 0\leq t \leq a-x\\
\frac{27}{8}(x+t) - \frac{3}{8} & a-x \leq t \leq a+x\\
\frac{27}{4}x & a+x\leq t \leq b-x\\
\frac{27}{8}(x-t) + \frac{15}{8} & b-x\leq t\leq b+x\\
0 & b+x\leq t
\end{cases}
\end{gather*}
where $a = \frac{1}{9}$ and $b=\frac{5}{9}$.
\end{lemma}
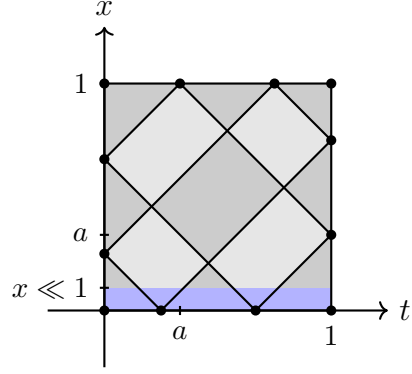
\begin{figure}
\centering
\begin{tikzpicture}[scale=3, thick]
    % Shading
    \fill [gray!40] (1,0) -- (2/3,0) -- (1,1/3) -- cycle;
    \fill [gray!20] (1,1/3) -- (2/3,0) -- (11/24,5/24) -- (19/24,13/24) -- cycle;
    \fill [gray!40] (1,1/3) -- (19/24,13/24) -- (1,3/4) -- cycle;
    \fill [gray!20] (0,2/3) -- (1/3,1) -- (13/24,19/24) -- (5/24,11/24) -- cycle;
    \fill [gray!40] (2/3,0) -- (11/24,5/24) -- (1/4,0) -- cycle;
    \fill [gray!40] (11/24,5/24) -- (19/24,13/24) -- (13/24,19/24) -- (5/24,11/24) -- cycle;
    \fill [gray!40] (1/3,1) -- (13/24,19/24) -- (3/4,1) -- cycle;
    \fill [gray!20] (3/4,1) -- (13/24,19/24) -- (19/24,13/24) -- (1,3/4) -- cycle;
    \fill [gray!40] (0,2/3) -- (5/24,11/24) -- (0,1/4) -- cycle;
    \fill [gray!20] (1/4,0) -- (11/24,5/24) -- (5/24,11/24) -- (0,1/4) -- cycle;
    \fill [gray!40] (1,1) -- (1,3/4) -- (3/4,1) -- cycle;
    \fill [gray!40] (0,1) -- (1/3,1) -- (0,2/3) -- cycle;
    \fill [gray!40] (0,0) -- (0,1/4) -- (1/4,0) -- cycle;

    % Region for small x
    \fill [blue!30] (0,0) -- (1,0) -- (1,0.1) -- (0,0.1) -- cycle;
    \draw (0.02, 0.1) -- (-0.02, 0.1) node[left]{\small $x\ll 1$};

    % Axes
    \draw[->] (-0.25,0) -- (1.25,0) node[right]{$t$};
    \draw[->] (0,-0.25) -- (0,1.25) node[above]{$x$};
    % Outer square
    \draw (0,0) rectangle (1,1);
    % Inner diamond
    \draw (2/3,0) -- (0,2/3) -- (1/3,1) -- (1,1/3) -- cycle;
    \draw (1/4,0) -- (0,1/4) -- (3/4,1) -- (1,3/4) -- cycle;
    % Dots at intersections
    \filldraw (1,0)   circle (0.5pt);
    \filldraw (0,0)   circle (0.5pt);
    \filldraw (0,1)   circle (0.5pt);
    \filldraw (1,1)   circle (0.5pt);
    \filldraw (2/3,0) circle (0.5pt);
    \filldraw (0,2/3) circle (0.5pt);
    \filldraw (1/3,1) circle (0.5pt);
    \filldraw (1,1/3) circle (0.5pt);
    \filldraw (1/4,0) circle (0.5pt);
    \filldraw (0,1/4) circle (0.5pt);
    \filldraw (3/4,1) circle (0.5pt);
    \filldraw (1,3/4) circle (0.5pt);

    % Tick marks and labels on t-axis
    \draw (1/3, 0.02) -- (1/3, -0.02) node[below]{\small $a$};
    \draw (1,   0.02) -- (1,   -0.02) node[below]{\small $1$};
    % Tick marks and labels on x-axis
    \draw (0.02, 1/3) -- (-0.02, 1/3) node[left]{\small $a$};
    \draw (0.02, 1)   -- (-0.02, 1)   node[left]{\small $1$};
\end{tikzpicture}
\caption{Regions of $M$, from Lemma \ref{lm:E_7_limit_h_function_kernel}. Lemma \ref{lm:formula_second_derivative_M} gives the formula for $\partial^2_tM(t,x)$ throughout the five regions in {\color{blue}blue}.}
\label{fig:domain_M}
\end{figure}
\begin{proof}
A similar analysis to the one we did in the case of $L_n$ in \ref{lm:D_n_limit_h_function_kernel} yields that the double derivative and the integral commute, so
\begin{gather*}
\frac{\partial^2}{\partial t^2}M(t,x) = \frac{27}{4}\int_a^b \frac{\partial^2}{\partial t^2} D_\infty (s,t,x)ds.
\end{gather*}
The result follows from direct evaluation of the integral, see Figure \ref{fig:domain_M} for the different regions in the domain of $M$.
\end{proof}

\begin{theorem}[Limit $F$-invariants of the $E_7$ singularity]
Let $d\geq 2$, 
\begin{gather*}
R_d = \frac{\mathbb{C}[[x_0,\dots,x_d]]}{(x_0^3+x_0x_1^3+x_2^2+\dots+x_d^2)}.
\end{gather*}
Then,
\begin{equation*}
\limeHK(R_{d}) = 1+\dfrac{c_d}{d!}, \quad \text{and} \quad \lims(R_{d}) = 1-\dfrac{c_d}{d!}
\end{equation*}
where $c_d/d!$ is the $d$-th Taylor coefficient of the function
\begin{equation*}
    \frac{27}{16}\left(\cos\dfrac{2z}{9} + \sin\dfrac{z}{9}\right)T(z) - \frac{27}{16}\left(\sin\dfrac{2z}{9} - \cos\dfrac{z}{9}\right) %+ \frac{z}{2} - \frac{27}{8}
\end{equation*}
where $T(z) = \sec z+\tan z$.
\end{theorem}
\begin{proof}
    According to Lemma \ref{lm:E_7_limit_h_function_kernel},
\begin{align*}
\Phi(z,x) &= \phi_1(x)z + z^2\int_0^{1^+} \Psi(z,t) (-\frac{\partial^2}{\partial t^2})M(t,x)dt=\\
&= \phi_1(x)z + z^2\int_{a-x}^{a+x}\Psi(z,t)\left(\frac{27}{8}(x+t)-\frac{3}{8}\right)dt\\
&+ \frac{27z^2x}{4}\int_{x+a}^{b-x}\Psi(z,t)dt + z^2\int_{b-x}^{b+x}\Psi(z,t)\left(\frac{27}{8}(x-t)+\frac{15}{8}\right)dt,
\end{align*}
where $a = \frac{1}{9}$ and $b=\frac{5}{9}$. A similar analysis of the different parts like we did in \ref{th: D_n_singularities_eHK_formula} yields that
\begin{gather*}
\lim_{x\rightarrow 0^+}\partial_x \Phi(z,x) = 3z+\frac{27z^2}{4}\int_a^b\Psi(z,t)dt.
\end{gather*}
The result is obtained from direct evaluation of the integral using the formulas in Lemma \ref{lm:gen_limit_h_function_A_n}.% \todo{maybe write explicitly, because of the separation of the $[a,b]$ interval in three intervals.}
\end{proof}
\begin{remark}
These computations can be repeated for suspensions of other hypersurfaces for which the $h$-function is known, despite us focusing on suspensions of binomials and the simple singularites in this paper.
\end{remark}

\section{Limit Hilbert–Kunz multiplicity of the Fermat hypersurfaces.}
\label{sec:fermat}

In this section, we compute the (generating function of the) limit Hilbert–Kunz multiplicity of the Fermat hypersurfaces, the hypersurfaces of the form $\sum x_i^d$, by mere extension of Gessel and Monsky's technique from \cite{gesselLimit2010}. This section is independent from the previous ones.

For the rest of the section, we make a change in indexing for convenience: let $s,d_i\geq 2$, $i=1,\dots,s$ and write
    \begin{gather*}
        R_{0,s-1}^{(d_1,\dots,d_s)} := \frac{\mathbb{C}[[x_1,\dots,x_s]]}{(x_1^{d_1}+\dots+x_s^{d_s})} \qquad \text{and} \qquad R_{p,s-1}^{(d_1,\dots,d_s)} := \frac{\overline{\mathbb{F}_p}[[x_1,\dots,x_s]]}{(x_1^{d_1}+\dots+x_s^{d_s})}
    \end{gather*}
    the coordinate ring of the $(s-1)$-dimensional diagonal hypersurface $x_1^{d_1}+\dots+x_s^{d_s}$ over $\mathbb{C}$ and its reductions mod $p$ (see \ref{def:reduction_mod_p} for the reduction mod $p$ setup). We start by recalling some notation.
\begin{definition}[Definition 2.1, \cite{gesselLimit2010}]
    For $\lambda\in \mathbb{Z}$, denote
    \begin{gather*}
        C_\lambda := \sum (\epsilon_1\cdots \epsilon_s)\left(\frac{\epsilon_1}{d_1}+\cdots+\frac{\epsilon_s}{d_s}-2\lambda\right)^{s-1},
    \end{gather*}
    where the sum extends over the $s$-tuples $\epsilon_1,\dots,\epsilon_s$ where $\epsilon_i = \pm 1$ and $\sum \frac{\epsilon_i}{d_i}>2\lambda$.
\end{definition}
\begin{theorem}[Theorem 2.4, \cite{gesselLimit2010}]
\label{th: general_limit_formula}
    For given $s\geq 2$ and $d_1,\dots,d_s$, the limit Hilbert–Kunz multiplicity of $R_{0,s-1}^{(d_1,\dots,d_s)}$ exists, and
    \begin{gather*}
        \limeHK\left(R_{s-1}^{(d_1,\dots,d_s)}\right) = \frac{1}{(s-1)!}\frac{\prod d_i}{2^{s-1}}\sum_{\lambda\in \mathbb{Z}} C_\lambda.
    \end{gather*}
\end{theorem}

In order to compute the limit Hilbert–Kunz multiplicity of the Fermat quadrics, Gessel and Monsky introduce the following auxiliary function over the integers, which we will also use for the rest of the Fermat hypersurfaces:
\begin{definition}[Definition 2.6 in \cite{gesselLimit2010}]
    For integer $a,s$, with $s\geq 1$,
    \begin{gather*}
        f_{s}(a) := \sum (-1)^k{\binom{s}{k}}(a-2k)^{s-1},
    \end{gather*}
    where the sum runs over the integers $s\geq k\geq 0$ such that $a>2k$. Note that $f_{s}$ is a finite sum, and that $f_{s}(a)=0$ for every $a\in \mathbb{Z}_{\leq 0}$.
\end{definition}
\begin{theorem}[cf. Theorem 2.7, \cite{gesselLimit2010}]
\label{th: limit_Fermat_with_fs}
    Let $d\geq 2$. Then,
    \begin{gather*}
        \limeHK\left(R_{0,s-1}^{(d,\dots,d)}\right) = \frac{1}{(s-1)!}\frac{d}{2^{s-1}}\left(f_s(s)+2f_s(s-2d)+2f_s(s-4d)+2f_s(s-6d)+\dots\right).
    \end{gather*}
\end{theorem}
\begin{proof}
    Just note that $C_\lambda = \frac{1}{d^{s-1}}f_s(s-2d\lambda)$, and apply Theorem \ref{th: general_limit_formula}.
\end{proof}
\begin{remark}
    \label{rmk:f_s_lambda_plusminus}
    By \cite[Lemma 2.3]{gesselLimit2010}, it follows that $f_s(s-2d\lambda) = f_s(s+2d\lambda)$, and therefore the sum in the theorem above can also be expressed as the sum $\sum f_s(a)$ extending over all $a \equiv s\mod{2d}$. This rewording is instrumental in the proof (also used in Gessel and Monsky's paper). Also, note that \cite[Theorem 2.7]{gesselLimit2010} follows from this theorem after setting $d=2$.
\end{remark}

\begin{definition}[Eulerian polynomials]
\label{def:eulerian_polynomials}
For $n\geq 0$, $A_n(T) := (1-T)^{n+1}\sum_{k=1}^{\infty} k^nT^{k-1}.$ Note that $A_0(T) = 1$, and $A_n(1) = n!$.
\end{definition}
The next lemma uses the following useful trick to compute $\sum_{j\in \mathbb{Z}} f_s(s-2dj)$: let $\lbrace a_j\rbrace_{j=0}^{\infty}$ be a sequence such that $\sum_{j=0}^{\infty} a_j<\infty$, and let $P(T) = \sum_{j=0}^{\infty} a_jT^j$ be its generating function; if $\zeta = e^{\frac{\pi}{d}j}$ is a $2d$-th root of unity, then $\sum a_{2dj+k} = \frac{1}{2d}\sum_{r=0}^{2d-1}\zeta^{-rk}P(\zeta^r)$.
\begin{lemma}
\label{lm:fermat_eHK_C}
    Fix $d,s\geq 2$ integers. Set $\zeta := e^{\frac{\pi}{d}i}$ a root of unity, and for $r=1,\dots,d-1$, set $\lambda_r := \frac{1}{2}\zeta^{-r}(1+\zeta^r)^2$. Define
    \begin{gather*}
        C_{s-1} := \sum_{r=1}^{d-1} \lambda_r^{s-1}\frac{A_{s-1}(\zeta^r)}{(1+\zeta^r)^{s-2}}.
    \end{gather*}
    Then,
    \begin{gather*}
\limeHK\left(R_{0,s-1}^{(d,\dots,d)}\right) = 1+\dfrac{1}{(s-1)!}\frac{C_{s-1}+\overline{C_{s-1}}}{2}.
\end{gather*}

\end{lemma}
\begin{proof}
    Fix $d\geq 2$ and $s\geq 2$ and write $C := C_{s-1}$. The result will follow once we proof that
    \begin{gather}
        \label{eq:sum_f_s_goal}
        \sum_{j = s+2d\mathbb{Z}} f_s(j) = \frac{2^{s-1}}{d}\left(A_{s-1}(1)+\frac{C+\bar{C}}{2}\right).
    \end{gather}
    Let $P(T) := T^{1-s}\sum_{j\in \mathbb{Z}} f_s(j)T^{j-1}$. Using that $\sum_{k=0}^{2d-1} \zeta^{kr} = 0$, one can see that
\begin{gather}
\label{eq: sum_1}
\sum_{j=s+2d\mathbb{Z}} f_s(j) = \frac{1}{2d}\left(\sum_{r=0}^{2d-1} P(\zeta^r)\right).
\end{gather}
Lemma 3.6 in \cite{gesselLimit2010} says that $\sum_{j\in \mathbb{Z}} f_s(j)T^{j-1} = (1+T)^sA_{s-1}(T)$, so for $0\leq r \leq d$,
\begin{equation*}
P(\zeta^r) = \begin{cases}
    P(1) = 2^sA_{s-1}(1) & r = 0,\\
    \zeta^{r(1-s)}(1+\zeta^r)^sA_{s-1}(\zeta^r) = \lambda_r^{s-1}\dfrac{A_{s-1}(\zeta^r)}{(1+\zeta^r)^{s-2}} & 1\leq r\leq d-1,\\
    P(-1) = 0 & r=d.\\
\end{cases}
\end{equation*}
Now, note that $\sum_{r=1}^{d-1} P(\zeta^r) = 2^{s-1}C$ and $\sum_{r=d+1}^{2d-1} P(\zeta^r) = 2^{s-1}\bar{C},$  so the sum (\ref{eq: sum_1}) is
\begin{gather*}
\frac{1}{2d}\left(P(1) + \sum_{r=1}^{d-1} P(\zeta^r) + \sum_{r=d+1}^{2d-1} P(\zeta^r) + P(-1)\right) =
\frac{2^{s-1}}{d}\left(A_{s-1}(1)+\frac{C}{2}+\frac{\bar{C}}{2} + 0\right),
\end{gather*}
so \ref{eq:sum_f_s_goal} follows. The result now follows from Theorems \ref{th: limit_Fermat_with_fs} and \ref{lm:fermat_eHK_C}, and Remark \ref{rmk:f_s_lambda_plusminus}.
\end{proof}
\begin{remark}
    \label{rmk:fermat_lambda_real}
Note that $\lambda_r = \frac{1}{2}\zeta^{-r}(1+\zeta^r)^2 = \frac{1}{2}(2+\zeta^{-r}+\zeta^r) = 1 + \cos\frac{r\pi}{d} \in \mathbb{R}.$ Later, we will conclude that also $C_{s-1}\in \mathbb{R}$.
\end{remark}

\begin{lemma}[cf. Lemma 3.4 in \cite{gesselLimit2010}]
% \label{lm: first_taylor_expansion}
\label{lm:fermat_phi_r_taylor_expansion}
    For $d\geq 2$ and $0\leq r\leq d-1$ integers, $\zeta = e^{\frac{\pi}{d}i}$ a $2d$-root of unity, and $z\in (0,1)\subset \mathbb{R}$,
\begin{gather*}
1+\sum_{n = 1}^{\infty} \frac{A_{n}(\zeta^r)}{(1+\zeta^r)^{n-1}}\frac{z^n}{n!} = \frac{1-a\cos cz+b\sin cz}{\cos cz-a} =: \varphi_r(z)
\end{gather*}
where $a = \cos \frac{\pi r}{d}$ and $b =\sin \frac{\pi r}{d}$ and $c = \tan \frac{r\pi}{2d}$. In particular, for every $z\in \mathbb{R}$, then $\varphi_r(\lambda_rz)\in \mathbb{R}$.
\end{lemma}
\begin{proof}
    Due to Lemma 3.3 in \cite{gesselLimit2010},
    \begin{equation*}
    \sum_{n=1}^{\infty} \frac{A_n(T)}{(1-T)^{n+1}}\frac{z^n}{n!} = \dfrac{e^z-1}{1-Te^z} \stackrel{z \mapsto z(1-T)}\Longrightarrow \sum_{n=1}^{\infty} A_n(T)\frac{z^n}{n!} = \dfrac{e^{z(1-T)}-1}{1-Te^{z(1-T)}}.
    \end{equation*}
    Now, fix $0\leq r \leq d-1$, and evaluate $T$ at $\zeta^r$ and $z$ at $z/(1+\zeta^r)$, and set $\omega_r:=(\zeta^r-1)/(\zeta^r+1)$, to get
    \begin{gather}
        \label{eq:complex_identity}
        \sum_{n = 1}^{\infty} \frac{A_n(\zeta^r)}{(1+\zeta^r)^n}\frac{z^n}{n!} = \frac{e^{-\omega_r z}-1}{1-\zeta^re^{-\omega_r z}} = \frac{1-e^{\omega_r z}}{e^{\omega_r z}-\zeta^r}.
    \end{gather}
    By multiplying by $(1+\zeta^r)$ and adding $1$ in both sides of \ref{eq:complex_identity}, we obtain that
    \begin{gather}
    \label{fm: second_taylor_expansion}
        1 + \sum_{n=1}^\infty \frac{A_{n}(\zeta^r)}{(1+\zeta^r)^{n-1}}\frac{z^{n}}{n!} = \frac{1-\zeta^re^{\omega_rz}}{e^{\omega_r z}-\zeta^r}.
    \end{gather}
Note that $\omega_r = i\tan\frac{r\pi}{d}$ is pure imaginary, so let us denote $\zeta^r = a+bi$ and $\omega_r = ci,$ where $a,b$ and $c$ are real numbers, for the sake of simplicity. Note that $c = \tan \frac{r\pi}{2d}$. Now, let us prove that the above function is a real-valued function over the real numbers, so let us from now on assume $z\in \mathbb{R}$:
\begin{align}
\nonumber \frac{1-(a+bi)e^{czi}}{e^{czi}-a-bi} &= \frac{1-(a+bi)(\cos cz+i\sin cz)}{\cos cz+i\sin cz-a-bi}\\
\label{fm: trig_expression}
&=\frac{(1-a\cos cz+b\sin cz)+i(-b\cos cz-a\sin cz)}{\cos cz-a+i(\sin cz-b)} = \frac{A-iB}{C-iD}
\end{align}
where $A := 1-a\cos cz+b\sin cz,B := b\cos cz+a\sin cz,C := \cos cz-a$ and $D := b-\sin cz$. Now, we can check that $\frac{A}{C} = \frac{B}{D}$: given that $1=\sin^2 cz + \cos^2 cz$ and $1 = a^2+b^2$ (because $\zeta$ is a root of unity), we have that
\begin{align*}
\frac{A}{C} &= \frac{1-a\cos cz + b\sin cz}{\cos cz-a} = \frac{(1-a\cos cz + b\sin cz)(\cos cz + a)}{\cos^2 cz-a^2} =\\
&=\frac{\cos cz + a - a\cos^2 cz-a^2\cos cz+b\sin cz\cos cz+ab\sin cz}{1-\sin^2 cz+b^2-1} =\\
&=\frac{b^2\cos cz + a\sin^2 cz+b\sin cz\cos cz+ab\sin cz}{b^2-\sin^2cz} = \\
&=\frac{(b\cos cz+a\sin cz)(b+\sin cz)}{b^2-\sin^2 cz}=\frac{b\cos cz+a\sin cz}{b-\sin cz} = \frac{B}{D}.
\end{align*}
Thus, the claim implies that
\begin{gather*}
\frac{C}{A}(A-iB) = \frac{CA}{A} - i\frac{C}{A}B = C - i\frac{D}{B}B = C - iD.
\end{gather*}
Therefore,
\begin{gather*}
\frac{A-iB}{C-iD} = \frac{A}{C} = \frac{1-a\cos cz + b\sin cz}{\cos cz -a}.
\end{gather*}
\end{proof}

\begin{theorem}
    \label{th:fermat_limit_eHK_gen}
    Let $s\geq 2$ and $d\geq 2$. Then,
\begin{equation*}
\limeHK\left(R_{0,s-1}^{(d,\dots,d)}\right) = 1+\dfrac{C_{s-1}}{(s-1)!},
\end{equation*}
where $C_{s-1}/(s-1)!$ is the $(s-1)$-th Taylor coefficient of the function $\sum_{r=1}^{d-1}\psi_r(z)$ for
\begin{equation*}
    \psi_r(z) = \dfrac{1-a_r\cos b_rz + b_r\sin b_rz}{\cos b_rz -a_r}
\end{equation*}
where $a_r = \cos \frac{\pi r}{d}$, $b_r =\sin \frac{\pi r}{d}$.
%$\lambda_r = 2\cdot \frac{\zeta^{r}-1}{\zeta^{r}+1}\cdot \cos^{2}\frac{\pi r}{2d} \in \mathbb{Q}(\zeta_{2d})$ 
Note that $a_r,b_r\in \mathbb{Q}(\zeta)\cap \mathbb{R}$.
\end{theorem}
\begin{proof}
   We first show that $C_{s-1}\in \mathbb{R}$. By Lemma \ref{lm:fermat_phi_r_taylor_expansion}, $\sum_{r=1}^{d-1}\varphi_r(\lambda_rz) = \sum C_{s-1}\frac{z^{s-1}}{(s-1)!}$, and for all $z\in \mathbb{R}$, then $\varphi_r(\lambda_rz)\in \mathbb{R}$. It follows that $C_{s-1}\in \mathbb{R}$, so $C_{s-1}+\overline{C_{s-1}} = 2C_{s-1}$, and hence $\limeHK\left(R_{0,s-1}^{(d,\dots,d)}\right) = 1+\dfrac{C_{s-1}}{(s-1)!}$.

   Finally, note that $\lambda_rc_r = (1+\cos \frac{\pi r}{d})\tan\frac{r\pi}{2d}= \sin\frac{\pi r}{d} = b_r$. It follows that $\varphi_r(\lambda_r z) = \psi_r(z)$.
% By Theorem \ref{lm:fermat_phi_r_taylor_expansion},
% \begin{align*}
% \sum_{r=1}^{d-1} \varphi_r\left(\lambda_rz\right) &= \sum_{r=1}^{d-1}\sum_{s=1}^{\infty} \lambda_r^{s-1}\frac{A_{s-1}(\zeta^r)}{(1+\zeta^r)^{s-2}}\frac{z^{s-1}}{(s-1)!}\\
% &=\sum_{s=1}^{\infty} z^{s-1}\sum_{r=1}^{d-1}\lambda_r^{s-1}\frac{A_{s-1}(\zeta^r)}{(1+\zeta^r)^{s-2}}\frac{1}{(s-1)!}\\
% &=\sum_{s=1}^{\infty} z^{s-1}\frac{C_{s-1}}{(s-1)!}.
% \end{align*}
\end{proof}
\begin{corollary}[Conjecture 8.8, \cite{mengLimits2026}]
For $d\geq 2$, the generating function of the limit Hilbert–Kunz multiplicities of the Fermat hypersurfaces
\begin{equation*}
    \sum_{n=0}^{\infty}\limeHK\left(R_{0,n}^{(d,\dots,d)}\right)z^{n}
\end{equation*}
is a rational function on the functions $z, \cos(a_iz)$ and $\sin(a_iz)$, for some finitely many $a_i\in \mathbb{Q}(\zeta_{2d})\cap \mathbb{R}$ with $\zeta_{2d} = e^{\frac{\pi}{d}i}$. 
\end{corollary}
\begin{remark}
    \cite{monsky2009transcendence}
Although this method computes the limit Hilbert–Kunz multiplicities of the Fermat hypersurfaces, it does not compute the limit F-signatures, whereas \cite[Corollary 8.7]{mengLimits2026} shows that the algorithm laid out by Meng would compute both invariants. However, a similar process applied to \cite[Corollary 3.13]{shidelerFsignatureFunctionsDiagonal2024} might work to yield an analogous result for the limit $F$-signature.
\end{remark}

\section{The Watanabe–Yoshida conjecture in characteristic zero.}
\label{sec:wat_yos_char_0}
The goal of this section is to extend three results from Hilbert–Kunz theory to characteristic zero, in terms of the limit F-invariants, mainly, Watanabe and Yoshida's \cite{WY00} and Huneke and Leuschke's \cite{hunekeTwoTheoremsMaximal2004} characterisation of regularity via $F$-invariants, and the known cases of the Watanabe–Yoshida conjecture \cite{WY05,enescushimomotoDenseUpperSemicontinuity2005,aberbachenescu2013,coxsteibaberbach2024,castilloreyStrong2025}.

\subsection{Reduction mod $p$ of $A_1$ singularities}
\begin{definition}
    \label{def:formally_A_1_singularity}
Let $k$ be a field of characteristic $p\neq 2.$ We say that a local equicharacteristic $k$-algebra $(R,\m,k)$ is an $A_1$ singularity if there is a $k$-algebra isomorphism $\hat{R} \cong k[[x_1,\dots,x_n]]/(a_1x_1^2+\dots+a_nx_n^2)$ for some $a_i\in k$. If $k$ is quadratically closed (e.g., $k=\bar{k}$), then one can set $Q_n = x_1^2+\dots+x_n^2$.
\end{definition}
Given a commutative ring $L$, for any $f \in L[x_1,\dots,x_n]$ we define the Hessian of $f$ at $0$ as
\begin{gather*}
    H(f) := \left(\dfrac{\partial^2f}{\partial x_ix_j}(0)\right)_{1\leq i,j\leq n}
\end{gather*}
\begin{lemma}[Reduction mod $p$ of an $A_1$ singularity]
    \label{lm:reduction_A_1_singularity}
Let $S := K[x_1,...,x_n]_{(x_1,...,x_n)}$, $\mathfrak{a} := (f_1,...,f_r)\subset S$ and $R = S/\mathfrak{a}$. Let $(A,S_A,R_A,\mathfrak{a}_A)$ descent data for $(K,S,R,\mathfrak{a})$. If $R_k$ is an $A_1$ singularity for some closed fiber $k$, then $R$ is an $A_1$ singularity.
\end{lemma}
\begin{proof}
Let $(A,S_A,\mathfrak{a}_A)$ be descent data for the triple, where $\mathfrak{a}_A = (f_{1,A},\dots,f_{r,A})$. Since $\mathfrak{a}_A$ is $A$-free, $\mu_{k(\p)}(\mathfrak{a}_{k(\p)}) = r$ is constant throughout $\Spec{A}$. Now, by hypothesis, for some closed fiber $k$, we have that $R_k = S_k/\mathfrak{a}_k$ such that $\mathfrak{a}_k\widehat{S_k}$ is principal, which implies that $\mathfrak{a}_k$ is also principal, and therefore $r = 1$ throughout. Thus, possibly enlarging the model, we take $f_A\in S_A$ such that $\mathfrak{a}_A = f_A S_A$, $\mathfrak{a} = f_A S$, and $\mathfrak{a}_k = f_k S_k$.

Now, the Splitting Lemma \cite[Theorem 2.1]{greuelSplittingLemmaAny2026} implies that $R_k$ is an $A_1$ singularity in the sense of Definition \ref{def:formally_A_1_singularity} if and only if $\det H(f_k) \neq 0$. Let $u := \det H(f) \in K$, and take $u_A = u$, and set $u_k$ the images of $u$ in the closed fibers $k$. As we have defined the Hessian matrix, we have that $u_A = \det H(f_A)$ and $u_k = \det H(f_k)$. But if $u_k \neq 0$ at some closed fiber $k$, then $u_A\neq 0$, and the result follows from the Splitting Lemma over $K$.
\end{proof}
\subsection{Characteristic 0 version of some results in Hilbert–Kunz theory}
We start off by characterising the regularity of local rings essentially of finite type over a field of characeristic 0 via limit $F$-invariants. We need the following upper bound for the $F$-signature:
\begin{proposition}[A characteristic-free $(e(R),\dim R)$-dependent upper bound for $F$-signature]
    \label{th:F_signature_char_free_lower_bound}
Let $d\geq 2$. There exists $\epsilon(d)>0$ such that for every $(R,\m,k)$ non-regular $F$-finite local ring of characteristic $p>0$, dimension $d\geq 2$ and (Hilbert–Samuel) multiplicity $e = e(R)$,
\begin{gather*}
s(R) \leq 1-\frac{\epsilon(d)}{e-1}
\end{gather*}
\end{proposition}
\begin{proof}
We assume that $R$ is strongly $F$-regular since, otherwise, $s(R)=0$. If $R$ is not Gorenstein, \cite[Proposition 4.1]{jeffriesmaxFsign2023} gives the stronger inequality $s(R) <\frac{1}{2}$. For the strongly $F$-regular Gorenstein case, the inequality follows from combining the inequalities in Proposition 4.4 in \textit{op. cit.} (or \cite[Proposition 14]{hunekeTwoTheoremsMaximal2004}) and \cite[Theorem 4.12]{aberbachenescu2008}.
\end{proof}
\begin{remark}
It still remains to find a characteristic-free and multiplicity-free upper bound for the $F$-signature, like Aberbach and Enescu's for the Hilbert–Kunz multiplicity. The best guess we have would be provided by the Watanabe–Yoshida conjecture for F-signature, \cite[Conjecture 2.10]{jeffriesmaxFsign2023}: if the conjecture is true, given that $p\mapsto s(R_{p,d})$ is an increasing sequence (by \cite{mengQuadrics2026}), it would follow from the conjecture that $s(R)\leq \lims(R_{0,d})$, for all non-regular local rings $(R,\m,k)$ of dimension $d$.
\end{remark}
We say that a local Noetherian ring $(R,\m)$ is \textit{unmixed} if $\dim R= \dim R/\p$ for every $\p\in \operatorname{Ass}(R)$, and \textit{formally unmixed} when the $\m$-adic completion $\widehat{R}$ is unmixed. Note that these two notions are equivalent over excellent rings.
\begin{theorem}
    \label{th:regularity_char0_infinitely}
Let $K$ be a field of characteristic zero and $(R,\m,K)$ be the ring
\begin{gather*}
R = \left(K[x_1,\dots,x_n]/\mathfrak{a}\right)_{(x_1,\dots,x_n)}
\end{gather*}
If $R$ is a regular local ring, then $\eHK^+(R) = \eHK^-(R) = s^+(R) = s^-(R) = 1$. If one of the invariants $\eHK^+(R),\eHK^-(R),s^+(R)$ or $s^-(R)$ equals one, and $R$ is unmixed, then $R$ is a regular local ring.
\end{theorem}
\begin{proof}
If $R$ is regular, after localising $A$ at some element, we can assume that $R_k$ is regular for all fibers of any family $A\rightarrow R_A$ of characteristic $p$ models, and the first statement follows.

For the second statement, it is enough to check the statement when $\eHK^-(R) = 1$ or $s^+(R) = 1$. Fix descent data $(A,R_A)$ for $(K,R)$. Since $R$ is unmixed, by \cite[Theorem (2.3.9)(g)]{hochsterTightClosureEqual1999}, after possibly localising $A$ at some element, we can further assume that all closed fibers of $A\rightarrow R_A$ are unmixed. The locus $U\subset\Spec{A}$ of geometrically regular fibers $R_k$ is open by \cite[Théorème 12.2.4]{EGA4_3}, so we only need to show that $U$ is non-empty, since then it contains the generic point, and thus that $R_K = R$ is regular.

By flatness, all closed fibers $R_k$ of $A\rightarrow R_A$ have the same dimension $d$ \cite[Theorem 15.1]{matsumura1989commutative}, so by \cite[Theorem 4.12]{aberbachenescu2008}, there exists some $\varepsilon(d)>0$ such that if $\eHK(R_k)<1+\varepsilon(d)$, then $\eHK(R_k) = 1$. If $\eHK^-(R) = 1$, there are fibers $k$ such that $R_k$ is geometrically regular over $k$ by \cite[Theorem 1.5]{WY00}. Since multiplicity is constant across the closed fibers of $A$ by \cite[Lemma 4.2]{nunezbetancourtsmirnov2020}, Proposition \ref{th:F_signature_char_free_lower_bound} yields that $s^+(R)=1$ also implies that $U$ is non-empty.
\end{proof}
\begin{corollary}
Let $K,R$ and $\mathfrak{a}$ be as in Theorem \ref{th:regularity_char0_infinitely}. Let $A$ be any finitely generated $\mathbb{Z}$-algebra containing the coefficients of a set of generators of $\mathfrak{a}$. If there is some dense $U \subset \MaxSpec{A}$ such that $\eHK(R_{k(\mathfrak{m})}) = 1$ for every $\m\in U$, then $R$ is regular. In particular, if $R$ is unmixed, $A$ is finitely generated as a $\mathbb{Z}$-module, $\eHK(R_{k}) = 1$ at infinitely many fibers, then $R$ is regular.
\end{corollary}
We will now extend the known cases of the Watanabe–Yoshida conjecture \cite[Conjecture 4.2]{WY05} to characteristic zero.
\begin{notation}
    Set
\begin{equation*}
R_{0,d} = \left(\dfrac{\mathbb{C}[x_0,\dots,x_d]}{(x_0^2+\dots+x_d^2)}\right)_{(x_0,\dots,x_d)} \quad \text{and} \quad S_{0,d} = \left(\dfrac{\mathbb{C}[x_0,\dots,x_d]}{(x_0^3+x_1^2+\dots+x_d^2)}\right)_{(x_0,\dots,x_d)}
\end{equation*}
the rings of germs of complex rational functions at the origin of the $A_1$ and $A_2$ singularities, respectively (see \ref{def:simple_singularities} for the definition of the ADE singularities).
\end{notation}

\begin{theorem}[Characteristic 0 Watanabe–Yoshida conjecture in low dimension]
    \label{th:watanabe_yoshida_low_dimension}
    Let $K$ be a field of characteristic $0$. Let $(R,\m,K)$ be a $d$-dimensional $K$-algebra of essentially finite type. Moreover, assume that $R$ is unmixed and nonregular. If either $2\leq d\leq 7$, or $R$ is a complete intersection of dimension $\geq 2$, then either $\eHK^-(R)> \limeHK(R_{0,d})$, or $R$ is an $A_1$ singularity.
    \end{theorem}

\begin{lemma}
    \label{lm:A_1_and_A_2_different_eHK}
For every $d\geq 2$, $\limeHK\left(S_{0,d}\right)>\limeHK\left(R_{0,d}\right)$ i.e., the limit Hilbert–Kunz multiplicity distinguishes the $A_1$ and $A_2$ singularities.
\end{lemma}
\begin{proof}
We already know that $\limeHK(S_{0,d})\geq \limeHK(R_{0,d})$ by taking the limit $p\rightarrow \infty$ in \cite[Theorem 2.14]{castilloreyStrong2025}. Hence, it now suffices to show that these two values are different at every $d\geq 2$. Due to \ref{cor:limit_eHK_A_n_singularities}, this equates to showing that every coefficient of degree $\geq 2$ in the Taylor expansion of $A(z) = (1+\sin z)\sec z$ is different from every coefficient of the Taylor expansion of $B(z) = \frac{3}{2}(\cos \frac{z}{3} + \sin \frac{2z}{3})\sec z$.

By Proposition \ref{pr:augmented_r_signed_permutations_generating_function}, every coefficient of $A(z)$ is of the form $E_d/d!$ where $E_d$ are the Euler numbers, which are integers, and in the case of $B(z)$, the coefficients are $\frac{N_d(2,3)}{2\cdot 3^{d-1}d!}$, where $N_d(3,2)$ are also integers (see Definition \ref{def:augmented_r_signed_permutations} for the definition of the numbers $N_d(p,r)$). Thus, we need to see that $E_d \neq \frac{N_d(2,3)}{2\cdot 3^{d-1}}$ for $d\geq 2$. Since $E_d\in \mathbb{Z}$, it suffices to show that $N_d(2,3)\not\equiv 0 \pmod{3}$, so that $\frac{N_d(2,3)}{2\cdot 3^{d-1}}\in \mathbb{Q}\setminus \mathbb{Z}$ for all $d\geq 2$.

Now, notice that $\sec(3z) = \sum E_{2k}3^{2k}\frac{z^{2k}}{(2k)!}$, and since the $E_{2k}$ are integers, all coefficients of $\sec(3z)$ starting from degree $2$ are multiples of $3$, i.e., $\sec(3z) = 1 + 3C(z)$ where $C(z) = \sum_{d\geq 2} \frac{c_d}{d!} z^d$ and $c_d$ are integers. Therefore, 
\begin{equation*}
\sum N_d(2,3)\dfrac{z^d}{d!} = (\cos z + \sin 2z)\sec 3z = (\cos z + \sin 2z) + 3(\cos z + \sin 2z)C(z),
\end{equation*}
that is, for $d\geq 0$,
\begin{gather*}
N_d(2,3) \equiv \begin{cases}
(-1)^k \pmod{3} & d = 2k,\\
(-1)^k2^{2k+1} \pmod{3} & d = 2k + 1.
\end{cases}
\end{gather*}
In particular, $N_d(2,3) \not\equiv 0\pmod{3}$ for $d\geq 2$, and therefore $\limeHK(R_{0,d})\neq \limeHK(S_{0,d})$ for $d\geq 2$.
\end{proof}

\begin{proof}[Proof of \ref{th:watanabe_yoshida_low_dimension}]
Let $(A,R_A)$ be descent data for $(K,R)$. The fibers $R_{k(\m)}$ are $d$-dimensional at every closed point. By \cite[Théorème 12.2.4]{EGA4_3}, since $R=R_K$ is not regular, and $R_K$ is an extension of the generic fiber of $A$, also $R_{k(\m)}$ is not regular for all $\m\in \MaxSpec{A}$. Furthermore, again by \cite[(2.3.9)]{hochsterTightClosureEqual1999}, after localising $A$ at an element we can assume that if $R$ is unmixed, or a complete intersection, the fibers $R_{k(\m)}$ are also unmixed, or a complete intersection, respectively, for all closed points $\m\in \MaxSpec{A}$.

Now, let $k = A/\m$ of characteristic $p$ for some $\m \in \MaxSpec{A}$. Then, if $2\leq d\leq 7$, or $R$ is a complete intersection, then $\eHK(R_{k(\m)}) \geq \eHK(R_{p,d})$ by \cite{WY05,enescushimomotoDenseUpperSemicontinuity2005,aberbachenescu2013,coxsteibaberbach2024}, and therefore $\eHKp(R) \geq \eHK(R_{p,d})$, which implies that $\eHK^-(R) \geq \limeHK(R_{0,d})$. If $\eHK^-(R) = \limeHK(R_{0,d})$, due to \ref{lm:A_1_and_A_2_different_eHK}, it follows that $\eHK(R_k) = \eHK(R_{p,d})$ for some closed fiber $k$, which under the hypotheses above, implies that $R_k$ is an $A_1$ singularity by \cite[Theorem 5.1]{castilloreyStrong2025}. Now, the result follows from \ref{lm:reduction_A_1_singularity}.
\end{proof}

\bibliographystyle{alpha}
\bibliography{refs_1}

\newcommand{\etalchar}[1]{$^{#1}$}
\begin{thebibliography}{BCPT25}

\bibitem[AE08]{aberbachenescu2008}
Ian~M Aberbach and Florian Enescu.
\newblock Lower bounds for {{Hilbert-Kunz}} multiplicities in local rings of
  fixed dimension.
\newblock {\em Michigan Mathematical Journal}, 57:1--16, 2008.

\bibitem[AE13]{aberbachenescu2013}
Ian~M Aberbach and Florian Enescu.
\newblock New estimates of {{Hilbert}}--{{Kunz}} multiplicities for local rings
  of fixed dimension.
\newblock {\em Nagoya Mathematical Journal}, 212:59--85, 2013.

\bibitem[Arn72]{arnoldNormalFormsFunctions1972}
Vladimir~Igorevich Arnol'd.
\newblock Normal forms for functions near degenerate critical points, the
  {{Weyl}} groups of {{Ak}}, {{Dk}}, {{Ek}} and {{Lagrangian}} singularities.
\newblock {\em Functional Analysis and its applications}, 6(4):254--272, 1972.

\bibitem[BCPT25]{brosowskyLimitFsignatureFunctions2025}
Anna Brosowsky, Izzet Coskun, Suchitra Pande, and Kevin Tucker.
\newblock Limit {{F-signature}} functions of two-variable binomial
  hypersurfaces, 2025.

\bibitem[BE04]{blickleRingsSmallHK2004}
Manuel Blickle and Florian Enescu.
\newblock On {{Rings}} with {{Small Hilbert-Kunz Multiplicity}}.
\newblock {\em Proceedings of the American Mathematical Society},
  132(9):2505--2509, 2004.

\bibitem[BH98]{brunsherzog1998}
Winfried Bruns and H~J{\"u}rgen Herzog.
\newblock {\em Cohen-{{Macaulay Rings}}}.
\newblock Cambridge University Press, 1998.

\bibitem[Bri17]{brinkmann2017}
Daniel Brinkmann.
\newblock The {{Hilbert}}--{{Kunz}} functions of two-dimensional rings of type
  {{ADE}}.
\newblock {\em Journal of Algebra}, 469:358--389, 2017.

\bibitem[BST12]{blickleFsignaturePairsAsymptotic2012}
Manuel Blickle, Karl Schwede, and Kevin Tucker.
\newblock F-signature of pairs and the asymptotic behavior of {{Frobenius}}
  splittings.
\newblock {\em Advances in Mathematics}, 231(6):3232--3258, 2012.

\bibitem[BST13]{blickleFsignaturePairsContinuity2013}
Manuel Blickle, Karl Schwede, and Kevin Tucker.
\newblock F-signature of pairs: Continuity, p-fractals and minimal log
  discrepancies.
\newblock {\em Journal of the London Mathematical Society}, 87(3):802--818,
  2013.

\bibitem[CA24]{coxsteibaberbach2024}
Nicholas~O {Cox-Steib} and Ian~M Aberbach.
\newblock Bounds for the {{Hilbert-Kunz Multiplicity}} of {{Singular Rings}}.
\newblock {\em Acta Mathematica Vietnamica}, 49(1):39--60, March 2024.

\bibitem[{Cas}25]{castilloreyStrong2025}
Joel {Castillo-Rey}.
\newblock Strong {{Watanabe-Yoshida}} conjecture for {{Complete
  Intersections}}, 2025.

\bibitem[CST21]{carvajal2021bertini}
Javier {Carvajal-Rojas}, Karl Schwede, and Kevin Tucker.
\newblock Bertini theorems for {{F-signature}} and {{Hilbert}}--{{Kunz}}
  multiplicity: {{J}}. {{Carvajal-Rojas}} et al.
\newblock {\em Mathematische Zeitschrift}, 299(1):1131--1153, 2021.

\bibitem[CSTZ24]{shidelerFsignatureFunctionsDiagonal2024}
Alessio Caminata, Samuel Shideler, Kevin Tucker, and Francesco Zerman.
\newblock F-signature functions of diagonal hypersurfaces, 2024.

\bibitem[ER95]{ehrenborg1995}
Richard Ehrenborg and Margareta Readdy.
\newblock Sheffer posets and r-signed permutations.
\newblock {\em Ann. Sci. Math. Qu\'ebec}, 19(2):173--196, 1995.

\bibitem[ES05]{enescushimomotoDenseUpperSemicontinuity2005}
Florian Enescu and Kazuma Shimomoto.
\newblock On the upper semi-continuity of the {{Hilbert}}--{{Kunz}}
  multiplicity.
\newblock {\em Journal of Algebra}, 285(1):222--237, 2005.

\bibitem[GK90]{greuelSimpleSingularitiesPositive1990}
Gert-Martin Greuel and H~Kr{\"o}ning.
\newblock Simple {{Singularities}} in {{Positive Characteristic}}.
\newblock {\em Mathematische Zeitschrift}, 203(2):339--354, 1990.

\bibitem[GM10]{gesselLimit2010}
Ira~M Gessel and Paul Monsky.
\newblock The limit as $p\rightarrow \infty$ of the {{Hilbert-Kunz}}
  multiplicity of $\sum_{i=1}^d x_i^{d_i}$, 2010.

\bibitem[GP26]{greuelSplittingLemmaAny2026}
Gert-Martin Greuel and Gerhard Pfister.
\newblock The {{Splitting Lemma}} in {{Any Characteristic}}.
\newblock {\em Journal of Algebra}, 689:610--628, March 2026.

\bibitem[Gro66]{EGA4_3}
Alexander Grothendieck.
\newblock \'el\'ements de g\'eom\'etrie alg\'ebrique: {{IV}}. {{\'Etude}}
  locale des sch\'emas et des morphismes de sch\'emas, {{Troisi\`eme}} partie.
\newblock {\em Publications Math\'ematiques de l'IH\'ES}, 28:5--255, 1966.

\bibitem[HH99]{hochsterTightClosureEqual1999}
Melvin Hochster and Craig Huneke.
\newblock Tight closure in equal characteristic zero, 1999.

\bibitem[HL04]{hunekeTwoTheoremsMaximal2004}
Craig Huneke and Graham Leuschke.
\newblock Two theorems about maximal {{Cohen}}--{{Macaulay}} modules.
\newblock {\em Mathematische Annalen}, 324, September 2004.

\bibitem[HM93]{hanmonskySurprising1993}
Chungsim Han and Paul Monsky.
\newblock Some surprising {{Hilbert-Kunz}} functions.
\newblock {\em Mathematische Zeitschrift}, 214(1):119--135, 1993.

\bibitem[Hof99]{hoffmanDerivativePolynomialsEuler1999}
Michael~E Hoffman.
\newblock Derivative polynomials, {{Euler}} polynomials, and associated integer
  sequences.
\newblock {\em The Electronic Journal of Combinatorics}, pages R21--R21, 1999.

\bibitem[JNS{\etalchar{+}}23]{jeffriesmaxFsign2023}
Jack Jeffries, Yusuke Nakajima, Ilya Smirnov, Kei-Ichi Watanabe, and Ken-Ichi
  Yoshida.
\newblock Lower bounds on {{Hilbert}}--{{Kunz}} multiplicities and maximal
  {{F-signatures}}.
\newblock {\em Mathematical Proceedings of the Cambridge Philosophical
  Society}, 174(2):247--271, 2023.

\bibitem[Kun69]{kunz1969}
Ernst Kunz.
\newblock Characterizations of {{Regular Local Rings}} of {{Characteristic}} p.
\newblock {\em American Journal of Mathematics}, 91(3):772--784, 1969.

\bibitem[LP26]{liuPositivity2026}
Yuchen Liu and Suchitra Pande.
\newblock On positivity of the limit {{F-signature}}, 2026.

\bibitem[Mat89]{matsumura1989commutative}
Hideyuki Matsumura.
\newblock {\em Commutative Ring Theory}.
\newblock Number~8. Cambridge university press, 1989.

\bibitem[Men26a]{mengQuadrics2026}
Cheng Meng.
\newblock Hilbert-{{Kunz}} multiplicity of quadrics decreases, 2026.

\bibitem[Men26b]{mengLimits2026}
Cheng Meng.
\newblock Limits of {{F-invariants}} and {{Riemann-Stieltjes}} integral, 2026.

\bibitem[MM25]{mengHfunction2025}
Cheng Meng and Alapan Mukhopadhyay.
\newblock H-function, {{Hilbert-Kunz}} density function and
  {{Frobenius-Poincar\'e}} function, 2025.

\bibitem[Mon83]{monskyHilbertKunzFunction1983}
Paul Monsky.
\newblock The {{Hilbert-Kunz Function}}.
\newblock {\em Mathematische Annalen}, 263:43--50, 1983.

\bibitem[Mon09]{monsky2009transcendence}
Paul Monsky.
\newblock Transcendence of some {{Hilbert-Kunz}} multiplicities (modulo a
  conjecture).
\newblock {\em arXiv:0908.0971}, 2009.

\bibitem[MT06]{monskyPFractals2006}
Paul Monsky and Pedro Teixeira.
\newblock P-{{Fractals}} and power series---{{II}}.: {{Some}} applications to
  {{Hilbert}}--{{Kunz}} theory.
\newblock {\em Journal of algebra}, 304(1):237--255, 2006.

\bibitem[NS20]{nunezbetancourtsmirnov2020}
Luis {N{\'u}{\~n}ez-Betancourt} and Ilya Smirnov.
\newblock Hilbert--{{Kunz}} multiplicities and {{F-thresholds}}: {{L}}.
  {{N\'u\~nez-Betancourt}}, {{I}}. {{Smirnov}}.
\newblock {\em Bolet\'in de la Sociedad Matem\'atica Mexicana}, 26(1):15--25,
  2020.

\bibitem[Pol18]{polstra2018uniform}
Thomas Polstra.
\newblock Uniform bounds in {{F-finite}} rings and lower semi-continuity of the
  {{F-signature}}.
\newblock {\em Transactions of the American Mathematical Society},
  370(5):3147--3169, 2018.

\bibitem[PS20]{polstraContinuity2020}
Thomas Polstra and Ilya Smirnov.
\newblock Continuity of {{Hilbert}}--{{Kunz}} multiplicity and {{F-signature}}.
\newblock {\em Nagoya Mathematical Journal}, 239:322--345, 2020.

\bibitem[PSSY26]{pakQuadricsEhrhart2026}
Igor Pak, Boris Shapiro, Ilya Smirnov, and Ken-ichi Yoshida.
\newblock Hilbert--{{Kunz}} multiplicity of quadrics via {{Ehrhart}} theory,
  2026.

\bibitem[Sei97]{seibertFCMtype1997}
Gerhard Seibert.
\newblock The {{Hilbert-Kunz}} function of rings of finite {{Cohen-Macaulay}}
  type.
\newblock {\em Archiv der Mathematik}, 69(4):286--296, 1997.

\bibitem[Smi20]{smirnov2020semicontinuity}
Ilya Smirnov.
\newblock On semicontinuity of multiplicities in families.
\newblock {\em Documenta Mathematica}, 25:381--399, 2020.

\bibitem[{Sta}26]{stacks-project}
The {Stacks project authors}.
\newblock The {{Stacks}} project, 2026.

\bibitem[Tuc12]{tuckerFsignatureExists2012}
Kevin Tucker.
\newblock F-signature exists.
\newblock {\em Inventiones mathematicae}, 190(3):743--765, March 2012.

\bibitem[WY00]{WY00}
Kei-ichi Watanabe and Ken-ichi Yoshida.
\newblock Hilbert--{{Kunz Multiplicity}} and an {{Inequality}} between
  {{Multiplicity}} and {{Colength}}.
\newblock {\em Journal of Algebra}, 230:295--317, April 2000.

\bibitem[WY05]{WY05}
Kei-ichi Watanabe and Ken-ichi Yoshida.
\newblock Hilbert-{{Kunz Multiplicity}} of {{Three-Dimensional Local Rings}}.
\newblock {\em Nagoya Mathematical Journal}, 177:47--75, 2005.

\bibitem[Yos09]{Yo09}
Ken-ichi Yoshida.
\newblock Small {{Hilbert-Kunz}} multiplicity and $({A}_1)$-type singularity.
\newblock {\em Proceedings of the 4th Japan-Vietnam Joint Seminar on
  Commutative Algebra by and for Young Mathematicians, Meiji University,
  Japan}, 2009.

\end{thebibliography}
\end{document}